\documentclass[fontsize=10]{scrartcl}
\usepackage[utf8]{inputenc}
\usepackage{mathtools}
\usepackage{amsmath}
\usepackage{amsthm}
\usepackage{amssymb}
\usepackage[a4paper]{geometry}
\usepackage[pdfusetitle,
 bookmarks=true,bookmarksnumbered=false,bookmarksopen=false,
 breaklinks=false,pdfborder={0 0 1},backref=false,colorlinks=false]
 {hyperref}

\makeatletter
\usepackage{tikz}

\makeatother

\theoremstyle{remark}
\newtheorem{rem}{\protect\remarkname}
\theoremstyle{plain}
\newtheorem{thm}{\protect\theoremname}
\newtheorem{question}{\protect\questionname}
\theoremstyle{definition}
\newtheorem{defn}{\protect\definitionname}
\theoremstyle{plain}
\newtheorem{choice}{\protect\choicename}
\newtheorem{lem}{\protect\lemmaname}
\newtheorem{prop}{\protect\propositionname}
\newtheorem{cor}{\protect\corollaryname}
\theoremstyle{remark}
\newtheorem*{rem*}{\protect\remarkname}
\providecommand{\choicename}{Choice}
\providecommand{\corollaryname}{Corollary}
\providecommand{\definitionname}{Definition}
\providecommand{\lemmaname}{Lemma}
\providecommand{\propositionname}{Proposition}
\providecommand{\questionname}{Question}
\providecommand{\remarkname}{Remark}
\providecommand{\theoremname}{Theorem}

\begin{document}
\title{Vorticity blow-up for the 2D incompressible non-homogeneous Euler
equations with uniform $C^{1,\sqrt{\frac{4}{3}}-1-\varepsilon}$ force}
\author{Diego Córdoba\thanks{Instituto de Ciencias Matemáticas CSIC-UAM-UCM-UC3M, Spain. E-mail:
dcg@icmat.es}, Andrés Laín-Sanclemente\thanks{Instituto de Ciencias Matemáticas CSIC-UAM-UCM-UC3M, Spain. E-mail:
andres.lain@icmat.es} and Luis Martínez-Zoroa\thanks{CUNEF Universidad, Spain. E-mail: luis.martinezzoroa@cunef.edu}}
\maketitle
\begin{abstract}
We establish the existence of solutions of the 2D incompressible non-homogeneous
Euler equations with $C^{0}_{t}C^{1,\,\sqrt{\frac{4}{3}}-1-\varepsilon}_{x}\cap C^{0}_{t}L^{2}_{x}$
source terms that develop a singularity in finite time. In order to
achieve this, we adapt the Boussinesq blow-up we set up in \cite{Articulo Boussinesq}
to the non-homogeneous Euler setting. Furthermore, we bring the potential
existence of two different types of singularities of the forced system
to light.
\end{abstract}
\tableofcontents{}

\section{Introduction}

\subsection{Physical model}

We will consider the 2D non-homogeneous Euler equations
\begin{equation}
\begin{aligned}\frac{\partial\rho}{\partial t}+u\cdot\nabla\rho & =f_{\rho}\quad\left(\text{mass equation}\right),\\
\frac{\partial u}{\partial t}+u\cdot\nabla u & =-\frac{\nabla P}{\rho}+f_{u}\quad\left(\text{momentum equation}\right),\\
\nabla\cdot u & =0\quad\left(\text{incompressibility constraint}\right),
\end{aligned}
\label{eq:inhomogeneous 2D Euler}
\end{equation}
in the domain $\left[0,T\right]\times\mathbb{R}^{2}$. In the above,
$\rho:\left[0,T\right]\times\mathbb{R}^{2}\to\mathbb{R}^{+}$ is the
density of the fluid, $u:\left[0,T\right]\times\mathbb{R}^{2}\to\mathbb{R}^{2}$
denotes its velocity field, $P:\left[0,T\right]\times\mathbb{R}^{2}\to\mathbb{R}$
represents its pressure, $f_{\rho}:\left[0,T\right]\times\mathbb{R}^{2}\to\mathbb{R}$
is a mass source term and $f_{u}:\left[0,T\right]\times\mathbb{R}^{2}\to\mathbb{R}^{2}$
depicts an external acceleration. Notice that, although we will still
denote this acceleration by $f_{u}$ (following the notation of, for
instance, \cite{Danchin-Fanelli}), this $f_{u}$ is really an acceleration
and not a force. Bear in mind that, since $\rho$ denotes a physical
density, it must be a non-negative function. In fact, in the present
paper we will work in the absence of vacuum, i.e., with $\rho$ strictly
positive. The first equation encodes the conservation of mass, the
second can be derived from Newton's second law (and, thus, represents
the conservation of momentum) and the third embodies the incompressibility
of the fluid.

\subsection{\label{subsec:Boussinesq is first order model of non-homogeneous Euler}Boussinesq
as a formal first-order model of non-homogeneous Euler}

As we will see in section \ref{sec:ideas of the proof}, the finite-time
singularity for the Boussinesq equation that we constructed in \cite{Articulo Boussinesq}
will play a crucial role for the present result. Hence, we wish to
briefly discuss the relationship that exists between the non-homogeneous
2D Euler equations and the 2D Boussinesq equations, as it will help
us understand why a solution of Boussinesq can be ``transformed''
into a solution of non-homogeneous Euler. If one departs from system
\eqref{eq:inhomogeneous 2D Euler} and expands $\rho$, $P$, $u$,
$f_{\rho}$, $f_{u}$ in a formal power series of a small parameter
$\varepsilon$ as follows
\[
\begin{aligned}\rho & =\rho^{\left(0\right)}+\rho^{\left(1\right)}\varepsilon+\dots, & u & =u^{\left(0\right)}+u^{\left(1\right)}\varepsilon+\dots,\\
f_{\rho} & =f^{\left(0\right)}_{\rho}+f^{\left(1\right)}_{\rho}\varepsilon+\dots, & P & =P^{\left(-1\right)}\frac{1}{\varepsilon}+P^{\left(0\right)}+P^{\left(1\right)}\varepsilon+\dots,\\
f_{u} & =f^{\left(-1\right)}_{u}\frac{1}{\varepsilon}+f^{\left(0\right)}_{u}+f^{\left(1\right)}_{u}\varepsilon+\dots,
\end{aligned}
\]
along with the additional assumptions $\rho^{\left(0\right)}\in\mathbb{R}_{+}$
constant in space and time, $f^{\left(-1\right)}_{u}=\nabla G$ for
some $G$, $f^{\left(0\right)}_{u}=0$ and $f^{\left(0\right)}_{\rho}=0$,
one can easily deduce that $\left(\rho_{B},u_{B},P_{B},f_{\rho_{B}},f_{u_{B}}\right)=\left(\frac{\rho^{\left(1\right)}}{\rho^{\left(0\right)}},u^{\left(0\right)},\frac{P^{\left(0\right)}}{\rho^{\left(0\right)}},\frac{f^{\left(1\right)}_{\rho}}{\rho^{\left(0\right)}},f^{\left(-1\right)}_{u}\right)$
is a solution of the system
\begin{equation}
\begin{aligned}\frac{\partial\rho_{B}}{\partial t}+u_{B}\cdot\nabla\rho_{B} & =f_{\rho_{B}},\\
\frac{\partial u_{B}}{\partial t}+u_{B}\cdot\nabla u_{B} & =-\nabla P_{B}+\rho_{B}f_{u_{B}},\\
\nabla\cdot u_{B} & =0.
\end{aligned}
\label{eq:Boussinesq-like}
\end{equation}
This corresponds to the regime of a very strong conservative force
that only depends on variations of density, but not on the value of
the density itself (like the buoyancy force). Finally, notice that
\eqref{eq:Boussinesq-like} coincides with the actual Boussinesq equations
when $f_{u_{B}}=f^{\left(-1\right)}_{u}=-\hat{e}_{1}$.

\subsection{Context}

The issue of local existence of solutions for the non-homogeneous
incompressible Euler equations (see \eqref{eq:inhomogeneous 2D Euler})
has been studied since the 1970s. As far as the authors are aware,
the first result concerning local well-posedness is due to Marsden
(1976) \cite{Marsden}, who studies the equations in the absence of
external source terms on smooth compact manifolds of any dimension
while measuring the solution in Sobolev and Hölder spaces. In 1980,
Beirão da Veiga and Valli incorporate an external acceleration and
show local well-posedness in Hölder spaces in bounded connected open
subsets of $\mathbb{R}^{2}$ and $\mathbb{R}^{3}$ (see \cite{Beirao-Valli1,Beirao-Valli2}).
Furthermore, they are able to prove local-in-time existence for smooth
initial data in bounded connected open subsets of $\mathbb{R}^{3}$
(see \cite{Beirao-Valli3}). In 1985, Valli and Zajaczkowski give
in \cite{Valli-Zajaczkowski} an alternative technique that proves
local well-posedness in $W^{2,p}$-spaces in bounded domains of $\mathbb{R}^{d}$
($d\ge2$). In 1993, Itoh is the first one to show local-in-time well-posedness
in Sobolev spaces when the domain is the whole space $\mathbb{R}^{3}$
(see \cite{Itoh 1}). Some years later, in 2010, Danchin \cite{Danchin}
shows local-in-time well-posedness in Besov spaces $B^{s}_{p,r}\left(\mathbb{R}^{d}\right)$
($d\geq2$) with $1<p<\infty$ and gives the first blow-up criterion.
In addition, Zhou, Xin and Fan present similar results for the periodic
case in \cite{Zhou-Xin-Fan} (2010). Just a year later, Danchin and
Fanelli extend Danchin's result to the case $p=\infty$ in \cite{Danchin-Fanelli},
which includes Hölder spaces. Importantly, these results are proven
in the presence of an external acceleration. Lastly, further blow-up
criteria were developed by Bae, Lee and Shin in \cite{Bae-Lee-Shin}
(2020) and Fanelli in \cite{Fanelli Blow-up} (2025).

In contrast with the homogeneous 2D Euler case, where global-in-time
existence of smooth solutions has been known since the work of Wolibner
(1933) \cite{Wolibner}, no analogous result for the non-homogeneous
2D case is available. Moreover, the present paper proves singularities
do form, so such a result would in fact be impossible. As far as the
authors are aware, this is the first article where singularity formation
for the incompressible non-homogeneous Euler equations is established.

For the sake of completeness and to place our current work in the
zoo of singularities in fluid dynamics, we will briefly mention the
known singularities of the 3D incompressible homogeneous Euler and
2D Boussinesq equations, restricting ourselves to the ones achieved
in the absence of boundaries, i.e., where the spatial domain is $\mathbb{R}^{2}$
or $\mathbb{R}^{3}$. We can distinguish two important families: self-similar
blow-ups and non-self-similar ones.
\begin{itemize}
\item On the one hand, in the first family, we have Elgindi's result \cite{Elgindi C1alpha I}
(2021) about finite-time singularity formation for the forced homogeneous
axisymmetric 3D Euler equations without swirl departing from $C^{1,\alpha}\left(\mathbb{R}^{3}\right)$
initial datum with very small $\alpha$ ---later adapted for the
unforced equations by Elgindi, Ghoul and Masmoudi (see \cite{Elgindi C1alpha II})---
and Elgindi's and Pasqualotto's construction \cite{Elgindi Boussinesq I,Elgindi Boussinesq II}
(2023) for the unforced 2D Boussinesq and unforced homogeneous axisymmetric
3D Euler equations with swirl, also departing from $C^{1,\alpha}$
initial data with very small $\alpha$. 
\item On the other hand, in the family of non-self-similar finite-time singularities,
we have Córdoba's, Martínez-Zoroa's and Zheng's construction \cite{Euler no autosimilar}
(2025) for the unforced homogeneous axisymmetric 3D Euler equations
without swirl starting from $C^{1,\alpha}\left(\mathbb{R}^{3}\right)$
initial datum with very small $\alpha$, Córdoba's and Martínez-Zoroa's
``rotating'' singularity \cite{Fuerza Euler} (2023) for the forced
homogeneous 3D Euler equations with uniform $C^{1,\frac{1}{2}-}\cap L^{2}$
force --later adapted by the same authors and Zheng to prove finite-time
blow-up for the hypodissipative Navier-Stokes equations with a force
in $L^{1}_{t}C^{1,\varepsilon}_{x}\cap L^{\infty}_{t}L^{2}_{x}$ in
\cite{Fuerza Navier Stokes} (2024)--, Córdoba's and Martínez-Zoroa's
finite-time singularity for the incompressible porous media equation
(IPM) with a smooth source \cite{Fuerza IPM} (2024) and the result
\cite{Articulo Boussinesq} (2025) from the same authors of this paper
about singularity formation for the 2D Boussinesq equation with $C^{0}_{t}C^{1,\,\sqrt{\frac{4}{3}}-1-\varepsilon}_{x}\cap C^{0}_{t}L^{2}_{x}$
source terms. This article builds heavily on the work done in \cite{Articulo Boussinesq}. 
\item Moreover, very recently, the global existence of solutions for the
incompressible homogeneous axisymmetric 3D Euler equations without
swirl for $C^{1,\frac{1}{3}+}$ initial data has been proved (almost)
sharp by counterexamples provided by three different methods: a Lagrangian
clock-and-strain framework by Shkoller \cite{Shkoller}, a computer-assisted
proof of an asymptotically self-similar blow-up by Chen \cite{Chen1,Chen2}
and an exact self-similar blow-up solution (necessarily with infinite
energy, but with finite-codimension-stability) by Shao, Wei, Zhang
and Zhang \cite{ShaoWeiZhangZhang}.
\end{itemize}
We conclude this subsection by mentioning further works that concern
themselves with other properties of the non-homogeneous Euler equations.
Lower bounds for the time of existence of solutions with small initial
data (in a certain sense) are provided by Danchin in \cite{Danchin}
(2010), by Danchin and Fanelli in \cite{Danchin-Fanelli} (2011) and
by Bae, Lopes Filho, Mazzucato and Nussenzveig Lopes in \cite{Bae-Lopes-Mazzucato-Nussenzveig}
(2025). Moreover, the propagation of striated and conormal regularity
of initial data (a way of generalizing 2D vortex patches) is studied
by Fanelli in \cite{Fanelli Geometric Structures}. Besides, asymptotic
stability of shear flows close to the Couette flow on $\mathbb{T}\times\mathbb{R}$
in Gevrey spaces is proved by Chen, Wei, Zhang and Zhang in \cite{Chen-Wei-Zhang-Zhang}
(2025). Furthermore, the existence and nonexistence in Sobolev spaces
of traveling waves near monotonic shear flows in a periodic channel
are studied by Zhao and Zhao in \cite{Zhao-Zhao} (2026). On the other
hand, local-in-time existence in Sobolev spaces over bounded and unbounded
domains of $\mathbb{R}^{3}$ of the incompressible non-homogeneous
Navier-Stokes equations (the incompressible non-homogeneous Euler
equations with viscosity) was shown by Itoh and Tani in \cite{Itoh-Tani}
(1999), where the vanishing viscosity limit is also investigated.
This vanishing viscosity limit has also been researched in the context
of weak solutions; for instance, see the recent paper \cite{Schr=0000F6der-Wiedemann}
(2025) of Schröder and Wiedemann. Continuing in the weak context,
non-uniqueness of $C^{\alpha}$ solutions with $\alpha<\frac{1}{7}$
was established by Giri and Koley in \cite{Giri-Koley} (2025). In
addition, the distribution describing the local energy flux of merely
bounded solutions was studied by Inversi and Violini in \cite{Inversi-Violini}
(2024). Leaving the topic of weak solutions, the existence of infinitely
many smooth steady solutions is shown by Fanelli and Feireisl in \cite{Fanelli-Feireisl}
(2020). Moreover, a derived model (with applications for the study
of the mean dynamics in the ocean) and the global existence of its
solutions under a smallness condition was researched by Bravin and
Fanelli in \cite{Bravin-Fanelli} (2025). Lastly, a Lagrangian formulation
of the non-homogeneous Euler equations has recently been introduced
in \cite{PanLagrangian} by Pan (2025).

\subsection{\label{subsec:Canonical-force-decomposition}Canonical force decomposition
and blow-up classes}

As occurs in the homogeneous Euler equations, the incompressibility
constraint allows for the existence of multiple forces that trigger
the same dynamics. Indeed, given a solution $\left(\rho,u,P,f_{\rho},f_{u}\right)$
of \eqref{eq:inhomogeneous 2D Euler}, $\left(\rho,u,P+\phi,f_{\rho},f_{u}+\frac{\nabla\phi}{\rho}\right)$
is another solution of \eqref{eq:inhomogeneous 2D Euler} with the
same density and velocity (this is what we mean when we say ``the
same dynamics''). Furthermore, if no singularities occur, we can
guarantee that, if $f_{u}$ is in the well-posedness regime, so will
$f_{u}+\frac{\nabla\phi}{\rho}$. However, and contrary to the 3D
homogeneous Euler equations or the 2D Boussinesq system, this universal
regularity breaks down in the case of formation of singularities.

\subsubsection{Canonical decomposition in the homogeneous case}

To justify the claims above, we have to introduce the canonical decomposition
of $f_{u}$ in \eqref{eq:inhomogeneous 2D Euler} and, to undertake
this last task, it will prove helpful to review the decomposition
in the homogeneous case first. To this end, consider the homogeneous
Euler equations
\begin{equation}
\frac{\partial u}{\partial t}+u\cdot\nabla u=-\nabla P+f,\label{eq:homogeneous Euler}
\end{equation}
in a bounded domain $\Omega\subseteq\mathbb{R}^{2}$ (we will study
the natural extension to unbounded domains later). On the one hand,
taking the divergence at both sides of \eqref{eq:homogeneous Euler},
we arrive at the equation of the pressure
\[
\nabla\cdot\left(u\cdot\nabla u\right)=-\Delta P+\nabla\cdot f.
\]
On the other hand, taking the curl at both sides of \eqref{eq:homogeneous Euler},
we deduce the vorticity equation
\[
\frac{\partial\omega}{\partial t}+u\cdot\nabla\omega=\nabla^{\perp}\cdot f,
\]
where $\omega=\nabla^{\perp}\cdot u$. In view of these equations,
if we wish to write $f$ as the sum of two functions, one that only
appears in the pressure equation and another that only appears in
the vorticity equation, we quickly arrive at the Helmholtz decomposition
\begin{equation}
f=\nabla\phi+\nabla^{\perp}\psi,\label{eq:canonical decomposition homogeneous}
\end{equation}
which leads to
\[
\nabla\cdot\left(u\cdot\nabla u\right)=-\Delta P+\Delta\phi,\quad\frac{\partial\omega}{\partial t}+u\cdot\nabla\omega=\Delta\psi.
\]
To compute $\phi$ and $\psi$ from $f$, we take the divergence and
the curl at both sides of \eqref{eq:canonical decomposition homogeneous},
obtaining
\[
\nabla\cdot f=\Delta\phi,\quad\nabla^{\perp}\cdot f=\Delta\psi.
\]
Nonetheless, a key question remains: what are the boundary conditions
for $\phi$ and $\psi$? Multiplying \eqref{eq:homogeneous Euler}
by the normal vector $\hat{n}$ at the boundary, we see that some
Neumann boundary conditions for the pressure follow. If we wish for
$\nabla^{\perp}\psi$ to have no influence over $P$, we should impose
$\left.\nabla^{\perp}\psi\cdot\hat{n}\right|_{\partial\Omega}=0$.
One can also arrive at this conclusion by enforcing $\left\langle \nabla\phi,\nabla^{\perp}\psi\right\rangle _{L^{2}\left(\Omega\right)}=0$.
Indeed, by integration by parts, we have
\[
\left\langle \nabla\phi,\nabla^{\perp}\psi\right\rangle _{L^{2}\left(\Omega\right)}=\int_{\partial\Omega}\phi\nabla^{\perp}\psi\cdot\hat{n}\mathrm{d}\Gamma.
\]
Since the value of $\phi$ has no physical meaning (only its gradient
has), the only natural condition is to require $\left.\nabla^{\perp}\psi\cdot\hat{n}\right|_{\partial\Omega}=0$.
Thereby, $\phi$ and $\psi$ solve the following problems
\[
\left\{ \begin{aligned}\Delta\phi & =\nabla\cdot f & \text{in }\Omega,\\
\nabla\phi\cdot\hat{n} & =f\cdot\hat{n} & \text{on }\partial\Omega,
\end{aligned}
\right.\quad\left\{ \begin{aligned}\Delta\psi & =\nabla^{\perp}\cdot f & \text{in }\Omega,\\
\nabla^{\perp}\psi\cdot\hat{n} & =0 & \text{on }\partial\Omega.
\end{aligned}
\right.
\]
The natural extension of these problems to $\mathbb{R}^{2}$ is
\[
\left\{ \begin{aligned}\Delta\phi & =\nabla\cdot f & \text{in }\mathbb{R}^{2},\\
\nabla\phi\cdot\hat{x} & =f\cdot\hat{x} & \text{at }\infty,
\end{aligned}
\right.\quad\left\{ \begin{aligned}\Delta\psi & =\nabla^{\perp}\cdot f & \text{in }\mathbb{R}^{2},\\
\nabla\psi & \in L^{2}\left(\mathbb{R}^{2}\right).
\end{aligned}
\right.
\]

\subsubsection{Canonical decomposition in the non-homogeneous case}

Now, let us follow the same procedure in the case of the non-homogeneous
Euler equations. Recall (see equation \eqref{eq:inhomogeneous 2D Euler})
that
\begin{equation}
\frac{\partial u}{\partial t}+u\cdot\nabla u=-\frac{\nabla P}{\rho}+f_{u}.\label{eq:inhomogeneous Euler momentum}
\end{equation}
Again, we will consider the equations in a bounded domain $\Omega\subseteq\mathbb{R}^{2}$.
On the one hand, taking the divergence at both sides of \eqref{eq:inhomogeneous Euler momentum},
we arrive at the equation of the pressure
\[
\nabla\cdot\left(u\cdot\nabla u\right)=-\nabla\cdot\left(\frac{\nabla P}{\rho}\right)+\nabla\cdot f_{u}.
\]
On the other hand, taking the curl at both sides of \eqref{eq:inhomogeneous Euler momentum},
we deduce the vorticity equation
\[
\frac{\partial\omega}{\partial t}+u\cdot\nabla\omega=\frac{\nabla P}{\rho}\cdot\frac{\nabla^{\perp}\rho}{\rho}+\nabla^{\perp}\cdot f_{u},
\]
where $\omega=\nabla^{\perp}\cdot u$. This time, the pressure does
not disappear but we may still use the momentum equation \eqref{eq:inhomogeneous Euler momentum}
to eliminate $\frac{\nabla P}{\rho}$, which leads to
\begin{equation}
\frac{\partial\omega}{\partial t}+u\cdot\nabla\omega=\left[f_{u}-\frac{\partial u}{\partial t}-u\cdot\nabla u\right]\cdot\frac{\nabla^{\perp}\rho}{\rho}+\nabla^{\perp}\cdot f_{u}.\label{eq:vorticity equation}
\end{equation}
Again, if we wish to write $f_{u}$ as the sum of two functions, one
that only appears in the pressure equation and another that only appears
in the vorticity equation, we are obliged to choose
\begin{equation}
f_{u}=\frac{\nabla\phi}{\rho}+\nabla^{\perp}\psi.\label{eq:inhomogeneous canonical decomposition fu}
\end{equation}
Indeed, the first summand does not appear in the vorticity equation
because
\[
\frac{\nabla\phi}{\rho}\cdot\frac{\nabla^{\perp}\rho}{\rho}+\nabla^{\perp}\cdot\left(\frac{\nabla\phi}{\rho}\right)=\frac{\nabla\phi}{\rho}\cdot\frac{\nabla^{\perp}\rho}{\rho}-\frac{1}{\rho^{2}}\nabla\phi\cdot\nabla^{\perp}\rho=0
\]
and the second summand does not appear in the pressure equation. How
do we compute $\phi$ and $\psi$ from the knowledge of $f_{u}$?
Taking the divergence and the curl at both sides of \eqref{eq:inhomogeneous canonical decomposition fu},
we obtain
\[
\nabla\cdot f_{u}=\nabla\cdot\left(\frac{\nabla\phi}{\rho}\right),\quad\nabla^{\perp}\cdot f_{u}=-\frac{\nabla\phi}{\rho}\cdot\frac{\nabla^{\perp}\rho}{\rho}+\Delta\psi.
\]
Multiplying \eqref{eq:inhomogeneous Euler momentum} by the normal
vector $\hat{n}$ at the boundary and following the same argument
as in the homogeneous case, we deduce that $\left.\nabla^{\perp}\psi\cdot\hat{n}\right|_{\partial\Omega}=0$.
In this way, $\phi$ and $\psi$ solve the following problems
\[
\left\{ \begin{aligned}\nabla\cdot\left(\frac{\nabla\phi}{\rho}\right) & =\nabla\cdot f_{u} & \text{in }\Omega,\\
\nabla\phi\cdot\hat{n} & =\rho f_{u}\cdot\hat{n} & \text{on }\partial\Omega,
\end{aligned}
\right.\quad\left\{ \begin{aligned}\Delta\psi & =\nabla^{\perp}\cdot f_{u}+\frac{\nabla\phi}{\rho}\cdot\frac{\nabla^{\perp}\rho}{\rho} & \text{in }\Omega,\\
\nabla^{\perp}\psi\cdot\hat{n} & =0 & \text{on }\partial\Omega.
\end{aligned}
\right.
\]
The natural extension of these problems to $\mathbb{R}^{2}$ is
\begin{equation}
\left\{ \begin{aligned}\nabla\cdot\left(\frac{\nabla\phi}{\rho}\right) & =\nabla\cdot f_{u} & \text{in }\mathbb{R}^{2},\\
\nabla\phi\cdot\hat{x} & =\rho f_{u}\cdot\hat{x} & \text{at }\infty,
\end{aligned}
\right.\quad\left\{ \begin{aligned}\Delta\psi & =\nabla^{\perp}\cdot f_{u}+\frac{\nabla\phi}{\rho}\cdot\frac{\nabla^{\perp}\rho}{\rho} & \text{in }\mathbb{R}^{2},\\
\nabla\psi & \in L^{2}\left(\mathbb{R}^{2}\right).
\end{aligned}
\right.\label{eq:decomposition force inhomogeneous}
\end{equation}

\begin{rem}
In view of equation \eqref{eq:decomposition force inhomogeneous},
if we were to consider the non-homogeneous Euler equations with gravity
(which would imply $f_{u}\notin L^{2}_{x}$), the term that absorbs
this unboundedness in $L^{2}_{x}$ is $\phi$ and never $\psi$; in
other words, the force that affects the dynamics always belongs to
$L^{2}_{x}$.
\end{rem}

\subsubsection{The regularity question}

Now that we have studied the canonical decomposition, we can proceed
to the central question of this section.
\begin{question}
\label{que:central question nonhomogeneous}If $f_{u}=\frac{\nabla\phi}{\rho}+\nabla^{\perp}\psi$
is in the well-posedness regime, must $\nabla^{\perp}\psi$ also be
in the well-posedness regime?
\end{question}
Before actually answering question \ref{que:central question nonhomogeneous},
we will consider its analog in the homogeneous case.
\begin{question}
\label{que:central question homogeneous}If $f_{u}=\nabla\phi+\nabla^{\perp}\psi$
is in the well-posedness regime, must $\nabla^{\perp}\psi$ also be
in the well-posedness regime?
\end{question}
The orthogonality of $\nabla\phi$ and $\nabla^{\perp}\psi$ in the
homogeneous case ensures that the answer to question \ref{que:central question homogeneous}
is affirmative. Actually, $f_{u}$ and $\nabla^{\perp}\psi$ must
have the same regularity. Indeed, suppose that $\nabla^{\perp}\psi$
were less regular than $f_{u}$. Then, the ``singular'' part of
$\nabla^{\perp}\psi$ would have to cancel some part of $\nabla\phi$,
i.e., there should be a common ``singular'' part to $\nabla\phi$
and $\nabla^{\perp}\psi$, but this is impossible because $\left\langle \nabla\phi,\nabla^{\perp}\psi\right\rangle _{L^{2}\left(\mathbb{R}^{2}\right)}=0$.
However, in the non-homogeneous case we no longer have $\left\langle \frac{\nabla\phi}{\rho},\nabla^{\perp}\psi\right\rangle _{L^{2}\left(\mathbb{R}^{2}\right)}=0$.
Actually, the answer to question \ref{que:central question nonhomogeneous}
will depend on whether $\rho$ develops a singularity or not.
\begin{itemize}
\item If $\rho$ does not become singular in finite time, Theorem 1 of \cite{Danchin-Fanelli}
ensures that $\rho$, $u$ and $\nabla\left(P+\phi\right)$ remain
as regular as $f_{u}$. Moreover, as $P$ only depends on $\rho$
and $u$, $\nabla P$ must be as regular as $f_{u}$, too. Consequently,
$\nabla\phi=\nabla\left(P+\phi\right)-\nabla P$ is as regular as
$f_{u}$ and, since $\rho$ is bounded away from zero, so is $\frac{\nabla\phi}{\rho}$.
Hence, $\frac{\nabla\phi}{\rho}$ is as regular as $f_{u}$ and, consequently,
so is $\nabla^{\perp}\psi$.
\item If $\rho$ does become singular in finite time, maybe we can absorb
the singular part of $\nabla^{\perp}\psi$ at the blow-up time with
$\frac{\nabla\phi}{\rho}$. Indeed, suppose that we can write $\nabla^{\perp}\psi=\frac{h}{\rho}$.
In view of \eqref{eq:decomposition force inhomogeneous}, such a $\nabla^{\perp}\psi$
is definitely possible under some conditions on $h$ ($\nabla\cdot\left(\frac{h}{\rho}\right)=0$).
Assume, for simplicity, that the singularity happens at the origin
at $t=1$. By taking $\phi\left(t,x\right)=-\left(h\left(t,0\right)\cdot x\right)\varphi\left(x\right)$,
where $\varphi\in C^{\infty}_{c}\left(\mathbb{R}^{2}\right)$ with
$\left.\varphi\right|_{\overline{B\left(0;1\right)}}\equiv1$, near
$x=0$ we have
\[
f_{u}\left(t,x\right)=\frac{\nabla\phi\left(t,x\right)}{\rho\left(t,x\right)}+\nabla^{\perp}\psi\left(t,x\right)=\frac{h\left(t,x\right)-h\left(t,0\right)}{\rho\left(t,x\right)}.
\]
Then, taking the gradient at both sides, we arrive at
\[
\nabla f_{u}\left(t,x\right)=\frac{\nabla h\left(t,x\right)\rho\left(t,x\right)-\left(h\left(t,x\right)-h\left(t,0\right)\right)\otimes\nabla\rho\left(t,x\right)}{\rho\left(t,x\right)^{2}}.
\]
Thus, by using this strategy, we can prevent the singularity of $\rho$
at the origin from altering the regularity of $f_{u}$, making $f_{u}$
more regular than $\nabla^{\perp}\psi$. 
\end{itemize}

\subsubsection{Two types of singularities}

In view of the preceding subsection, it makes sense to classify the
possible singularities of the forced non-homogeneous Euler equations
in two types.
\begin{defn}
\label{def:singularity classes nonhomogeneous}Assume that $\left(\rho,u,f_{\rho},f_{u}\right)$
is a singular solution of \eqref{eq:inhomogeneous 2D Euler}, where
$f_{\rho}$ and $f_{u}$ remain in the well-posedness regime. Let
$f_{u}=\frac{\nabla\phi}{\rho}+\nabla^{\perp}\psi$ be the canonical
decomposition of $f_{u}$ given in equation \eqref{eq:inhomogeneous canonical decomposition fu}.
\begin{itemize}
\item We say $\left(\rho,u,f_{\rho},f_{u}\right)$ is an \textbf{unscreened}
singularity if $f_{u}$ and $\nabla^{\perp}\psi$ have the same regularity.
\item We say $\left(\rho,u,f_{\rho},f_{u}\right)$ is a \textbf{screened}
singularity otherwise.
\end{itemize}
\end{defn}
\begin{rem}
In a screened singularity, the part of the force that drives the dynamics
becomes singular at the blow-up time, but the pressure is able to
absorb (screen) this singular part in such a way that it remains invisible
in the total force; whereas in an unscreened singularity, the part
of the force that drives the dynamics does not become singular at
the blow-up time.
\end{rem}
\begin{rem}
\label{rem:first-class singularity}If $\left(\rho,u,f_{\rho},f_{u}\right)$
is an unscreened singularity, so is $\left(\rho,u,f_{\rho},\nabla^{\perp}\psi\right)$.
Consequently, the search for an unscreened singularity can be reduced
to the case $\phi=0$.
\end{rem}

\subsection{Main result}
\begin{thm}
\label{thm:MAIN THM}Let $\alpha\in\left(0,\alpha_{*}\right)$, where
$\alpha_{*}=\sqrt{\frac{4}{3}}-1$. Let $\overline{\rho}\in\mathbb{R}_{+}$
be sufficiently large. Then, there is a solution $\left(\rho,u,f_{\rho},f_{u}\right)$
in $\left[0,1\right)\times\mathbb{R}^{2}$ of the forced incompressible
non-homogeneous Euler equations \eqref{eq:inhomogeneous 2D Euler}
that satisfies:
\begin{enumerate}
\item There is a finite time singularity at time $t=1$, i.e.,
\[
\lim_{T\to1^{-}}\int^{T}_{0}\left|\left|\omega\left(s,\cdot\right)\right|\right|_{L^{\infty}\left(\mathbb{R}^{2}\right)}\mathrm{d}s=\infty.
\]
(See Theorem 2 of \cite{Danchin-Fanelli}).
\item ~
\[
f_{\rho}\in C^{0}_{t}C^{1,\alpha}_{x,c}\left(\left[0,1\right]\times\mathbb{R}^{2}\right),\quad f_{u}\in C^{0}_{t}C^{1,\alpha}_{x}\left(\left[0,1\right]\times\mathbb{R}^{2};\mathbb{R}^{2}\right)\cap C^{0}_{t}L^{2}_{x}\left(\left[0,1\right]\times\mathbb{R}^{2};\mathbb{R}^{2}\right).
\]
\item The solution has the following structure:
\[
\rho=\overline{\rho}+\rho_{B},\quad u=u_{B},
\]
where $\left(\rho_{B},u_{B}\right)$ is one of the solutions of the
forced Boussinesq system constructed in \cite{Articulo Boussinesq},
which also develops a singularity at time $t=1$\@.
\item $\forall\varepsilon>0$, $\rho,u\in C^{\infty}_{t}C^{\infty}_{x}\left(\left[0,1-\varepsilon\right]\times\mathbb{R}^{2}\right)$
and $\nabla\rho$ and $u$ are compactly supported uniformly in time.
\item $\left(\rho,u,f_{\rho},f_{u}\right)$ is a screened singularity (see
definition \ref{def:singularity classes nonhomogeneous}).
\end{enumerate}
\end{thm}
\begin{rem}
We are in the local existence regime. In Theorem 1 of \cite{Danchin-Fanelli},
it is proven that, provided that $\rho\left(t,\cdot\right)\in C^{1,\alpha}\left(\mathbb{R}^{2}\right)$
is bounded away from zero, $u\left(t,\cdot\right)\in C^{1,\alpha}\left(\mathbb{R}^{2};\mathbb{R}^{2}\right)\cap L^{2}\left(\mathbb{R}^{2};\mathbb{R}^{2}\right)$,
$f_{u}\in L^{1}_{t}C^{1,\alpha}_{x}\left(\left[0,1\right]\times\mathbb{R}^{2};\mathbb{R}^{2}\right)\cap C^{0}_{t}L^{2}_{x}\left(\left[0,1\right]\times\mathbb{R}^{2};\mathbb{R}^{2}\right)$
and $f_{\rho}=0$, then we are in the well-posedness regime. In our
case, $\rho$ and $u$ satisfy the required conditions because of
point 4 of Theorem \ref{thm:MAIN THM} and because $\overline{\rho}$
is taken sufficiently large. $f_{u}$ also clearly satisfies the conditions
because of point 2 of Theorem \ref{thm:MAIN THM}. The only difference
is that, in our case, $f_{\rho}\neq0$. Nevertheless, one should be
able to adapt the same techniques used in \cite{Danchin,Danchin-Fanelli}
to this situation.
\end{rem}
\begin{rem}
The blow-up rates and loss of regularity at the singularity are exactly
the same as in \cite{Articulo Boussinesq}.
\end{rem}
\begin{rem}
Spatial and temporal regularity of the source terms can be traded
through interpolation.
\end{rem}
\begin{rem}
Although we do not prove it, the source terms $f_{\rho}$ and $f_{u}$
are $C^{\infty}$ in space and time before the blow-up time.
\end{rem}
With Theorem \ref{thm:MAIN THM}, the existence of screened singularities
for the non-homogeneous Euler equations is proven, but the following
question remains open:
\begin{question}
Is there any unscreened singularity of system \eqref{eq:inhomogeneous 2D Euler}
(see definition \ref{def:singularity classes nonhomogeneous})?
\end{question}

\subsection{Functional spaces and notation}

Let $p\in\left[1,\infty\right]$, $d\in\mathbb{N}$ and $w\in\mathbb{R}^{d}$.
We will use the notation $\left|\left|w\right|\right|_{p}$ to refer
to the $p$-norm of a vector in the Euclidean space $\mathbb{R}^{d}$,
i.e., 
\[
\left|\left|w\right|\right|_{p}\coloneqq\left(\sum^{d}_{n=1}\left|w_{n}\right|^{p}\right)^{\frac{1}{p}},\quad\left|\left|w\right|\right|_{\infty}\coloneqq\max_{n=1,\dots,d}\left|w_{n}\right|.
\]
There will be no confusion with norms in the Lebesgue spaces $L^{p}\left(\mathbb{R}^{d}\right)$,
because we will write these as $\left|\left|\cdot\right|\right|_{L^{p}\left(\mathbb{R}^{d}\right)}$.

Let $d\in\mathbb{N}$ and $F\subseteq\mathbb{R}^{d}$ be closed in
$\mathbb{R}^{d}$. We recall that a function $f:F\to\mathbb{R}$ is
said to be Hölder of parameter $\alpha\in\left(0,1\right)$ (denoted
$f\in C^{\alpha}\left(F\right)$) if
\[
\left|\left|f\right|\right|_{\dot{C}^{\alpha}\left(F\right)}\coloneqq\sup_{x,y\in F}\frac{\left|f\left(x\right)-f\left(y\right)\right|}{\left|\left|x-y\right|\right|^{\alpha}_{2}}<\infty.
\]
$\left|\left|\cdot\right|\right|_{\dot{C}^{\alpha}\left(F\right)}$
is a seminorm. The set $C^{\alpha}\left(F\right)$ can be made into
a Banach space by considering the norm
\[
\left|\left|f\right|\right|_{C^{\alpha}\left(F\right)}\coloneqq\max\left\{ \left|\left|f\right|\right|_{L^{\infty}\left(F\right)},\left|\left|f\right|\right|_{\dot{C}^{\alpha}\left(F\right)}\right\} .
\]
Likewise, we will work with the Hölder spaces 
\[
C^{k,\alpha}\left(F\right)=\left\{ f:F\to\mathbb{R}\text{ of class }C^{k}\text{ and with all derivatives of order }k\text{ belonging to }C^{\alpha}\left(F\right)\right\} 
\]
with $k\in\mathbb{N}$ and $\alpha\in\left(0,1\right)$. We will equip
this space with the norm
\[
\left|\left|f\right|\right|_{C^{k,\alpha}\left(F\right)}\coloneqq\max\left\{ \left|\left|f\right|\right|_{C^{k}\left(F\right)},\left|\left|f\right|\right|_{\dot{C}^{k,\alpha}\left(F\right)}\right\} ,
\]
where
\[
\left|\left|f\right|\right|_{C^{k}\left(F\right)}=\max_{0\le m\le k}\left|\left|f\right|\right|_{\dot{C}^{m}\left(F\right)},\quad\left|\left|f\right|\right|_{\dot{C}^{k}\left(F\right)}=\max_{\left|\beta\right|=k}\left|\left|\mathrm{D}^{\beta}f\right|\right|_{L^{\infty}\left(F\right)},
\]
\[
\left|\left|f\right|\right|_{\dot{C}^{k,\alpha}\left(F\right)}=\max_{\left|\beta\right|=k}\left|\left|\mathrm{D}^{\beta}f\right|\right|_{\dot{C}^{\alpha}\left(F\right)},\quad\mathrm{D}^{\beta}f\coloneqq\frac{\partial^{\left|\beta\right|}f}{\partial x^{\beta_{1}}_{1}\dots\partial x^{\beta_{d}}_{d}},\quad\beta\in\mathbb{N}^{d}_{0},\quad\left|\beta\right|=\beta_{1}+\dots+\beta_{d}.
\]

We will make use of the following two properties of the $\left|\left|\cdot\right|\right|_{\dot{C}^{\alpha}}$-seminorm:
\begin{enumerate}
\item $\forall f,g\in C^{\alpha}\left(F\right)$, we have
\begin{equation}
\left|\left|fg\right|\right|_{\dot{C}^{\alpha}\left(F\right)}\le\left|\left|f\right|\right|_{L^{\infty}\left(F\right)}\left|\left|g\right|\right|_{\dot{C}^{\alpha}\left(F\right)}+\left|\left|f\right|\right|_{\dot{C}^{\alpha}\left(F\right)}\left|\left|g\right|\right|_{L^{\infty}\left(F\right)},\label{eq:property Calpha multiplication}
\end{equation}
\item $\forall f\in C^{\alpha}\left(F\right)$ and $\forall\varphi\in C^{1}\left(F;\varphi\left(F\right)\right)$,
we have
\begin{equation}
\left|\left|f\circ\varphi\right|\right|_{\dot{C}^{\alpha}\left(F\right)}\leq K_{0}\left(d\right)\left|\left|\varphi\right|\right|^{\alpha}_{\dot{C}^{1}\left(F;\varphi\left(F\right)\right)}\left|\left|f\right|\right|_{\dot{C}^{\alpha}\left(\varphi\left(F\right)\right)}.\label{eq:property Calpha composition}
\end{equation}
\end{enumerate}
Moreover, we will work with the spaces
\[
\begin{aligned}L^{\infty}_{t}X_{x}\left(I\times F\right) & =L^{\infty}\left(I;X\left(F\right)\right)=\left\{ f:I\to X\left(F\right)\text{ s.t. }\sup_{t\in I}\left|\left|f\left(t\right)\right|\right|_{X\left(F\right)}<\infty\right\} ,\\
C^{0}_{t}X_{x}\left(I\times F\right) & =C^{0}\left(I;X\left(F\right)\right)=\left\{ f:I\to X\left(F\right)\text{ continuous}\right\} ,
\end{aligned}
\]
where $X$ is any Banach space of functions over $F$, $I$ is any
real interval and $F$ is closed subset of $\mathbb{R}^{2}$. Both
spaces will be equipped with the norm
\[
\left|\left|f\right|\right|_{L^{\infty}_{t}X_{x}\left(I\times F\right)}=\sup_{t\in I}\left|\left|f\left(t\right)\right|\right|_{X\left(F\right)}.
\]
We will use a lower-case $c$ added as a subscript to the space $X_{x}$
to indicate that the functions of $X$ have compact support and that
their compact support is uniform in time. We will specially use the
spaces $C^{0}_{t}C^{1,\alpha}_{x}$ and $C^{0}_{t}L^{p}_{x}$.

Throughout the article, the notation $\Upsilon\left(a,b,\dots\right)$
will be used to denote a positive threshold, depending only on $a,b,\dots$,
whose value may change from one result to the next one. We will allow
$\Upsilon$ to carry numerical subscripts $\left(\Upsilon_{1},\Upsilon_{2},\dots\right)$
when several distinct thresholds appear within the same statement.

\section{\label{sec:ideas of the proof}Heuristics and outline of the proof}

\subsection{Flexibility of the Boussinesq blow-up}

The main idea is to translate the blow-up we constructed for the 2D
incompressible Boussinesq equations in \cite{Articulo Boussinesq}
to system \eqref{eq:inhomogeneous 2D Euler}. In order to achieve
this, before studying \eqref{eq:inhomogeneous 2D Euler} in detail,
it will be useful to understand how rigid the solution of \cite{Articulo Boussinesq}
really is. Recall the forced 2D incompressible Boussinesq equations
in vorticity formulation
\begin{equation}
\begin{aligned}\frac{\partial\rho}{\partial t}+u\cdot\nabla\rho & =f_{\rho},\\
\frac{\partial\omega}{\partial t}+u\cdot\nabla\omega & =\frac{\partial\rho}{\partial x_{2}}+f_{\omega},\\
u & =\nabla^{\perp}\Delta^{-1}\omega.
\end{aligned}
\label{eq:forced Boussinesq}
\end{equation}
Since we followed the layer approach, the solution presented in \cite{Articulo Boussinesq}
can be expressed as an infinite sum of building blocks whose support
diminishes as we advance in the sum. This means that, as long as we
are not close to the blow-up point, only a finite number of layers
will be non-zero. In particular, this implies that we may modify our
solution $\left(\rho,\omega\right)$ as we please ``far'' from the
blow-up point and compensate all changes with the force (i.e., choose
the force so that $\left(\rho,\omega\right)$ remains a solution of
the forced 2D Boussinesq equations) without changing the regularity
of the forces $f_{\rho}$ and $f_{\omega}$. Notice that, if we try
to change $\left(\rho,\omega\right)$ near the blow-up point and compensate
these changes with the forces $\left(f_{\rho},f_{\omega}\right)$,
we could easily destroy the summability of the series that define
$f_{\rho}$ and $f_{\omega}$, leading to a catastrophic loss of regularity.
Coming back to the changes ``far'' from the blow-up point, observe
that, instead of modifying the solution $\left(\rho,\omega\right)$,
we could vary the equations themselves and use the forces $f_{\rho}$
and $f_{\omega}$ so that $\left(\rho,\omega\right)$ remain a solution
of the new equations. In other words, this means that, if we wish
to translate the blow-up of \cite{Articulo Boussinesq} to other equations,
the only requirement is for these equations to look like the 2D Boussinesq
equations ``near'' the blow-up point.

\subsection{Translation of the blow-up}

We will work with the vorticity formulation of system \eqref{eq:inhomogeneous 2D Euler},
which we obtained in \eqref{eq:vorticity equation}:
\begin{equation}
\begin{aligned}\frac{\partial\rho}{\partial t}+u\cdot\nabla\rho & =f_{\rho}\quad\left(\text{mass equation}\right),\\
\frac{\partial\omega}{\partial t}+u\cdot\nabla\omega & =\left[f_{u}-\frac{\partial u}{\partial t}-u\cdot\nabla u\right]\cdot\frac{\nabla^{\perp}\rho}{\rho}+\nabla^{\perp}\cdot f_{u}\quad\left(\text{vorticity equation}\right),\\
u & =\nabla^{\perp}\Delta^{-1}\omega\quad\left(\text{velocity-vorticity equation}\right).
\end{aligned}
\label{eq:inhomogeneous Euler equations vorticity}
\end{equation}
Following the idea that the Boussinesq equations are a first order
approximation of the non-homogeneous Euler equations (see subsection
\ref{subsec:Boussinesq is first order model of non-homogeneous Euler}),
it makes sense to take $\rho=\overline{\rho}+\rho_{B}$ and $u=u_{B}$
(which implies $\omega=\omega_{B}$) above, where $\overline{\rho}>0$
is a constant and $\left(\rho_{B},u_{B},f_{\rho_{B}},f_{\omega_{B}}\right)$
is one of the solutions of 2D Boussinesq constructed in \cite{Articulo Boussinesq}.
Notice that, in the non-homogeneous Euler equations, $\rho$ denotes
a physical density and, consequently, must remain nonnegative at all
times, whereas the density in the Boussinesq equation can have an
alternating sign. Therefore, $\overline{\rho}$ must be chosen so
that $\rho$ remains positive for all time. These choices lead us
to
\[
\begin{aligned}\frac{\partial\rho_{B}}{\partial t}+u_{B}\cdot\nabla\rho_{B} & =f_{\rho},\\
\frac{\partial\omega_{B}}{\partial t}+u_{B}\cdot\nabla\omega_{B} & =\left[f_{u}-\frac{\partial u_{B}}{\partial t}-u_{B}\cdot\nabla u_{B}\right]\cdot\frac{\nabla^{\perp}\rho_{B}}{\overline{\rho}+\rho_{B}}+\nabla^{\perp}\cdot f_{u}.
\end{aligned}
\]
Now, making use of the fact that $\left(\rho_{B},\omega_{B},f_{\rho_{B}},f_{\omega_{B}}\right)$
solve the Boussinesq system \eqref{eq:forced Boussinesq}, we deduce
that
\begin{equation}
\begin{aligned}f_{\rho_{B}} & =f_{\rho},\\
\underbrace{\frac{\partial\rho_{B}}{\partial x_{2}}\left(1+\frac{1}{\overline{\rho}+\rho_{B}}\left(f_{u,1}-\left(\frac{\partial u_{B}}{\partial t}+u_{B}\cdot\nabla u_{B}\right)_{1}\right)\right)}_{\eqqcolon S_{1}}+\\
\underbrace{-\frac{1}{\overline{\rho}+\rho_{B}}\frac{\partial\rho_{B}}{\partial x_{1}}\left(f_{u,2}-\left(\frac{\partial u_{B}}{\partial t}+u_{B}\cdot\nabla u_{B}\right)_{2}\right)}_{\eqqcolon S_{2}}+\\
+f_{\omega_{B}} & =\nabla^{\perp}\cdot f_{u}.
\end{aligned}
\label{eq:vorticity equation first attempt}
\end{equation}
The density equation above is easy to interpret, whereas the vorticity
equation turns out to be more delicate. To analyze it, we should have
the following ideas in mind:
\begin{enumerate}
\item For $f_{u}$ to be in the well-posedness regime, we need $f_{u}\in C^{0}_{t}C^{1,\alpha}_{x}\left(\left[0,1\right]\times\mathbb{R}^{2}\right)$
for some $\alpha>0$.
\item We expect $f_{u}$ to be one derivative more regular than $\nabla^{\perp}\cdot f_{u}$.
\item Both $f_{u}$ and $\nabla^{\perp}\cdot f_{u}$ appear in the equation.
\item $f_{\omega_{B}}\in C^{0}_{t}C^{\alpha}_{x}\left(\left[0,1\right]\times\mathbb{R}^{2}\right)$
for some $\alpha>0$ small enough. (See Main Theorem of \cite{Articulo Boussinesq}).
\item $\frac{\partial\rho_{B}}{\partial x_{2}}$ blows up in $\left|\left|\cdot\right|\right|_{L^{\infty}\left(\mathbb{R}^{2}\right)}$
at $t=1$. (See Proposition \ref{prop:bad bound drhodx2}).
\item $\frac{\partial\rho_{B}}{\partial x_{1}}\in C^{0}_{t}C^{\alpha}_{x}\left(\left[0,1\right]\times\mathbb{R}^{2}\right)$
for some $\alpha>0$ small enough. (See Proposition \ref{prop:bounded drhodx1}).
\item $\rho_{B}\in C^{0}_{t}C^{1-}_{x}\left(\left[0,1\right]\times\mathbb{R}^{2}\right)$.
(See Proposition \ref{prop:bounded density}).
\item $\frac{\partial u_{B}}{\partial t}+u_{B}\cdot\nabla u_{B}\in C^{0}_{t}C^{1-}_{x}\left(\left[0,1\right]\times\mathbb{R}^{2}\right)$.
(See Proposition \ref{prop:bound material derivative u}).
\end{enumerate}
Because of point 4, it is clear that $f_{\omega_{B}}$ poses no obstacle
for $\nabla^{\perp}\cdot f_{u}$ being $C^{0}_{t}C^{\alpha}_{x}\left(\left[0,1\right]\times\mathbb{R}^{2}\right)$.
Since we expect $f_{u}\in C^{0}_{t}C^{1,\alpha}_{x}\left(\left[0,1\right]\times\mathbb{R}^{2}\right)$,
by points 6, 7 and 8, neither does the second summand $S_{2}$ (the
one where $\frac{\partial\rho_{B}}{\partial x_{1}}$ appears). However,
the same cannot be said about the first summand $S_{1}$ (the one
that contains $\frac{\partial\rho_{B}}{\partial x_{2}}$); the fact
that $\frac{\partial\rho_{B}}{\partial x_{2}}$ blows up in $\left|\left|\cdot\right|\right|_{L^{\infty}\left(\mathbb{R}^{2}\right)}$
prevents the first summand $S_{1}$ from being $C^{0}_{t}C^{\alpha}_{x}\left(\left[0,1\right]\times\mathbb{R}^{2}\right)$
unless a miraculous cancellation happens. Since we absolutely need
$\nabla^{\perp}\cdot f_{u}\in C^{0}_{t}C^{\alpha}_{x}\left(\left[0,1\right]\times\mathbb{R}^{2}\right)$,
let us think about how such a cancellation could occur. To do this,
let us resort to a one-dimensional analog. On the one hand, even if
$\frac{\partial\rho_{B}}{\partial x_{2}}$ blows up in $\left|\left|\cdot\right|\right|_{L^{\infty}\left(\mathbb{R}^{2}\right)}$
at $t=1$, we know that $\frac{\partial\rho_{B}}{\partial x_{2}}$
remains $C^{\infty}$ away from the singularity, so maybe the function
$g\left(x\right)=\frac{1}{\left|x\right|^{\delta}}$ (with $\delta>0$)
is a good analog of $\frac{\partial\rho_{B}}{\partial x_{2}}$. On
the other hand, the factor that multiplies $\frac{\partial\rho_{B}}{\partial x_{2}}$
is $C^{0}_{t}C^{1-}_{x}\left(\left[0,1\right]\times\mathbb{R}^{2}\right)$
according to points 7 and 8, so its analog should be any $C^{1-}\left(\mathbb{R}\right)$
function $f$. Is there any $f$ such that $fg\in C^{\alpha}\left(\mathbb{R}\right)$
for some $\alpha>0$? The answer to this question is clearly in the
affirmative: just choose $f\left(x\right)=\left|x\right|^{\delta+\alpha}$
and, then, $fg\in C^{\alpha}\left(\mathbb{R}\right)$. Needless to
say, there are many more functions $f$ that would be valid. What
do all these $f$'s have in common? They all vanish at the origin
and vanish sufficiently fast to compensate the singularity of $g$.
Actually, we will prove in Proposition \ref{prop:dyadic decomposition H=0000F6lder}
that, provided that $\delta<1$, the conditions $f\left(0\right)=0$
and $f\in C^{\delta+}\left(\mathbb{R}\right)$ are enough to ensure
that $fg\in C^{\alpha}\left(\mathbb{R}\right)$ for some small $\alpha>0$.
Thus, we conclude from this one-dimensional analog that we need the
factor that multiplies $\frac{\partial\rho_{B}}{\partial x_{2}}$
to vanish at the blow-up point at the blow-up time. How could we achieve
this? Using the force $f_{u}$, of course!

Here is where opting for a screened singularity instead of an unscreened
one gives us extra degrees of freedom to accomplish our goal. Indeed,
if we were looking for an unscreened singularity, in light of Remark
\ref{rem:first-class singularity}, nothing could be won by considering
$\phi\neq0$ in the force decomposition (see equation \eqref{eq:inhomogeneous canonical decomposition fu}),
so we would have to impose $\nabla\cdot f_{u}=0$, which would make
$f_{u}$ completely determined by $\nabla^{\perp}\cdot f_{u}$. Fortunately,
this is not our case, so we can choose
\begin{equation}
f_{u}\left(t,x\right)=g\left(t\right)\varPhi\left(x\right)+\nabla^{\perp}\varPsi\left(t,x\right),\label{eq:fu first attempt}
\end{equation}
where $\varPhi\in C^{\infty}_{c}\left(\mathbb{R}^{2}\right)$ is $1$
at the blow-up point and $g\left(t\right)$ is a free time-dependent
vector that we can use to make the desired factor vanish. Notice that
we should choose $\varPhi$ even in $x_{2}$ because $\left(\frac{\partial u_{B}}{\partial t}+u_{B}\cdot\nabla u_{B}\right)_{1}$
is also even in $x_{2}$ (see Proposition \ref{prop:oddness and evenness U}).
The divergence-free factor $\nabla^{\perp}\varPsi$ will capture the
rest of the force. Substituting \eqref{eq:fu first attempt} into
\eqref{eq:vorticity equation first attempt} provides
\begin{equation}
\begin{aligned}\frac{\partial\rho_{B}}{\partial x_{2}}\left(1+\frac{1}{\overline{\rho}+\rho_{B}}\left(g_{1}\left(t\right)\varPhi\left(x\right)-\frac{\partial\varPsi}{\partial x_{2}}-\left(\frac{\partial u_{B}}{\partial t}+u_{B}\cdot\nabla u_{B}\right)_{1}\right)\right)+\\
-\frac{1}{\overline{\rho}+\rho_{B}}\frac{\partial\rho_{B}}{\partial x_{1}}\left(g_{2}\left(t\right)\varPhi\left(x\right)+\frac{\partial\varPsi}{\partial x_{1}}-\left(\frac{\partial u_{B}}{\partial t}+u_{B}\cdot\nabla u_{B}\right)_{2}\right)+\\
+f_{\omega_{B}}-g\left(t\right)\cdot\nabla^{\perp}\varPhi\left(x\right) & =\Delta\varPsi,
\end{aligned}
\label{eq:vorticity equation final first attempt}
\end{equation}
where $\left(g_{1}\left(t\right),g_{2}\left(t\right)\right)=g\left(t\right)$.
Clearly, the extra term $g\left(t\right)\cdot\nabla^{\perp}\varPhi\left(x\right)$
does no harm to the regularity of $\varPsi$ and we can make the desired
factor vanish by setting
\begin{equation}
\begin{aligned}1 & +\frac{1}{\overline{\rho}+\rho_{B}\left(t,0\right)}\left[g_{1}\left(t\right)-\frac{\partial\varPsi}{\partial x_{2}}\left(t,0\right)-\left(\frac{\partial u_{B}}{\partial t}+u_{B}\cdot\nabla u_{B}\right)_{1}\left(t,0\right)\right]=0\iff\\
\iff & g_{1}\left(t\right)=\left(\frac{\partial u_{B}}{\partial t}+u_{B}\cdot\nabla u_{B}\right)_{1}\left(t,0\right)+\frac{\partial\varPsi}{\partial x_{2}}\left(t,0\right)-\left(\overline{\rho}+\rho_{B}\left(t,0\right)\right),
\end{aligned}
\label{eq:g1 first attempt}
\end{equation}
where we are assuming for simplicity that the singularity of $\rho_{B}$
takes place at the origin. (Notice that we can always ensure this
is the case by displacing our solution appropriately). Nonetheless,
a key difficulty remains: $\Delta\varPsi$ and $\nabla\varPsi$ appear
in the same equation; we need to know $\nabla\varPsi$ to compute
$g\left(t\right)$, but we need $g\left(t\right)$ to obtain $\Delta\varPsi$,
which determines $\nabla\varPsi$. One way out of this vicious circle
is to find a scenario where changes in $\nabla^{\perp}\varPsi$ almost
do not alter the value of the left-hand-side of equation \eqref{eq:vorticity equation final first attempt}\@.
In this context, a fixed-point argument can be set up to find the
desired $\varPsi$. We provide the details of the proof in section
\ref{sec:fixed-point argument}.

Lastly, we formalize some of the ideas into a Choice.
\begin{choice}
\label{choice: fu}We choose
\[
\rho=\overline{\rho}+\rho_{B},\quad u=u_{B},\quad f_{u}=g\left(t\right)\varPhi\left(x\right)+\nabla^{\perp}\varPsi
\]
in \eqref{eq:vorticity equation}, where $\left(\rho_{B},u_{B},f_{\rho_{B}},f_{\omega_{B}}\right)$
is one of the solutions constructed in \cite{Articulo Boussinesq},
$\varPhi\in C^{\infty}_{c}\left(\mathbb{R}^{2}\right)$, $\varPhi$
is even in $x_{2}$ and $\varPhi\left(0\right)=1$.
\end{choice}

\section{Some useful properties of Hölder spaces}

\subsection{Miscellaneous}
\begin{lem}
\label{lem:inverse Holder}Let $f\in C^{\alpha}\left(\mathbb{R}^{d}\right)$
bounded below, i.e., $\exists\lambda>0$ such that $f\left(x\right)\ge\lambda$
$\forall x\in\mathbb{R}^{d}$. Then, $\frac{1}{f}\in C^{\alpha}\left(\mathbb{R}^{d}\right)$.
Moreover,
\[
\left|\left|\frac{1}{f}\right|\right|_{L^{\infty}\left(\mathbb{R}^{d}\right)}\leq\frac{1}{\lambda},\quad\left|\left|\frac{1}{f}\right|\right|_{\dot{C}^{\alpha}\left(\mathbb{R}^{d}\right)}\leq\frac{1}{\lambda^{2}}\left|\left|f\right|\right|_{\dot{C}^{\alpha}\left(\mathbb{R}^{d}\right)}.
\]
\end{lem}
\begin{proof}
The bound for $\left|\left|\frac{1}{f}\right|\right|_{L^{\infty}\left(\mathbb{R}^{d}\right)}$
is obvious. For the second statement, let $x,y\in\mathbb{R}^{d}$
and consider
\[
\left|\frac{1}{f\left(y\right)}-\frac{1}{f\left(x\right)}\right|=\left|g\left(f\left(y\right)\right)-g\left(f\left(x\right)\right)\right|,
\]
where $g\left(z\right)=\frac{1}{z}$. By the Mean Value Theorem,
\[
\left|\frac{1}{f\left(y\right)}-\frac{1}{f\left(x\right)}\right|\leq\left(\sup_{z\in\overline{f\left(\left[x,y\right]\right)}}\left|g'\left(z\right)\right|\right)\left|f\left(y\right)-f\left(x\right)\right|,
\]
where $\left[x,y\right]$ denotes the segment that joins $x$ and
$y$ and $\overline{f\left(\left[x,y\right]\right)}$ represents the
closure of its image. Since $f$ is bounded below and $\left|g'\left(z\right)\right|=\frac{1}{z^{2}}$
is decreasing, we infer that
\[
\left|\frac{1}{f\left(y\right)}-\frac{1}{f\left(x\right)}\right|\leq\frac{1}{\lambda^{2}}\left|\left|f\right|\right|_{\dot{C}^{\alpha}}\left|\left|x-y\right|\right|^{\alpha}_{2}.
\]
Dividing by $\left|\left|x-y\right|\right|^{\alpha}_{2}$ at both
sides and taking suprema over $x$ and $y$ provides the result.
\end{proof}

\begin{prop}
\label{prop:dyadic decomposition H=0000F6lder}Consider $\beta\in\left(0,1\right)$.
Let $f\in C^{\beta}\left(\mathbb{R}^{2}\right)$ and $g\in C^{\beta}\left(\mathbb{R}^{2}\setminus B\left(0;r\right)\right)$
$\forall r>0$ with
\begin{enumerate}
\item $f\left(0\right)=0$,
\item $\left|\left|g\right|\right|_{L^{\infty}\left(\mathbb{R}^{2}\setminus B\left(0;r\right)\right)}\leq K_{0}\frac{1}{r^{\sigma}}$
$\forall r\in\left(0,R\right]$, for some $K_{0}>0$, some $R>0$
and some $0<\sigma<\beta<1$,
\item $\left|\left|g\right|\right|_{\dot{C}^{\beta-\sigma}\left(\mathbb{R}^{2}\setminus B\left(0;r\right)\right)}\leq K_{\beta-\sigma}\frac{1}{r^{\eta}}$
$\forall r\in\left(0,R\right]$, for some $K_{\beta-\sigma}>0$ and
some $0<\eta<\beta$.
\end{enumerate}
Then, $fg$ admits an extension to the origin such that $\left(fg\right)\left(0\right)=0$
and $fg\in C^{\beta-\sigma}\left(\mathbb{R}^{2}\right)$. Furthermore,
\[
\begin{aligned}\left|\left|fg\right|\right|_{L^{\infty}\left(\mathbb{R}^{2}\right)} & \leq\max\left\{ K_{0}R^{\beta-\sigma}\left|\left|f\right|\right|_{\dot{C}^{\beta}\left(\mathbb{R}^{2}\right)},\left|\left|f\right|\right|_{L^{\infty}\left(\mathbb{R}^{2}\setminus\overline{B\left(0;R\right)}\right)}\left|\left|g\right|\right|_{L^{\infty}\left(\mathbb{R}^{2}\setminus\overline{B\left(0;R\right)}\right)}\right\} ,\\
\left|\left|fg\right|\right|_{\dot{C}^{\beta-\sigma}\left(\mathbb{R}^{2}\right)} & \leq\max\left\{ \left|\left|g\right|\right|_{L^{\infty}\left(\mathbb{R}^{2}\setminus\overline{B\left(0;R\right)}\right)}\left|\left|f\right|\right|_{\dot{C}^{\beta-\sigma}\left(\mathbb{R}^{2}\setminus\overline{B\left(0;R\right)}\right)}+\left|\left|f\right|\right|_{L^{\infty}\left(\mathbb{R}^{2}\setminus\overline{B\left(0;R\right)}\right)}\left|\left|g\right|\right|_{\dot{C}^{\beta-\sigma}\left(\mathbb{R}^{2}\setminus\overline{B\left(0;R\right)}\right)},\right.\\
 & \quad\left.K_{\beta-\sigma}R^{\beta-\eta}\left|\left|f\right|\right|_{\dot{C}^{\beta}\left(\mathbb{R}^{2}\right)}+\max\left\{ 2^{\sigma}K_{0}\left|\left|f\right|\right|_{\dot{C}^{\beta}\left(\mathbb{R}^{2}\right)},\left|\left|g\right|\right|_{L^{\infty}\left(\mathbb{R}^{2}\setminus\overline{B\left(0;R\right)}\right)}\left|\left|f\right|\right|_{\dot{C}^{\beta-\sigma}\left(\mathbb{R}^{2}\right)}\right\} \right\} .
\end{aligned}
\]

\end{prop}
\begin{proof}
First of all, we will see that $fg$ admits an extension to the origin
and that $\left(fg\right)\left(0\right)=0$. Let $x\in\partial B\left(0;r\right)$
with $r\in\left(0,R\right]$. Then,
\[
\left|\left(fg\right)\left(x\right)\right|=\left|f\left(x\right)\right|\left|g\left(x\right)\right|=\left|f\left(x\right)-\overbrace{f\left(0\right)}^{=0}\right|\left|g\left(x\right)\right|\leq\left|\left|f\right|\right|_{\dot{C}^{\beta}\left(\mathbb{R}^{2}\right)}r^{\beta}K_{0}\frac{1}{r^{\sigma}}=K_{0}\left|\left|f\right|\right|_{\dot{C}^{\beta}\left(\mathbb{R}^{2}\right)}r^{\beta-\sigma}\xrightarrow[r\to0]{}0
\]
since $\beta>\sigma$.

Now, we will deal with $\left|\left|fg\right|\right|_{L^{\infty}\left(\mathbb{R}^{2}\right)}$.
Clearly,
\[
\begin{aligned}\left|\left|fg\right|\right|_{L^{\infty}\left(\mathbb{R}^{2}\right)} & =\sup_{x\in\mathbb{R}^{2}}\left|f\left(x\right)\right|\left|g\left(x\right)\right|=\sup_{r>0}\sup_{\left|\left|x\right|\right|_{2}=r}\left|f\left(x\right)\right|\left|g\left(x\right)\right|=\\
 & =\max\left\{ \sup_{r\in\left(0,R\right]}\sup_{\left|\left|x\right|\right|_{2}=r}\left|f\left(x\right)-\overbrace{f\left(0\right)}^{=0}\right|\left|g\left(x\right)\right|,\sup_{r\in\left(R,\infty\right)}\sup_{\left|\left|x\right|\right|_{2}=r}\left|f\left(x\right)\right|\left|g\left(x\right)\right|\right\} \leq\\
 & \leq\max\left\{ \sup_{r\in\left(0,R\right]}\left|\left|f\right|\right|_{\dot{C}^{\beta}\left(\mathbb{R}^{2}\right)}r^{\beta}K_{0}\frac{1}{r^{\sigma}},\left|\left|f\right|\right|_{L^{\infty}\left(\mathbb{R}^{2}\setminus\overline{B\left(0;R\right)}\right)}\left|\left|g\right|\right|_{L^{\infty}\left(\mathbb{R}^{2}\setminus\overline{B\left(0;R\right)}\right)}\right\} =\\
 & \leq\max\left\{ K_{0}\left|\left|f\right|\right|_{\dot{C}^{\beta}\left(\mathbb{R}^{2}\right)}\underbrace{\sup_{r\in\left(0,R\right]}r^{\beta-\sigma}}_{=R^{\beta-\sigma}\,\because\beta>\sigma},\left|\left|f\right|\right|_{L^{\infty}\left(\mathbb{R}^{2}\setminus\overline{B\left(0;R\right)}\right)}\left|\left|g\right|\right|_{L^{\infty}\left(\mathbb{R}^{2}\setminus\overline{B\left(0;R\right)}\right)}\right\} =\\
 & \leq\max\left\{ K_{0}R^{\beta-\sigma}\left|\left|f\right|\right|_{\dot{C}^{\beta}\left(\mathbb{R}^{2}\right)},\left|\left|f\right|\right|_{L^{\infty}\left(\mathbb{R}^{2}\setminus\overline{B\left(0;R\right)}\right)}\left|\left|g\right|\right|_{L^{\infty}\left(\mathbb{R}^{2}\setminus\overline{B\left(0;R\right)}\right)}\right\} .
\end{aligned}
\]

Similarly, we can find bounds for 
\[
\begin{aligned}\left|\left|fg\right|\right|_{\dot{C}^{\beta-\sigma}\left(\mathbb{R}^{2}\right)} & =\sup_{\substack{x,y\in\mathbb{R}^{2}\\
x\neq y
}
}\frac{\left|\left(fg\right)\left(x\right)-\left(fg\right)\left(y\right)\right|}{\left|\left|x-y\right|\right|^{\beta-\sigma}_{2}}.\end{aligned}
\]
Inside the supremum, we can assume without loss of generality that
$\left|\left|x\right|\right|_{2}\leq\left|\left|y\right|\right|_{2}$.
Moreover, thanks to the identity
\[
f\left(x\right)g\left(x\right)-f\left(y\right)g\left(y\right)=\left(f\left(x\right)-f\left(y\right)\right)g\left(y\right)+\left(g\left(x\right)-g\left(y\right)\right)f\left(x\right),
\]
we can write
\[
\begin{aligned}\left|\left|fg\right|\right|_{\dot{C}^{\beta-\sigma}\left(\mathbb{R}^{2}\right)} & \leq\sup_{\substack{x,y\in\mathbb{R}^{2}\\
x\neq y\\
\left|\left|x\right|\right|_{2}\leq\left|\left|y\right|\right|_{2}
}
}\left(\left|g\left(y\right)\right|\frac{\left|f\left(x\right)-f\left(y\right)\right|}{\left|\left|x-y\right|\right|^{\beta-\sigma}_{2}}+\left|f\left(x\right)\right|\frac{\left|g\left(x\right)-g\left(y\right)\right|}{\left|\left|x-y\right|\right|^{\beta-\sigma}_{2}}\right)=\\
 & \leq\sup_{r>0}\sup_{\left|\left|x\right|\right|_{2}=r}\sup_{\left|\left|y\right|\right|_{2}\geq\left|\left|x\right|\right|_{2}}\left(\left|g\left(y\right)\right|\frac{\left|f\left(x\right)-f\left(y\right)\right|}{\left|\left|x-y\right|\right|^{\beta-\sigma}_{2}}+\left|f\left(x\right)\right|\frac{\left|g\left(x\right)-g\left(y\right)\right|}{\left|\left|x-y\right|\right|^{\beta-\sigma}_{2}}\right)=\\
 & \leq\max\left\{ \overbrace{\sup_{r\in\left(0,R\right]}\sup_{\left|\left|x\right|\right|_{2}=r}\sup_{\left|\left|y\right|\right|_{2}\geq\left|\left|x\right|\right|_{2}}\left(\left|g\left(y\right)\right|\frac{\left|f\left(x\right)-f\left(y\right)\right|}{\left|\left|x-y\right|\right|^{\beta-\sigma}_{2}}+\left|f\left(x\right)-\overbrace{f\left(0\right)}^{=0}\right|\frac{\left|g\left(x\right)-g\left(y\right)\right|}{\left|\left|x-y\right|\right|^{\beta-\sigma}_{2}}\right)}^{\eqqcolon I_{1}},\right.\\
 & \quad\left.\underbrace{\sup_{r\in\left(R,\infty\right)}\sup_{\left|\left|x\right|\right|_{2}=r}\sup_{\left|\left|y\right|\right|_{2}\geq\left|\left|x\right|\right|_{2}}\left(\left|g\left(y\right)\right|\frac{\left|f\left(x\right)-f\left(y\right)\right|}{\left|\left|x-y\right|\right|^{\beta-\sigma}_{2}}+\left|f\left(x\right)\right|\frac{\left|g\left(x\right)-g\left(y\right)\right|}{\left|\left|x-y\right|\right|^{\beta-\sigma}_{2}}\right)}_{\eqqcolon I_{2}}\right\} .
\end{aligned}
\]
Bounding $I_{2}$ is easy because
\[
I_{2}\leq\left|\left|g\right|\right|_{L^{\infty}\left(\mathbb{R}^{2}\setminus\overline{B\left(0;R\right)}\right)}\left|\left|f\right|\right|_{\dot{C}^{\beta-\sigma}\left(\mathbb{R}^{2}\setminus\overline{B\left(0;R\right)}\right)}+\left|\left|f\right|\right|_{L^{\infty}\left(\mathbb{R}^{2}\setminus\overline{B\left(0;R\right)}\right)}\left|\left|g\right|\right|_{\dot{C}^{\beta-\sigma}\left(\mathbb{R}^{2}\setminus\overline{B\left(0;R\right)}\right)}.
\]
Now, concerning $I_{1}$,
\[
\begin{aligned}I_{1} & \leq\sup_{r\in\left(0,R\right]}\sup_{\left|\left|x\right|\right|_{2}=r}\left|\left|f\right|\right|_{\dot{C}^{\beta}\left(\mathbb{R}^{2}\right)}\left|\left|x\right|\right|^{\beta}_{2}\left|\left|g\right|\right|_{\dot{C}^{\beta-\sigma}\left(\mathbb{R}^{2}\setminus B\left(0;\left|\left|x\right|\right|_{2}\right)\right)}+\\
 & \quad+\max\left\{ \sup_{r\in\left(0,R\right]}\sup_{\left|\left|x\right|\right|_{2}=r}\sup_{\left|\left|x\right|\right|_{2}\leq\left|\left|y\right|\right|_{2}\leq R}\left|\left|g\right|\right|_{L^{\infty}\left(\mathbb{R}^{2}\setminus B\left(0;\left|\left|y\right|\right|_{2}\right)\right)}\left|\left|f\right|\right|_{\dot{C}^{\beta}\left(\mathbb{R}^{2}\right)}\left|\left|x-y\right|\right|^{\sigma}_{2},\right.\\
 & \qquad\left.\sup_{r\in\left(0,R\right]}\sup_{\left|\left|x\right|\right|_{2}=r}\sup_{\left|\left|y\right|\right|_{2}\geq R}\left|\left|g\right|\right|_{L^{\infty}\left(\mathbb{R}^{2}\setminus\overline{B\left(0;R\right)}\right)}\left|\left|f\right|\right|_{\dot{C}^{\beta-\sigma}\left(\mathbb{R}^{2}\right)}\right\} \leq\\
 & \leq\sup_{r\in\left(0,R\right]}\left|\left|f\right|\right|_{\dot{C}^{\beta}\left(\mathbb{R}^{2}\right)}r^{\beta}K_{\beta-\sigma}\frac{1}{r^{\eta}}+\max\left\{ \left|\left|g\right|\right|_{L^{\infty}\left(\mathbb{R}^{2}\setminus\overline{B\left(0;R\right)}\right)}\left|\left|f\right|\right|_{\dot{C}^{\beta-\sigma}\left(\mathbb{R}^{2}\right)},\right.\\
 & \quad\left.\sup_{r\in\left(0,R\right]}\sup_{\left|\left|x\right|\right|_{2}=r}\sup_{\left|\left|x\right|\right|_{2}\leq\left|\left|y\right|\right|_{2}\leq R}K_{0}\frac{1}{\left|\left|y\right|\right|^{\sigma}_{2}}\left|\left|f\right|\right|_{\dot{C}^{\beta}\left(\mathbb{R}^{2}\right)}\left|\left|x-y\right|\right|^{\sigma}_{2}\right\} .
\end{aligned}
\]
Using that
\[
\frac{\left|\left|x-y\right|\right|_{2}}{\left|\left|y\right|\right|_{2}}\leq\frac{\left|\left|x\right|\right|_{2}+\left|\left|y\right|\right|_{2}}{\left|\left|y\right|\right|_{2}}\leq2,
\]
because $\left|\left|x\right|\right|_{2}\leq\left|\left|y\right|\right|_{2}$,
we arrive at
\[
I_{1}\leq K_{\beta-\sigma}\overbrace{\sup_{r\in\left(0,R\right]}r^{\beta-\eta}}^{=R^{\beta-\eta}\;\because\beta>\eta}\left|\left|f\right|\right|_{\dot{C}^{\beta}\left(\mathbb{R}^{2}\right)}+\max\left\{ 2^{\sigma}K_{0}\left|\left|f\right|\right|_{\dot{C}^{\beta}\left(\mathbb{R}^{2}\right)},\left|\left|g\right|\right|_{L^{\infty}\left(\mathbb{R}^{2}\setminus\overline{B\left(0;R\right)}\right)}\left|\left|f\right|\right|_{\dot{C}^{\beta-\sigma}\left(\mathbb{R}^{2}\right)}\right\} .
\]
\end{proof}

We finish this subsection with a partial ``converse'' to the last
Proposition.
\begin{prop}
\label{prop:dyadic decomposition f=00003D0 necessary}Let $f\in C^{0}_{t}C^{0}_{x}\left(\left[0,1\right]\times\mathbb{R}^{2}\right)$
and $g\in L^{\infty}_{t}L^{\infty}_{x}\left(\left[0,1-r\right]\times\mathbb{R}^{2}\right)\cap L^{\infty}_{t}L^{\infty}_{x}\left(\left[0,1\right]\times\left(\mathbb{R}^{2}\setminus B\left(0;r\right)\right)\right)$
$\forall r\in\left(0,1\right)$ with
\begin{enumerate}
\item $f\left(1,0\right)\neq0$,
\item $\int^{1}_{0}\left|\left|g\left(t,\cdot\right)\right|\right|_{L^{\infty}\left(\mathbb{R}^{2}\right)}\mathrm{d}t=\infty$.
\end{enumerate}
Then, $\limsup_{t\to1^{-}}\left|\left|\left(fg\right)\left(t,\cdot\right)\right|\right|_{L^{\infty}\left(\mathbb{R}^{2}\right)}=\infty$.

\end{prop}
\begin{proof}
We will prove the result by contradiction. Assume that $\limsup_{t\to1^{-}}\left|\left|\left(fg\right)\left(t,\cdot\right)\right|\right|_{L^{\infty}\left(\mathbb{R}^{2}\right)}\leq M$
for some $M>0$. Then, on the one hand, $\left|\left|\left(fg\right)\left(t,\cdot\right)\right|\right|_{L^{\infty}\left(\mathbb{R}^{2}\right)}\leq2M$
$\forall t\in\left[1-\delta,1\right)$ for some $\delta>0$ small
enough. On the other hand, by the continuity of $f$, we can guarantee
that $\left|f\left(t,x\right)\right|\geq\frac{1}{2}\left|f\left(1,0\right)\right|$
$\forall\left(t,x\right)\in\left[1-\eta,1\right]\times B\left(0;\eta\right)$
for $\eta>0$ sufficiently small. Take $\varepsilon\coloneqq\min\left\{ \delta,\eta\right\} $.
Then,
\[
2M\geq\left|\left|\left(fg\right)\left(t,\cdot\right)\right|\right|_{L^{\infty}\left(B\left(0;\varepsilon\right)\right)}\geq\frac{1}{2}\left|f\left(1,0\right)\right|\left|\left|g\left(t,\cdot\right)\right|\right|_{L^{\infty}\left(B\left(0;\varepsilon\right)\right)}\quad\text{a.e. }t\in\left[1-\varepsilon,1\right)
\]
from which we deduce
\[
\left|\left|g\left(t,\cdot\right)\right|\right|_{L^{\infty}\left(B\left(0;\varepsilon\right)\right)}\leq\frac{4M}{\left|f\left(1,0\right)\right|}\quad\text{a.e. }t\in\left[1-\varepsilon,1\right).
\]
As $g\in L^{\infty}_{t}L^{\infty}_{x}\left(\left[0,1-\varepsilon\right]\times\mathbb{R}^{2}\right)$
by hypothesis, integrating in time we obtain that
\[
\int^{1}_{0}\left|\left|g\left(t,\cdot\right)\right|\right|_{L^{\infty}\left(B\left(0;\varepsilon\right)\right)}\mathrm{d}t=\int^{1-\varepsilon}_{0}\left|\left|g\left(t,\cdot\right)\right|\right|_{L^{\infty}\left(B\left(0;\varepsilon\right)\right)}\mathrm{d}t+\int^{1}_{1-\varepsilon}\left|\left|g\left(t,\cdot\right)\right|\right|_{L^{\infty}\left(B\left(0;\varepsilon\right)\right)}\mathrm{d}t<\infty.
\]
Moreover, since $g\in L^{\infty}_{t}L^{\infty}_{x}\left(\left[0,1\right]\times\left(\mathbb{R}^{2}\setminus B\left(0;\varepsilon\right)\right)\right)$
and 
\[
\left|\left|g\left(t,\cdot\right)\right|\right|_{L^{\infty}\left(\mathbb{R}^{2}\right)}=\max\left\{ \left|\left|g\left(t,\cdot\right)\right|\right|_{L^{\infty}\left(B\left(0;\varepsilon\right)\right)},\left|\left|g\left(t,\cdot\right)\right|\right|_{L^{\infty}\left(\mathbb{R}^{2}\setminus B\left(0;\varepsilon\right)\right)}\right\} \quad\text{a.e. }t\in\left[1-\varepsilon,1\right),
\]
we conclude that
\[
\int^{1}_{0}\left|\left|g\left(t,\cdot\right)\right|\right|_{L^{\infty}\left(\mathbb{R}^{2}\right)}\mathrm{d}t<\infty,
\]
contradicting hypothesis 2.
\end{proof}

\subsection{Interpolation}

In this section, we recall some well-known interpolation inequalities
in Hölder spaces. They can be seen as a particular case of real-interpolation
of Besov spaces (see, for example, Examples 1.8 and 1.10 of \cite{Lunardi})
or can be obtained with more straightforward methods (see the exercises
at the end of section 3.2 of \cite{krylov}).
\begin{prop}
\label{prop:easy interpolation H=0000F6lder}Let $\alpha\in\left(0,1\right)$,
$\gamma\in\left(0,1\right]$ with $\alpha<\gamma$ and $f\in C^{\gamma}\left(\mathbb{R}^{d}\right)$. 
\begin{enumerate}
\item $\forall\beta\in\left(\alpha,\gamma\right)$,
\[
\left|\left|f\right|\right|_{\dot{C}^{\beta}\left(\mathbb{R}^{d}\right)}\leq\left|\left|f\right|\right|^{\frac{\gamma-\beta}{\gamma-\alpha}}_{\dot{C}^{\alpha}\left(\mathbb{R}^{d}\right)}\left|\left|f\right|\right|^{\frac{\beta-\alpha}{\gamma-\alpha}}_{\dot{C}^{\gamma}\left(\mathbb{R}^{d}\right)}.
\]
\item $\forall\beta\in\left(0,\gamma\right)$,
\[
\left|\left|f\right|\right|_{\dot{C}^{\beta}\left(\mathbb{R}^{d}\right)}\leq2\left|\left|f\right|\right|^{\frac{\gamma-\beta}{\gamma}}_{L^{\infty}\left(\mathbb{R}^{d}\right)}\left|\left|f\right|\right|^{\frac{\beta}{\gamma}}_{\dot{C}^{\gamma}\left(\mathbb{R}^{d}\right)}.
\]
\end{enumerate}
\end{prop}
\begin{prop}
\label{prop:hard interpolation H=0000F6lder}Let $\alpha\in\left(0,1\right)$,
$\gamma\in\left(0,1\right]$ and $f\in C^{1,\gamma}\left(\mathbb{R}^{d}\right)$.
Then,
\begin{enumerate}
\item ~
\[
\left|\left|f\right|\right|_{\dot{C}^{1}\left(\mathbb{R}^{d}\right)}\lesssim_{\alpha,\gamma}\left|\left|f\right|\right|^{\frac{\gamma}{1+\gamma-\alpha}}_{\dot{C}^{\alpha}\left(\mathbb{R}^{d}\right)}\left|\left|f\right|\right|^{\frac{1-\alpha}{1+\gamma-\alpha}}_{\dot{C}^{1,\gamma}\left(\mathbb{R}^{d}\right)},
\]
\item $\forall\beta\in\left(\alpha,1\right)$,
\[
\left|\left|f\right|\right|_{\dot{C}^{\beta}\left(\mathbb{R}^{d}\right)}\lesssim_{\alpha,\gamma}\left|\left|f\right|\right|^{\frac{1+\gamma-\beta}{1+\gamma-\alpha}}_{\dot{C}^{\alpha}\left(\mathbb{R}^{d}\right)}\left|\left|f\right|\right|^{\frac{\beta-\alpha}{1+\gamma-\alpha}}_{\dot{C}^{1,\gamma}\left(\mathbb{R}^{d}\right)},
\]
\item $\forall\beta\in\left(0,\gamma\right)$,
\[
\left|\left|f\right|\right|_{\dot{C}^{1,\beta}\left(\mathbb{R}^{d}\right)}\lesssim_{\alpha,\gamma}\left|\left|f\right|\right|^{\frac{\gamma-\beta}{1+\gamma-\alpha}}_{\dot{C}^{\alpha}\left(\mathbb{R}^{d}\right)}\left|\left|f\right|\right|^{\frac{1-\alpha+\beta}{1+\gamma-\alpha}}_{\dot{C}^{1,\gamma}\left(\mathbb{R}^{d}\right)}.
\]
\end{enumerate}
\end{prop}

\section{\label{sec:properties of Boussinesq solution}Properties of the Boussinesq
solution}

In this section, we will perform multiple computations related to
the solution constructed in \cite{Articulo Boussinesq}. All of these
are necessary to obtain the Main Theorem of this paper, but have no
direct connection with the non-homogeneous Euler equations as they
only deal with the Boussinesq system. Throughout this section, and
only in this section, the subscript $B$ from $\left(\rho_{B},\omega_{B},f_{\rho_{B}},f_{\omega_{B}}\right)$
used to refer to a solution of the Boussinesq system will be dropped
for ease of notation.

\subsection{\label{subsec:summary Boussinesq}Brief summary of the Boussinesq
construction}

In \cite{Articulo Boussinesq}, we construct a finite-time singularity
for the forced Boussinesq system, where the source terms $f_{\rho}$
and $f_{u}$ had regularity $C^{0}_{t}C^{1,\alpha}_{x}\left(\left[0,1\right]\times\mathbb{R}^{2}\right)$.
$\alpha$ could be chosen freely in the interval $\left(0,\alpha_{*}\right)$,
$\alpha_{*}=\sqrt{\frac{4}{3}}-1$. Thereby, $\alpha_{*}$ represents
the ``maximum regularity'' of the Boussinesq solution. Moreover,
the stream function and the density of the solution are given by
\begin{equation}
\begin{aligned}\psi\left(t,x\right) & =\sum^{\infty}_{n=1}\psi^{\left(n\right)}\left(t,x\right),\\
\widetilde{\psi^{\left(n\right)}}^{n}\left(t,x\right) & =B_{n}\left(t\right)\varphi\left(\lambda_{n}x_{1}\right)\varphi\left(\lambda_{n}x_{2}\right)\sin\left(x_{1}\right)\sin\left(x_{2}\right),\\
\rho\left(t,x\right) & =\sum^{\infty}_{n=1}\rho^{\left(n\right)}\left(t,x\right),\\
\widetilde{\rho^{\left(n\right)}}^{n}\left(t,x\right) & =-\frac{1}{b_{n}\left(t\right)}\frac{\mathrm{d}}{\mathrm{d}t}\left[B_{n}\left(t\right)\left(a_{n}\left(t\right)^{2}+b_{n}\left(t\right)^{2}\right)\right]\varphi\left(\lambda_{n}x_{1}\right)\varphi\left(\lambda_{n}x_{2}\right)\sin\left(x_{1}\right)\cos\left(x_{2}\right),
\end{aligned}
\label{eq:Boussinesq psi rho}
\end{equation}
where
\begin{enumerate}
\item $\varphi\in C^{\infty}_{c}\left(\mathbb{R}\right)$ is any compactly
supported even smooth function that satisfies $\varphi\equiv1$ in
the interval $\left[-8\pi,8\pi\right]$ and $\varphi\equiv0$ in $\mathbb{R}\setminus\left[-16\pi,16\pi\right]$.
\item ~
\begin{equation}
\lambda_{n}=C^{-\Lambda\left(\frac{1}{1-\gamma}\right)^{n}}\leq1,\label{eq:Boussinesq lambda_n}
\end{equation}
with $\Lambda=\alpha_{*}\left(1+\alpha_{*}\right)$ and where $\gamma\in\left(\frac{1}{2},1\right)$
is a parameter that is chosen very close to $1$.
\item $\widetilde{\psi^{\left(n\right)}}^{n}\left(t,x\right)=\psi^{\left(n\right)}\left(t,\phi^{\left(n\right)}\left(t,x\right)\right)$,
where $\phi^{\left(n\right)}\left(t,x\right)$ is the change of variables
\begin{equation}
\phi^{\left(n\right)}\left(t,x\right)=\phi^{\left(n\right)}\left(t,0\right)+\left(\frac{1}{a_{n}\left(t\right)}x_{1},\frac{1}{b_{n}\left(t\right)}x_{2}\right).\label{eq:Boussinesq phi}
\end{equation}
Moreover,
\begin{equation}
\left(\mathrm{J}\phi^{\left(n\right)}\right)^{-1}\left(t,x\right)=\mathrm{J}\left(\phi^{\left(n\right)}\right)^{-1}\left(t,\phi^{\left(n\right)}\left(t,x\right)\right)=\left(\begin{matrix}a_{n}\left(t\right) & 0\\
0 & b_{n}\left(t\right)
\end{matrix}\right)\label{eq:Boussinesq jacobian inverse}
\end{equation}
and
\begin{equation}
\tilde{\nabla}^{n}=\left(a_{n}\left(t\right)\frac{\partial}{\partial x_{1}},b_{n}\left(t\right)\frac{\partial}{\partial x_{2}}\right).\label{eq:Boussinesq def nabla_tilde}
\end{equation}
\item The temporal dynamics of $a_{n}\left(t\right)$, $b_{n}\left(t\right)$,
$B_{n}\left(t\right)$ and $\phi^{\left(n\right)}\left(t,0\right)$
are given by
\begin{equation}
\begin{alignedat}{1}\phi^{\left(n\right)}_{2}\left(t,0\right) & =0,\\
\frac{\partial\phi^{\left(n\right)}_{1}}{\partial t}\left(t,0\right) & =\sum^{n-1}_{m=1}B_{m}\left(t\right)b_{m}\left(t\right)\sin\left(a_{m}\left(t\right)\left(\phi^{\left(n\right)}_{1}\left(t,0\right)-\phi^{\left(m\right)}_{1}\left(t,0\right)\right)\right),\\
\frac{\mathrm{d}}{\mathrm{d}t}\left(\ln\left(b_{n}\left(t\right)\right)\right) & =\sum^{n-1}_{m=1}B_{m}\left(t\right)a_{m}\left(t\right)b_{m}\left(t\right)\cos\left(a_{m}\left(t\right)\left(\phi^{\left(n\right)}_{1}\left(t,0\right)-\phi^{\left(m\right)}_{1}\left(t,0\right)\right)\right),\\
\frac{\mathrm{d}}{\mathrm{d}t}\left(a_{n}\left(t\right)b_{n}\left(t\right)\right) & =0,
\end{alignedat}
\label{eq:Boussinesq dynamics}
\end{equation}
$\forall t\in\left[0,1\right]$. $\phi^{\left(n\right)}_{1}\left(t,0\right)$
evolves in time like a weighted sum of inverted degenerate half-pendula:
\[
\phi^{\left(n\right)}_{1}\left(t,0\right)\sim\sum^{n}_{m=1}\overbrace{\frac{1}{a_{m}\left(1\right)}}^{\text{weight}}\overbrace{\left(a_{m}\left(1\right)\Xi^{\left(m\right)}_{0}\left(t\right)\right)}^{{\footnotesize \begin{matrix}\text{inverted degenerate}\\
\text{half-pendulum}
\end{matrix}}},
\]
where
\[
\begin{aligned}\frac{\mathrm{d}\Xi^{\left(n\right)}_{0}}{\mathrm{d}t}\left(t\right) & =B_{n-1}\left(1\right)b_{n-1}\left(1\right)\sin\left(a_{n-1}\left(1\right)\Xi^{\left(n\right)}_{0}\left(t\right)\right)\quad\forall t\in\left[t_{n},1\right],\\
\sin\left(a_{n-1}\left(1\right)\Xi^{\left(n\right)}_{0}\left(t_{n}+\left(1-t_{n}\right)\hat{\hat{t}}\right)\right) & =\frac{1}{\cosh\left(\mathrm{arccosh}\left(C^{k_{\max}\left(\frac{1}{1-\gamma}\right)^{n}}\right)\left(1-2\hat{\hat{t}}\right)\right)}\quad\forall\hat{\hat{t}}\in\left[0,1\right].
\end{aligned}
\]
These choices of $a_{n}\left(t\right)$, $b_{n}\left(t\right)$, $B_{n}\left(t\right)$
and $\phi^{\left(n\right)}\left(t,0\right)$ also guarantee (see Proposition
6 of \cite{Articulo Boussinesq}) that
\begin{equation}
\frac{\partial\phi^{\left(n\right)}}{\partial t}\left(t,x\right)=\widetilde{U^{\left(n-1\right)}}^{n}\left(t,0\right)+\mathrm{J}\widetilde{U^{\left(n-1\right)}}^{n}\left(t,0\right)\cdot\left(\begin{matrix}x_{1}\\
x_{2}
\end{matrix}\right)\quad\forall t\in\left[t_{n},1\right].\label{eq:Boussinesq transport is Taylor of velocity}
\end{equation}
\item \label{enu:The-time-scales}The time scales $\left(t_{n}\right)_{n\in\mathbb{N}}$
are given by
\begin{equation}
1-t_{n}=\frac{1}{Y}C^{-\delta\left(\frac{1}{1-\gamma}\right)^{n-1}}\mathrm{arccosh}\left(C^{k_{\max}\left(\frac{1}{1-\gamma}\right)^{n}}\right),\label{eq:Boussinesq time scale}
\end{equation}
where $k_{\max}=\alpha_{*}$, $Y$ is a normalization constant (to
ensure that $t_{1}=0$) and $\delta$ is a parameter that is chosen
sufficiently small.
\item The initial conditions for $a_{n}\left(t\right)$ and $b_{n}\left(t\right)$
are specified through
\begin{equation}
b_{n}\left(t\right)=C^{\left(1+k_{n}\left(t\right)\right)\left(\frac{1}{1-\gamma}\right)^{n}},\quad a_{n}\left(t\right)=C^{\left(1-k_{n}\left(t\right)\right)\left(\frac{1}{1-\gamma}\right)^{n}},\quad k_{n}\left(1\right)=0,\label{eq:Boussinesq an(t) bn(t)}
\end{equation}
where $C$ is a parameter that is chosen sufficiently large. Moreover,
these functions $k_{n}\left(t\right)$ resemble their ideal models
$\overline{k}_{n}\left(t\right)$, which display the following behavior
in $n$:
\[
k_{n}\left(t_{n}+\left(1-t_{n}\right)\hat{\hat{t}}\right)\overset{n\to\infty}{\sim}k_{\max}\left(1-\left|1-2\hat{\hat{t}}\right|\right)\quad\forall\hat{\hat{t}}\in\left[0,1\right].
\]
By the definition of $k_{\max}$ given in Choice 12 of \cite{Articulo Boussinesq},
we have
\begin{equation}
\max_{s\in\left[t_{n},1\right]}\overline{k}_{n}\left(s\right)=k_{\max}\quad\forall n\in\mathbb{N}.\label{eq:Boussinesq def kmax}
\end{equation}
\item The initial conditions $\phi^{\left(n\right)}\left(t_{n},0\right)=c^{\left(n\right)}$
are given by
\[
\begin{aligned}a_{n-1}\left(1\right)\Xi^{\left(n\right)}\left(t_{n}\right) & =\arcsin\left(C^{-k_{\max}\left(\frac{1}{1-\gamma}\right)^{n}}\right),\\
\Xi^{\left(n\right)}\left(t\right) & =\phi^{\left(n\right)}_{1}\left(t,0\right)-\phi^{\left(n-1\right)}_{1}\left(t,0\right),\\
\phi^{\left(n\right)}_{2}\left(t,0\right) & \equiv0,
\end{aligned}
\]
where we take $\phi^{\left(0\right)}\left(t,0\right)\equiv0$.
\item \label{item:Boussinesq active layer}$B_{n}\left(t\right)$ satisfies
\begin{equation}
\frac{\mathrm{d}}{\mathrm{d}t}\left(B_{n}\left(t\right)\left(a_{n}\left(t\right)^{2}+b_{n}\left(t\right)^{2}\right)\right)=2M_{n}\frac{h^{\left(n\right)}\left(t\right)b_{n}\left(t\right)}{\int^{1}_{t_{n}}h^{\left(n\right)}\left(s\right)b_{n}\left(s\right)\mathrm{d}s},\quad B_{n}\left(t\right)=\frac{2M_{n}}{a_{n}\left(t\right)^{2}+b_{n}\left(t\right)^{2}}\frac{\int^{t}_{t_{n}}h^{\left(n\right)}\left(s\right)b_{n}\left(s\right)\mathrm{d}s}{\int^{1}_{t_{n}}h^{\left(n\right)}\left(s\right)b_{n}\left(s\right)\mathrm{d}s}.\label{eq:Boussinesq Bn(t)}
\end{equation}
Here, $h^{\left(n\right)}\left(s\right)$ is a smooth function with
codomain $\left[0,1\right]$ and such that $\left.h^{\left(n\right)}\right|_{\left[0,t_{n}\right]\cup\left[t_{n+1},1\right]}\equiv0$.
In this way, a single density layer is active at any given time.
\item $M_{n}$ is given by
\begin{equation}
M_{n}=B_{n}\left(1\right)a_{n}\left(1\right)b_{n}\left(1\right)=YC^{\delta\left(\frac{1}{1-\gamma}\right)^{n}}.\label{eq:Boussinesq Mn}
\end{equation}
\end{enumerate}
Notice that, actually, because of how $B_{n}\left(t\right)$ and $t_{n}$
have been defined, for every $t\in\left[0,1\right)$, only a finite
number of summands in $\psi\left(t,x\right)=\sum^{\infty}_{n=1}\psi^{\left(n\right)}\left(t,x\right)$
are nonzero. Thereby, for fixed $t<1$, this is always a finite sum,
i.e., the intrinsic limit associated to the series is illusory. In
the case of the density $\rho\left(t,x\right)=\sum^{\infty}_{n=1}\rho^{\left(n\right)}\left(t,x\right)$,
because of how $h^{\left(n\right)}$ has been chosen, only one summand
is non-zero, i.e., not only is the sum finite, it actually encompasses
one single element. With this construction, the following Theorem
is proved:
\begin{thm}
\label{thm:MAIN THM Boussinesq}Let $\alpha\in\left(0,\alpha_{*}\right)$,
where $\alpha_{*}=\sqrt{\frac{4}{3}}-1$. There are classical solutions
$\left(u,\rho\right)$ in $\left[0,1\right)\times\mathbb{R}^{2}$
of the forced Boussinesq system \eqref{eq:forced Boussinesq} that
satisfy:
\begin{enumerate}
\item $\forall\varepsilon>0$
\[
u\in C^{\infty}_{t}C^{\infty}_{x,c}\left(\left[0,1-\varepsilon\right]\times\mathbb{R}^{2}\right),\quad\rho\in C^{\infty}_{t}C^{\infty}_{x,c}\left(\left[0,1-\varepsilon\right]\times\mathbb{R}^{2}\right).
\]
\item $\forall\varepsilon>0$
\[
f_{\omega}\in C^{\infty}_{t}C^{\infty}_{x,c}\left(\left[0,1-\varepsilon\right]\times\mathbb{R}^{2}\right),\quad f_{\rho}\in C^{\infty}_{t}C^{\infty}_{x,c}\left(\left[0,1-\varepsilon\right]\times\mathbb{R}^{2}\right).
\]
\item ~
\[
f_{\omega}\in C^{0}_{t}C^{\alpha}_{x,c}\left(\left[0,1\right]\times\mathbb{R}^{2}\right),\quad f_{\rho}\in C^{0}_{t}C^{1,\alpha}_{x,c}\left(\left[0,1\right]\times\mathbb{R}^{2}\right).
\]
In particular, $\left|\left|f_{\omega}\left(t,\cdot\right)\right|\right|_{C^{\alpha}\left(\mathbb{R}^{2}\right)}$
and $\left|\left|f_{\rho}\left(t,\cdot\right)\right|\right|_{C^{1,\alpha}\left(\mathbb{R}^{2}\right)}$
are uniformly bounded in $t\in\left[0,1\right]$.
\item $f_{\omega}\left(t,x\right)$ is odd in $x_{2}$ $\forall t\in\left[0,1\right]$.
\item There is a finite-time singularity at $t=1$, i.e.,
\[
\lim_{T\to1^{-}}\int^{T}_{0}\left|\left|\nabla\rho\left(t,\cdot\right)\right|\right|_{L^{\infty}\left(\mathbb{R}^{2};\mathbb{R}^{2}\right)}\mathrm{d}t=\infty.
\]
(See \cite{Chae Kim Nam} for the blow-up criterion).
\end{enumerate}
\end{thm}
Furthermore, let us recall Propositions 4, 5, 9, 11, 12 and 17, Lemmas
4, 5, 6, 7 and 13 and Corollary 2 of \cite{Articulo Boussinesq},
as well as equations (70) and (188) and an identity from the proof
of Proposition 23 of that article:
\begin{prop}[Explicit expressions for the velocity and vorticity]
\label{prop:computations vorticity}{[}Proposition 4 of \cite{Articulo Boussinesq}{]}
We have
\[
\begin{aligned}\widetilde{u^{\left(n\right)}}^{n}\left(t,x\right) & =B_{n}\left(t\right)\varphi\left(\lambda_{n}x_{1}\right)\varphi\left(\lambda_{n}x_{2}\right)\left(\begin{matrix}b_{n}\left(t\right)\sin\left(x_{1}\right)\cos\left(x_{2}\right)\\
-a_{n}\left(t\right)\cos\left(x_{1}\right)\sin\left(x_{2}\right)
\end{matrix}\right)+\\
 & \quad+\lambda_{n}B_{n}\left(t\right)\left(\begin{matrix}b_{n}\left(t\right)\varphi\left(\lambda_{n}x_{1}\right)\varphi'\left(\lambda_{n}x_{2}\right)\\
-a_{n}\left(t\right)\varphi'\left(\lambda_{n}x_{1}\right)\varphi\left(\lambda_{n}x_{2}\right)
\end{matrix}\right)\sin\left(x_{1}\right)\sin\left(x_{2}\right).
\end{aligned}
\]
\[
\begin{aligned}\widetilde{\omega^{\left(n\right)}}^{n}\left(t,x\right) & =B_{n}\left(t\right)\left(a_{n}\left(t\right)^{2}+b_{n}\left(t\right)^{2}\right)\varphi\left(\lambda_{n}x_{1}\right)\varphi\left(\lambda_{n}x_{2}\right)\sin\left(x_{1}\right)\sin\left(x_{2}\right)+\\
 & \quad-\lambda^{2}_{n}B_{n}\left(t\right)a_{n}\left(t\right)^{2}\varphi''\left(\lambda_{n}x_{1}\right)\varphi\left(\lambda_{n}x_{2}\right)\sin\left(x_{1}\right)\sin\left(x_{2}\right)+\\
 & \quad-\lambda^{2}_{n}B_{n}\left(t\right)b_{n}\left(t\right)^{2}\varphi\left(\lambda_{n}x_{1}\right)\varphi''\left(\lambda_{n}x_{2}\right)\sin\left(x_{1}\right)\sin\left(x_{2}\right)+\\
 & \quad-2\lambda_{n}B_{n}\left(t\right)a_{n}\left(t\right)^{2}\varphi'\left(\lambda_{n}x_{1}\right)\varphi\left(\lambda_{n}x_{2}\right)\cos\left(x_{1}\right)\sin\left(x_{2}\right)+\\
 & \quad-2\lambda_{n}B_{n}\left(t\right)b_{n}\left(t\right)^{2}\varphi\left(\lambda_{n}x_{1}\right)\varphi'\left(\lambda_{n}x_{2}\right)\sin\left(x_{1}\right)\cos\left(x_{2}\right).
\end{aligned}
\]
\end{prop}
\begin{prop}[Parameters of past layers are almost constant in time]
\label{prop:time convergence}{[}Proposition 9 of \cite{Articulo Boussinesq}{]}
Let $\beta\in\left(0,1\right)$. Provided that $C$ is sufficiently
big (let us say $C\ge\Upsilon\left(\beta,\delta\right)$), then, $\forall n\in\mathbb{N}$
and $\forall t\in\left[t_{n+1},1\right]$,
\[
\left|\frac{b_{n}\left(t\right)}{b_{n}\left(1\right)}-1\right|\leq C^{-\beta\delta\gamma\left(\frac{1}{1-\gamma}\right)^{n}},\quad\left|\frac{a_{n}\left(t\right)}{a_{n}\left(1\right)}-1\right|\leq C^{-\beta\delta\gamma\left(\frac{1}{1-\gamma}\right)^{n}},\quad\left|\frac{B_{n}\left(t\right)}{B_{n}\left(1\right)}-1\right|\leq C^{-\beta\delta\gamma\left(\frac{1}{1-\gamma}\right)^{n}}.
\]
\end{prop}
\begin{prop}[Convergence of $k_{n}\left(t\right)$ to its ideal counterpart]
\label{prop:convergence kn to ideal model}{[}Proposition 11 of \cite{Articulo Boussinesq}{]}
Let $n\in\mathbb{N}$ with $n\ge2$ and $\beta,\beta'\in\left(0,1\right)$.
Moreover, let $k_{n}\left(t\right)$ be as defined in equation \eqref{eq:Boussinesq an(t) bn(t)}
and $\overline{k}_{n}\left(t\right)$ be the ideal model of $k_{n}\left(t\right)$.
Then, provided that $C$ is big enough (let us say $C\ge\Upsilon\left(\beta,\beta',\delta\right)$),
\[
\left|k_{n}\left(t\right)-\overline{k}_{n}\left(t\right)\right|\lesssim_{\delta}C^{-\frac{1}{2}\beta\beta'\delta\gamma\left(\frac{1}{1-\gamma}\right)^{n-1}}\quad\forall t\in\left[t_{n},1\right].
\]
\end{prop}
\begin{prop}[Layer center never leaves zone where cut-off functions are one]
\label{prop:layer center never leaves cutoff area}{[}Proposition
12 of \cite{Articulo Boussinesq}{]} Let $n\in\mathbb{N}$. Then,
provided that $C$ is big enough (let us say $C\ge\Upsilon\left(\delta\right)$),
$\forall m\in\mathbb{N}$ with $m\le n$, we have 
\[
\left|\phi^{\left(n\right)}_{1}\left(t,0\right)-\phi^{\left(m\right)}_{1}\left(t,0\right)\right|\le\frac{8\pi}{a_{m}\left(t\right)}\quad\forall t\in\left[t_{n},1\right].
\]
\end{prop}
\begin{lem}[Simple bound for superexponential sum]
\label{lem:estimate sum superexponential}{[}Lemma 5 of \cite{Articulo Boussinesq}{]}
Let $\gamma\in\left(\frac{1}{2},1\right)$, $\delta>0$, $C>2$ and
$n\in\mathbb{N}$. We have
\[
\sum^{n-1}_{m=1}C^{\delta\left(\frac{1}{1-\gamma}\right)^{m}}\lesssim_{\delta}C^{\delta\left(\frac{1}{1-\gamma}\right)^{n-1}}.
\]
\end{lem}
\begin{lem}[Lower bound of denominator of density]
\label{lem:estimate of integral hn bn}{[}Lemma 13 of \cite{Articulo Boussinesq}{]}
Let $n\in\mathbb{N}$ with $n\ge2$. For every $\mu>0$, provided
that $C$ is big enough (let us say $C\ge\Upsilon\left(\delta,\mu\right)$),
we have
\[
\int^{1}_{t_{n}}h^{\left(n\right)}\left(s\right)b_{n}\left(s\right)\mathrm{d}s\gtrsim_{\mu}\frac{1}{Y}C^{\left(1+k_{\max}-\mu-\delta\left(1-\gamma\right)\right)\left(\frac{1}{1-\gamma}\right)^{n}}.
\]
\end{lem}
\begin{cor}[Relation between $a_{n}\left(t\right),b_{n}\left(t\right)$ and the
ideal $\overline{k}_{n}\left(t\right)$]
\label{cor:bound for an(t). bn(t)}{[}Corollary 2 of \cite{Articulo Boussinesq}{]}
Let $n\in\mathbb{N}$ with $n\ge2$ and $\mu>0$. Provided that $C$
is big enough (let us say $C\ge\Upsilon\left(\delta,\mu\right)$),
then
\[
\begin{aligned} & C^{\left(1-\overline{k}_{n}\left(t\right)-\mu\right)\left(\frac{1}{1-\gamma}\right)^{n}}\le a_{n}\left(t\right)\leq C^{\left(1-\overline{k}_{n}\left(t\right)+\mu\right)\left(\frac{1}{1-\gamma}\right)^{n}}\le C^{\left(1+\mu\right)\left(\frac{1}{1-\gamma}\right)^{n}},\\
 & C^{\left(1-\mu\right)\left(\frac{1}{1-\gamma}\right)^{n}}\le C^{\left(1+\overline{k}_{n}\left(t\right)-\mu\right)\left(\frac{1}{1-\gamma}\right)^{n}}\le b_{n}\left(t\right)\leq C^{\left(1+\overline{k}_{n}\left(t\right)+\mu\right)\left(\frac{1}{1-\gamma}\right)^{n}}.
\end{aligned}
\]
Furthermore,
\[
\begin{aligned}C^{\left(1+\overline{k}_{n}\left(t\right)-\mu\right)\left(\frac{1}{1-\gamma}\right)^{n}}\le\max\left\{ a_{n}\left(t\right),b_{n}\left(t\right)\right\}  & \leq C^{\left(1+\overline{k}_{n}\left(t\right)+\mu\right)\left(\frac{1}{1-\gamma}\right)^{n}}\le b_{n}\left(t\right)C^{2\mu\left(\frac{1}{1-\gamma}\right)^{n}},\\
\min\left\{ a_{n}\left(t\right),b_{n}\left(t\right)\right\}  & \ge C^{\left(1-\overline{k}_{n}\left(t\right)-\mu\right)\left(\frac{1}{1-\gamma}\right)^{n}}\ge a_{n}\left(t\right)C^{-2\mu\left(\frac{1}{1-\gamma}\right)^{n}}.
\end{aligned}
\]
\end{cor}
\begin{prop}[Decomposition of the velocity]
\label{prop:form of the velocity}{[}Proposition 5 of \cite{Articulo Boussinesq}{]}
There are functions $V^{\left(n\right)},W^{\left(n\right)}:\left[0,1\right]\times\mathbb{R}^{2}\to\mathbb{R}^{2}$
such that
\[
\begin{aligned}\widetilde{u^{\left(n\right)}}^{n}\left(t,x\right) & =\widetilde{V^{\left(n\right)}}^{n}\left(t,x\right)+\lambda_{n}\widetilde{W^{\left(n\right)}}^{n}\left(t,x\right),\\
\widetilde{V^{\left(n\right)}}^{n}\left(t,x\right) & =B_{n}\left(t\right)\varphi\left(\lambda_{n}x_{1}\right)\varphi\left(\lambda_{n}x_{2}\right)\left(\begin{matrix}b_{n}\left(t\right)\sin\left(x_{1}\right)\cos\left(x_{2}\right)\\
-a_{n}\left(t\right)\cos\left(x_{1}\right)\sin\left(x_{2}\right)
\end{matrix}\right).
\end{aligned}
\]
Furthermore, given $\alpha\in\left(0,1\right)$,
\[
\begin{aligned}\left|\left|V^{\left(n\right)}_{1}\left(t,\cdot\right)\right|\right|_{L^{\infty}\left(\mathbb{R}^{2}\right)} & \le B_{n}\left(t\right)b_{n}\left(t\right), & \left|\left|V^{\left(n\right)}_{2}\left(t,\cdot\right)\right|\right|_{L^{\infty}\left(\mathbb{R}^{2}\right)} & \le B_{n}\left(t\right)a_{n}\left(t\right),\\
\left|\left|W^{\left(n\right)}_{1}\left(t,\cdot\right)\right|\right|_{L^{\infty}\left(\mathbb{R}^{2}\right)} & \lesssim_{\varphi}B_{n}\left(t\right)b_{n}\left(t\right), & \left|\left|W^{\left(n\right)}_{2}\left(t,\cdot\right)\right|\right|_{L^{\infty}\left(\mathbb{R}^{2}\right)} & \lesssim_{\varphi}B_{n}\left(t\right)a_{n}\left(t\right),\\
\left|\left|V^{\left(n\right)}_{1}\left(t,\cdot\right)\right|\right|_{\dot{C}^{\alpha}\left(\mathbb{R}^{2}\right)} & \lesssim_{\varphi}B_{n}\left(t\right)b_{n}\left(t\right)\max\left\{ a_{n}\left(t\right)^{\alpha},b_{n}\left(t\right)^{\alpha}\right\} , & \left|\left|V^{\left(n\right)}_{2}\left(t,\cdot\right)\right|\right|_{\dot{C}^{\alpha}\left(\mathbb{R}^{2}\right)} & \lesssim_{\varphi}B_{n}\left(t\right)a_{n}\left(t\right)\max\left\{ a_{n}\left(t\right)^{\alpha},b_{n}\left(t\right)^{\alpha}\right\} ,\\
\left|\left|W^{\left(n\right)}_{1}\left(t,\cdot\right)\right|\right|_{\dot{C}^{\alpha}\left(\mathbb{R}^{2}\right)} & \lesssim_{\varphi}B_{n}\left(t\right)b_{n}\left(t\right)\max\left\{ a_{n}\left(t\right)^{\alpha},b_{n}\left(t\right)^{\alpha}\right\} , & \left|\left|W^{\left(n\right)}_{2}\left(t,\cdot\right)\right|\right|_{\dot{C}^{\alpha}\left(\mathbb{R}^{2}\right)} & \lesssim_{\varphi}B_{n}\left(t\right)a_{n}\left(t\right)\max\left\{ a_{n}\left(t\right)^{\alpha},b_{n}\left(t\right)^{\alpha}\right\} .
\end{aligned}
\]
\end{prop}
\begin{prop}[Bounds for the Taylor development of the transport term]
\label{prop:bounds for Taylor development of transport}{[}Proposition
17 of \cite{Articulo Boussinesq}{]} Let $n\in\mathbb{N}$ with $n\ge2$
and define
\[
\begin{aligned}\widetilde{W^{\left(n\right)}}^{n}\left(t,x\right) & \coloneqq\widetilde{U^{\left(n-1\right)}}^{n}\left(t,x\right)-\widetilde{U^{\left(n-1\right)}}^{n}\left(t,0\right)-\mathrm{J}\widetilde{U^{\left(n-1\right)}}^{n}\left(t,0\right)\cdot\left(\begin{matrix}x_{1}\\
x_{2}
\end{matrix}\right),\\
D^{\left(n\right)}\left(t\right) & \coloneqq\left(\phi^{\left(n\right)}\right)^{-1}\left(t,\left[-\frac{16\pi}{\lambda_{n}},\frac{16\pi}{\lambda_{n}}\right]^{2}\right).
\end{aligned}
\]
Then, as long as $C$ is big enough (let us say $C\ge\Upsilon\left(\delta,\mu\right)$),
\[
\begin{aligned}\left|\left|W^{\left(n\right)}\left(t,\cdot\right)\right|\right|_{L^{\infty}\left(D^{\left(n\right)}\left(t\right);\mathbb{R}^{2}\right)} & \lesssim_{\delta,\varphi}YC^{\left[-2\left(1-\Lambda-k_{\max}\right)+2\mu+\left(1+\delta\right)\left(1-\gamma\right)\right]\left(\frac{1}{1-\gamma}\right)^{n}},\\
\left|\left|W^{\left(n\right)}\left(t,\cdot\right)\right|\right|_{\dot{C}^{\alpha}\left(D^{\left(n\right)}\left(t\right);\mathbb{R}^{2}\right)} & \lesssim_{\delta,\varphi}YC^{\left[-2\left(1-\Lambda-k_{\max}\right)+\max\left\{ \alpha\left(1-k_{\max}\right)-\Lambda,0\right\} +2\mu+\left(2+\delta\right)\left(1-\gamma\right)\right]\left(\frac{1}{1-\gamma}\right)^{n}}.
\end{aligned}
\]
\end{prop}
\begin{lem}[Monotone convergence of $\overline{k}_{n}\left(t\right)$]
\label{lem:monotone convergence kn}{[}Lemma 4 of \cite{Articulo Boussinesq}{]}
$\forall n\in\mathbb{N}$ and $\forall\hat{\hat{t}}\in\left[0,1\right]$,
\[
\overline{k}_{n+1}\left(t_{n+1}+\left(1-t_{n+1}\right)\hat{\hat{t}}\right)\le\overline{k}_{n}\left(t_{n}+\left(1-t_{n}\right)\hat{\hat{t}}\right).
\]
\end{lem}
\begin{lem}[Simple bound for $\mathrm{arccosh}$]
\label{lem:arccosh by ln}{[}Lemma 6 of \cite{Articulo Boussinesq}{]}
We have
\[
\mathrm{arccosh}\left(x\right)\le\ln\left(2x\right)\quad\forall x\in\left[1,\infty\right).
\]
\end{lem}
\begin{lem}[Exponentially increasing times superexponentially decreasing is superexponentially
decreasing]
\label{lem:exponential superexponential bound}{[}Lemma 7 of \cite{Articulo Boussinesq}{]}
Let $n\in\mathbb{N}$, $a\ge1$, $b>0$ and $\beta\in\left(0,1\right)$.
Then,
\[
a^{n+1}C^{-ba^{n}}\le\frac{4\mathrm{e}^{-2}}{\left(1-\beta\right)^{2}b^{2}\ln\left(C\right)^{2}}C^{-\beta ba^{n}}.
\]
\end{lem}
Furthermore, from \cite{Articulo Boussinesq} we also record the bound
\begin{equation}
\left|\frac{\mathrm{d}B_{n}}{\mathrm{d}t}\left(t\right)\right|\lesssim_{\delta,\mu}\frac{2M_{n}}{a_{n}\left(t\right)^{2}+b_{n}\left(t\right)^{2}}YC^{\left(2\mu+\delta\left(1-\gamma\right)\right)\left(\frac{1}{1-\gamma}\right)^{n}},\label{eq:Boussinesq dBndt}
\end{equation}
the exact formula
\begin{equation}
\overline{k}_{n}\left(t_{n}+\left(1-t_{n}\right)\hat{\hat{t}}\right)=k_{\max}+\frac{1}{\ln\left(C\right)\left(\frac{1}{1-\gamma}\right)^{n}}\ln\left(\frac{1}{\cosh\left(\mathrm{arccosh}\left(C^{k_{\max}\left(\frac{1}{1-\gamma}\right)^{n}}\right)\left(1-2\hat{\hat{t}}\right)\right)}\right),\label{eq:Boussinesq kn exact}
\end{equation}
and the identity resulting from the optimal choice of parameters:
\begin{equation}
1-\Lambda-k_{\max}=\frac{2}{3},\quad-1+2\Lambda+3k_{\max}=-\Lambda.\label{eq:Boussinesq identity Lambda kmax}
\end{equation}
Bear in mind that there is a typographical error in the published
version of \cite{Articulo Boussinesq} regarding equation \eqref{eq:Boussinesq identity Lambda kmax}.
The version presented here is correct.

\subsection{Existence of blow-up point}

Although the existence of a blow-up point in the construction given
in \cite{Articulo Boussinesq} is suggested, it is never explicitly
computed or proven. In this next Lemma, we formally solve this issue.
\begin{lem}
\label{lem:blow-up point}The sequence $\left(\phi^{\left(n\right)}_{1}\left(1,0\right)\right)_{n\in\mathbb{N}}$
is convergent. We denote its limit by $\phi^{\left(\infty\right)}_{1}\left(1,0\right)$.
Then,
\[
\left|\phi^{\left(\infty\right)}_{1}\left(1,0\right)-\phi^{\left(n\right)}_{1}\left(1,0\right)\right|\leq8\pi C^{-\left(\frac{1}{1-\gamma}\right)^{n}}.
\]
\end{lem}
\begin{proof}
By Proposition \ref{prop:layer center never leaves cutoff area},
$\forall\left(n,m\right)\in\mathbb{N}^{2}$ with $n\ge m$,
\[
\left|\phi^{\left(n\right)}_{1}\left(t,0\right)-\phi^{\left(m\right)}_{1}\left(t,0\right)\right|\leq\frac{8\pi}{a_{m}\left(t\right)}\quad\forall t\in\left[t_{n},1\right].
\]
Evaluating at $t=1$ and using equation \eqref{eq:Boussinesq an(t) bn(t)},
we deduce that
\[
\left|\phi^{\left(n\right)}_{1}\left(1,0\right)-\phi^{\left(m\right)}_{1}\left(1,0\right)\right|\leq8\pi C^{-\left(\frac{1}{1-\gamma}\right)^{m}}.
\]
Thus, $\left(\phi^{\left(n\right)}_{1}\left(1,0\right)\right)_{n\in\mathbb{N}}$
is a Cauchy sequence in $\mathbb{R}$ and, therefore, convergent.
Furthermore, taking limits as $n\to\infty$ and relabeling $m\leftrightarrow n$
, we deduce the bound of the statement.
\end{proof}

\begin{choice}
\label{choice:blow-up origin}From here onwards, we \textbf{will assume}
without loss of generality that the \textbf{singularity} of the Boussinesq
solution constructed in \cite{Articulo Boussinesq} \textbf{occurs
at the origin}.
\end{choice}
\begin{rem*}
Notice that adapting the solution given in \cite{Articulo Boussinesq}
so that it blows up at the origin only requires a translation in the
$x_{1}$-axis. Consequently, all symmetries in the $x_{2}$-axis are
preserved.
\end{rem*}

\subsection{Bounds for the good terms}

In this section, we will prove that $\rho,\frac{\partial u}{\partial t}+u\cdot\nabla u\in C^{0}_{t}C^{1-}_{x}$
and that $\frac{\partial\rho}{\partial x_{1}}\in C^{0}_{t}C^{\varepsilon}_{x}$
for some $\varepsilon>0$. These are the ``good'' terms, because
they retain some subcritical regularity at the blow-up time.
\begin{prop}
\label{prop:bounded density}Let $\left(\rho,\omega,f_{\rho},f_{\omega}\right)$
be the solution of the forced Boussinesq system presented in \cite{Articulo Boussinesq}.
Then, $\forall\alpha\in\left(0,1\right)$, as long as $\delta\le\Upsilon_{1}\left(\alpha\right)$
and $\mu\leq\Upsilon_{2}\left(\alpha\right)$, we have $\rho\in C^{0}_{t}C^{\alpha}_{x}\left(\left[0,1\right]\times\mathbb{R}^{2}\right)$.
\end{prop}
\begin{proof}
Employing equations \eqref{eq:Boussinesq psi rho} and \eqref{eq:Boussinesq Bn(t)},
we deduce that
\[
\widetilde{\rho^{\left(n\right)}}^{n}\left(t,x\right)=-2M_{n}\frac{h^{\left(n\right)}\left(t\right)}{\int^{1}_{t_{n}}h^{\left(n\right)}\left(s\right)b_{n}\left(s\right)\mathrm{d}s}\varphi\left(\lambda_{n}x_{1}\right)\varphi\left(\lambda_{n}x_{2}\right)\sin\left(x_{1}\right)\cos\left(x_{2}\right).
\]
By Lemma \ref{lem:estimate of integral hn bn}, equations \eqref{eq:Boussinesq lambda_n}, \eqref{eq:property Calpha multiplication}, and \eqref{eq:property Calpha composition}
and the fact that $h^{\left(n\right)}\left(t\right)\leq1$, we can
bound
\[
\begin{aligned}\left|\left|\widetilde{\rho^{\left(n\right)}}^{n}\left(t,\cdot\right)\right|\right|_{L^{\infty}\left(\mathbb{R}^{2}\right)} & \lesssim_{\varphi,\mu}\frac{2M_{n}}{\frac{1}{Y}C^{\left(1+k_{\max}-\mu-\delta\left(1-\gamma\right)\right)\left(\frac{1}{1-\gamma}\right)^{n}}},\\
\left|\left|\widetilde{\rho^{\left(n\right)}}^{n}\left(t,\cdot\right)\right|\right|_{\dot{C}^{\alpha}\left(\mathbb{R}^{2}\right)} & \lesssim_{\varphi,\mu}\frac{2M_{n}}{\frac{1}{Y}C^{\left(1+k_{\max}-\mu-\delta\left(1-\gamma\right)\right)\left(\frac{1}{1-\gamma}\right)^{n}}}.
\end{aligned}
\]
Employing equation \eqref{eq:Boussinesq Mn} and bounding $1-\gamma\leq1$,
we arrive at
\[
\begin{aligned}\left|\left|\widetilde{\rho^{\left(n\right)}}^{n}\left(t,\cdot\right)\right|\right|_{L^{\infty}\left(\mathbb{R}^{2}\right)} & \lesssim_{\varphi,\mu}Y^{2}\frac{C^{\delta\left(\frac{1}{1-\gamma}\right)^{n}}}{C^{\left(1+k_{\max}-\mu-\delta\left(1-\gamma\right)\right)\left(\frac{1}{1-\gamma}\right)^{n}}}\leq Y^{2}C^{\left(-1-k_{\max}+2\delta+\mu\right)\left(\frac{1}{1-\gamma}\right)^{n}},\\
\left|\left|\widetilde{\rho^{\left(n\right)}}^{n}\left(t,\cdot\right)\right|\right|_{\dot{C}^{\alpha}\left(\mathbb{R}^{2}\right)} & \lesssim_{\varphi,\mu}Y^{2}\frac{C^{\delta\left(\frac{1}{1-\gamma}\right)^{n}}}{C^{\left(1+k_{\max}-\mu-\delta\left(1-\gamma\right)\right)\left(\frac{1}{1-\gamma}\right)^{n}}}\leq Y^{2}C^{\left(-1-k_{\max}+2\delta+\mu\right)\left(\frac{1}{1-\gamma}\right)^{n}}.
\end{aligned}
\]
On the one hand, since $\left|\left|\cdot\right|\right|_{L^{\infty}\left(\mathbb{R}^{2}\right)}$
is invariant under diffeomorphisms of the domain, we obtain that
\[
\left|\left|\rho^{\left(n\right)}\left(t,\cdot\right)\right|\right|_{L^{\infty}\left(\mathbb{R}^{2}\right)}\lesssim_{\varphi,\mu}Y^{2}C^{\left(-1-k_{\max}+2\delta+\mu\right)\left(\frac{1}{1-\gamma}\right)^{n}}.
\]
On the other hand, employing equations \eqref{eq:property Calpha composition} and \eqref{eq:Boussinesq jacobian inverse}
and Corollary \ref{cor:bound for an(t). bn(t)}, we deduce that
\[
\begin{aligned}\left|\left|\rho^{\left(n\right)}\left(t,\cdot\right)\right|\right|_{\dot{C}^{\alpha}\left(\mathbb{R}^{2}\right)} & \lesssim\left|\left|\left(\phi^{\left(n\right)}\right)^{-1}\left(t,\cdot\right)\right|\right|^{\alpha}_{\dot{C}^{1}\left(\mathbb{R}^{2};\mathbb{R}^{2}\right)}\left|\left|\widetilde{\rho^{\left(n\right)}}^{n}\left(t,\cdot\right)\right|\right|_{\dot{C}^{\alpha}\left(\mathbb{R}^{2}\right)}\lesssim_{\varphi,\mu}\\
 & \lesssim_{\varphi,\mu}Y^{2}C^{\left(-1-k_{\max}+2\delta+\mu+\alpha\left(1+\overline{k}_{n}\left(t\right)+\mu\right)\right)\left(\frac{1}{1-\gamma}\right)^{n}}.
\end{aligned}
\]
Bounding $\overline{k}_{n}\left(t\right)\le k_{\max}$ (see equation
\eqref{eq:Boussinesq def kmax}) and $\alpha\le1$ where it is multiplied
by $\mu$, we arrive at
\[
\left|\left|\rho^{\left(n\right)}\left(t,\cdot\right)\right|\right|_{\dot{C}^{\alpha}\left(\mathbb{R}^{2}\right)}\lesssim_{\varphi,\mu}Y^{2}C^{\left(-\left(1-\alpha\right)\left(1+k_{\max}\right)+2\delta+2\mu\right)\left(\frac{1}{1-\gamma}\right)^{n}}.
\]
Thus, clearly,
\[
\left|\left|\rho^{\left(n\right)}\left(t,\cdot\right)\right|\right|_{C^{\alpha}\left(\mathbb{R}^{2}\right)}\lesssim_{\varphi,\mu}Y^{2}C^{\left(-\left(1-\alpha\right)\left(1+k_{\max}\right)+2\delta+2\mu\right)\left(\frac{1}{1-\gamma}\right)^{n}}.
\]

Since $C^{0}_{t}C^{\alpha}_{x}\left(\left[0,1\right]\times\mathbb{R}^{2}\right)$
is complete and $\rho=\lim_{n\to\infty}\sum^{n}_{m=1}\rho^{\left(m\right)}$,
to prove the result, it suffices to see that $P^{\left(n\right)}=\sum^{n}_{m=1}\rho^{\left(m\right)}$
is a Cauchy sequence in $C^{0}_{t}C^{\alpha}_{x}\left(\left[0,1\right]\times\mathbb{R}^{2}\right)$.
First of all, notice that, $\forall n\in\mathbb{N}$, $P^{\left(n\right)}$
is nothing but a finite sum of compositions of smooth functions in
space and time. Consequently, $P^{\left(n\right)}\in C^{0}_{t}C^{\alpha}_{x}\left(\left[0,1\right]\times\mathbb{R}^{2}\right)$
$\forall n\in\mathbb{N}$. Moreover, by the telescopic nature of $P^{\left(n\right)}$,
$\forall l,m\in\mathbb{N}$ with $l\ge m$, we have
\[
\begin{aligned}\left|\left|P^{\left(l\right)}-P^{\left(m\right)}\right|\right|_{C^{0}_{t}C^{\alpha}_{x}\left(\left[0,1\right]\times\mathbb{R}^{2}\right)} & =\left|\left|\sum^{l}_{n=m+1}\rho^{\left(n\right)}\right|\right|_{C^{0}_{t}C^{\alpha}_{x}\left(\left[0,1\right]\times\mathbb{R}^{2}\right)}\leq\sum^{\infty}_{n=m+1}\sup_{t\in\left[0,1\right]}\left|\left|\rho^{\left(n\right)}\left(t,\cdot\right)\right|\right|_{C^{\alpha}\left(\mathbb{R}^{2}\right)}\lesssim_{\varphi,\mu}\\
 & \lesssim_{\varphi,\mu}\sum^{\infty}_{n=m+1}Y^{2}C^{\left(-\left(1-\alpha\right)\left(1+k_{\max}\right)+2\delta+2\mu\right)\left(\frac{1}{1-\gamma}\right)^{n}}.
\end{aligned}
\]
Clearly, as long as $\delta$ and $\mu$ are small enough in relation
to $1-\alpha$, the series above is convergent. In this way, in the
bound above, we have the tail of a convergent series. Consequently,
we must have $\left|\left|P^{\left(l\right)}-P^{\left(m\right)}\right|\right|_{C^{0}_{t}C^{\alpha}_{x}\left(\left[0,1\right]\times\mathbb{R}^{2}\right)}\xrightarrow[m\to\infty]{}0$
and, therefore, $\left(P^{\left(n\right)}\right)_{n\in\mathbb{N}}$
is a Cauchy sequence in $C^{0}_{t}C^{\alpha}_{x}\left(\left[0,1\right]\times\mathbb{R}^{2}\right)$.
\end{proof}

\begin{prop}
\label{prop:evenness rho}Let $\left(\rho,\omega,f_{\rho},f_{\omega}\right)$
be the solution of the forced Boussinesq system presented in \cite{Articulo Boussinesq}.
$\rho\left(t,x\right)$ is even in $x_{2}$.
\end{prop}
\begin{proof}
Taking a look at equation \eqref{eq:Boussinesq psi rho}, we see that
$\widetilde{\rho^{\left(n\right)}}^{n}\left(t,x\right)$ is even in
$x_{2}$. Since the change of variables $\phi^{\left(n\right)}_{2}\left(t,x\right)$
is linear in $x_{2}$ (see equation \eqref{eq:Boussinesq phi}), $\rho\left(t,x\right)$
must also be even in $x_{2}$. As the sum of even functions is even
and the limit of even functions is even, we conclude that $\rho\left(t,x\right)$
is even in $x_{2}$. (Notice that the limit exists thanks to Proposition
\ref{prop:bounded density}).
\end{proof}

\begin{prop}
\label{prop:bounded drhodx1}Let $\left(\rho,\omega,f_{\rho},f_{\omega}\right)$
be the solution of the forced Boussinesq system presented in \cite{Articulo Boussinesq}.
Then, $\forall\alpha\in\left(0,\frac{\alpha_{*}}{1+\alpha_{*}}\right)$,
as long as $\delta\le\Upsilon_{1}\left(\alpha\right)$ and $\mu\leq\Upsilon_{2}\left(\alpha\right)$,
we have $\frac{\partial\rho}{\partial x_{1}}\in C^{0}_{t}C^{\alpha}_{x}\left(\left[0,1\right]\times\mathbb{R}^{2}\right)$.
\end{prop}
\begin{proof}
Employing equations \eqref{eq:Boussinesq psi rho} and \eqref{eq:Boussinesq Bn(t)},
we deduce that
\[
\widetilde{\rho^{\left(n\right)}}^{n}\left(t,x\right)=-2M_{n}\frac{h^{\left(n\right)}\left(t\right)}{\int^{1}_{t_{n}}h^{\left(n\right)}\left(s\right)b_{n}\left(s\right)\mathrm{d}s}\varphi\left(\lambda_{n}x_{1}\right)\varphi\left(\lambda_{n}x_{2}\right)\sin\left(x_{1}\right)\cos\left(x_{2}\right).
\]
Differentiating with respect to $x_{1}$, we get
\begin{equation}
\frac{\partial\widetilde{\rho^{\left(n\right)}}^{n}}{\partial x_{1}}\left(t,x\right)=-2M_{n}\frac{h^{\left(n\right)}\left(t\right)}{\int^{1}_{t_{n}}h^{\left(n\right)}\left(s\right)b_{n}\left(s\right)\mathrm{d}s}\varphi\left(\lambda_{n}x_{2}\right)\cos\left(x_{2}\right)\left[\varphi\left(\lambda_{n}x_{1}\right)\cos\left(x_{1}\right)+\lambda_{n}\varphi'\left(\lambda_{n}x_{1}\right)\sin\left(x_{1}\right)\right].\label{eq:drhondx1 moving}
\end{equation}
Besides, the chain rule along with equation \eqref{eq:Boussinesq jacobian inverse}
shows that
\[
\frac{\partial\rho^{\left(n\right)}}{\partial x_{1}}\left(t,x\right)=\frac{\partial}{\partial x_{1}}\left(\widetilde{\rho^{\left(n\right)}}^{n}\left(t,\left(\phi^{\left(n\right)}\right)^{-1}\left(t,x\right)\right)\right)=a_{n}\left(t\right)\frac{\partial\widetilde{\rho^{\left(n\right)}}^{n}}{\partial x_{1}}\left(t,\left(\phi^{\left(n\right)}\right)^{-1}\left(t,x\right)\right).
\]
Employing the fact that $\left|\left|\cdot\right|\right|_{L^{\infty}\left(\mathbb{R}^{2}\right)}$
remains invariant under diffeomorphisms of the domain and equation
\eqref{eq:property Calpha composition}, we deduce that
\[
\begin{aligned}\left|\left|\frac{\partial\rho^{\left(n\right)}}{\partial x_{1}}\left(t,\cdot\right)\right|\right|_{L^{\infty}\left(\mathbb{R}^{2}\right)} & =a_{n}\left(t\right)\left|\left|\frac{\partial\widetilde{\rho^{\left(n\right)}}^{n}}{\partial x_{1}}\left(t,\cdot\right)\right|\right|_{L^{\infty}\left(\mathbb{R}^{2}\right)},\\
\left|\left|\frac{\partial\rho^{\left(n\right)}}{\partial x_{1}}\left(t,\cdot\right)\right|\right|_{\dot{C}^{\alpha}\left(\mathbb{R}^{2}\right)} & \lesssim a_{n}\left(t\right)\left|\left|\left(\phi^{\left(n\right)}\right)^{-1}\left(t,\cdot\right)\right|\right|^{\alpha}_{\dot{C}^{1}\left(\mathbb{R}^{2};\mathbb{R}^{2}\right)}\left|\left|\frac{\partial\widetilde{\rho^{\left(n\right)}}^{n}}{\partial x_{1}}\left(t,\cdot\right)\right|\right|_{\dot{C}^{\alpha}\left(\mathbb{R}^{2}\right)}.
\end{aligned}
\]
Looking at equations \eqref{eq:drhondx1 moving}, \eqref{eq:property Calpha multiplication}, \eqref{eq:property Calpha composition}, \eqref{eq:Boussinesq lambda_n}, \eqref{eq:Boussinesq jacobian inverse}, and \eqref{eq:Boussinesq Mn},
along with Lemma \ref{lem:estimate of integral hn bn}, Corollary
\ref{cor:bound for an(t). bn(t)} and the fact that $h^{\left(n\right)}\left(t\right)\leq1$,
we can bound
\[
\begin{aligned}\left|\left|\frac{\partial\rho^{\left(n\right)}}{\partial x_{1}}\left(t,\cdot\right)\right|\right|_{L^{\infty}\left(\mathbb{R}^{2}\right)} & \lesssim_{\varphi,\mu,Y}C^{\left(1+\mu\right)\left(\frac{1}{1-\gamma}\right)^{n}}\frac{C^{\delta\left(\frac{1}{1-\gamma}\right)^{n}}}{C^{\left(1+k_{\max}-\mu-\delta\left(1-\gamma\right)\right)\left(\frac{1}{1-\gamma}\right)^{n}}},\\
\left|\left|\frac{\partial\rho^{\left(n\right)}}{\partial x_{1}}\left(t,\cdot\right)\right|\right|_{\dot{C}^{\alpha}\left(\mathbb{R}^{2}\right)} & \lesssim_{\varphi,\mu,Y}C^{\left(1+\mu\right)\left(\frac{1}{1-\gamma}\right)^{n}}C^{\alpha\left(1+\overline{k}_{n}\left(t\right)+\mu\right)\left(\frac{1}{1-\gamma}\right)^{n}}\frac{C^{\delta\left(\frac{1}{1-\gamma}\right)^{n}}}{C^{\left(1+k_{\max}-\mu-\delta\left(1-\gamma\right)\right)\left(\frac{1}{1-\gamma}\right)^{n}}}.
\end{aligned}
\]
Bounding $\overline{k}_{n}\left(t\right)\le k_{\max}$ (see \eqref{eq:Boussinesq def kmax}),
$1-\gamma\leq1$ and $\alpha\leq1$ where it is multiplied by $\mu$
leads us to
\begin{equation}
\begin{aligned}\left|\left|\frac{\partial\rho^{\left(n\right)}}{\partial x_{1}}\left(t,\cdot\right)\right|\right|_{L^{\infty}\left(\mathbb{R}^{2}\right)} & \lesssim_{\varphi,\mu,Y}C^{\left(-k_{\max}+2\delta+2\mu\right)\left(\frac{1}{1-\gamma}\right)^{n}},\\
\left|\left|\frac{\partial\rho^{\left(n\right)}}{\partial x_{1}}\left(t,\cdot\right)\right|\right|_{\dot{C}^{\alpha}\left(\mathbb{R}^{2}\right)} & \lesssim_{\varphi,\mu,Y}C^{\left(-k_{\max}+\alpha\left(1+k_{\max}\right)+2\delta+3\mu\right)\left(\frac{1}{1-\gamma}\right)^{n}}.
\end{aligned}
\label{eq:bound drhodx1 n}
\end{equation}
Thus, clearly,
\[
\left|\left|\frac{\partial\rho^{\left(n\right)}}{\partial x_{1}}\left(t,\cdot\right)\right|\right|_{C^{\alpha}\left(\mathbb{R}^{2}\right)}\lesssim_{\varphi,\mu,Y}C^{\left(-k_{\max}+\alpha\left(1+k_{\max}\right)+2\delta+3\mu\right)\left(\frac{1}{1-\gamma}\right)^{n}}.
\]

Next in line is to prove that $S^{\left(n\right)}\coloneqq\sum^{n}_{m=1}\frac{\partial\rho^{\left(m\right)}}{\partial x_{1}}$
has a well-defined limit in $C^{0}_{t}C^{\alpha}_{x}\left(\left[0,1\right]\times\mathbb{R}^{2}\right)$
as $n\to\infty$. Since $C^{0}_{t}C^{\alpha}_{x}\left(\left[0,1\right]\times\mathbb{R}^{2}\right)$
is complete, it suffices to see that $S^{\left(n\right)}$ is a Cauchy
sequence in $C^{0}_{t}C^{\alpha}_{x}\left(\left[0,1\right]\times\mathbb{R}^{2}\right)$.
First of all, notice that, $\forall n\in\mathbb{N}$, $S^{\left(n\right)}$
is nothing but a finite sum of compositions of smooth functions in
space and time. Consequently, $S^{\left(n\right)}\in C^{0}_{t}C^{\alpha}_{x}\left(\left[0,1\right]\times\mathbb{R}^{2}\right)$
$\forall n\in\mathbb{N}$. Moreover, by the telescopic nature of $S^{\left(n\right)}$,
$\forall l,m\in\mathbb{N}$ with $l\ge m$, we have
\begin{equation}
\begin{aligned}\left|\left|S^{\left(l\right)}-S^{\left(m\right)}\right|\right|_{C^{0}_{t}C^{\alpha}_{x}\left(\left[0,1\right]\times\mathbb{R}^{2}\right)} & =\left|\left|\sum^{l}_{n=m+1}\frac{\partial\rho^{\left(n\right)}}{\partial x_{1}}\right|\right|_{C^{0}_{t}C^{\alpha}_{x}\left(\left[0,1\right]\times\mathbb{R}^{2}\right)}\leq\sum^{\infty}_{n=m+1}\sup_{t\in\left[0,1\right]}\left|\left|\frac{\partial\rho^{\left(n\right)}}{\partial x_{1}}\left(t,\cdot\right)\right|\right|_{C^{\alpha}\left(\mathbb{R}^{2}\right)}\leq\\
 & \lesssim_{\varphi,\mu,Y}\sum^{\infty}_{n=m+1}C^{\left(-k_{\max}+\alpha\left(1+k_{\max}\right)+2\delta+3\mu\right)\left(\frac{1}{1-\gamma}\right)^{n}}.
\end{aligned}
\label{eq:cauchy S}
\end{equation}
Clearly, as long as $\alpha<\frac{k_{\max}}{1+k_{\max}}$ and $\delta$
and $\mu$ are small enough in relation to $k_{\max}-\alpha\left(1+k_{\max}\right)$,
the series above is convergent. In this way, in the bound above, we
have the tail of a convergent series. Consequently, $\left|\left|S^{\left(l\right)}-S^{\left(m\right)}\right|\right|_{C^{0}_{t}C^{\alpha}_{x}\left(\left[0,1\right]\times\mathbb{R}^{2}\right)}\xrightarrow[m\to\infty]{}0$
and, therefore, $\left(S^{\left(n\right)}\right)_{n\in\mathbb{N}}$
is a Cauchy sequence in $C^{0}_{t}C^{\alpha}_{x}\left(\left[0,1\right]\times\mathbb{R}^{2}\right)$.
Since by point \ref{enu:The-time-scales} of subsection \ref{subsec:summary Boussinesq},
we have $k_{\max}=\alpha_{*}$, the requirement for $\alpha$ is $\alpha<\frac{\alpha_{*}}{1+\alpha_{*}}$.

Finally, we need to prove that $\frac{\partial\rho}{\partial x_{1}}=\lim_{n\to\infty}S^{\left(n\right)}$.
We will do this by a standard calculus argument. Let $x_{1},y_{1},z\in\mathbb{R}$.
Then, the Fundamental Theorem of Calculus guarantees that
\[
\left(\sum^{n}_{m=1}\rho^{\left(m\right)}\right)\left(t,\left(y_{1},z\right)\right)-\left(\sum^{n}_{m=1}\rho^{\left(m\right)}\right)\left(t,\left(x_{1},z\right)\right)=\int^{y_{1}}_{x_{1}}\left(\sum^{n}_{m=1}\frac{\partial\rho^{\left(m\right)}}{\partial x_{1}}\right)\left(t,\left(\xi,z\right)\right)\mathrm{d}\xi=\int^{y_{1}}_{x_{1}}S^{\left(n\right)}\left(t,\left(\xi,z\right)\right)\mathrm{d}\xi.
\]
Since $S^{\left(n\right)}$ can be uniformly bounded in space, time
and $n\in\mathbb{N}$ via equation \eqref{eq:cauchy S} and the domain
of integration is compact (as $x$ and $y$ are fixed), the Dominated
Convergence Theorem ensures us that
\[
\rho\left(t,\left(y_{1},z\right)\right)-\rho\left(t,\left(x_{1},z\right)\right)=\int^{y_{1}}_{x_{1}}\lim_{n\to\infty}S^{\left(n\right)}\left(t,\left(\xi,z\right)\right)\mathrm{d}\xi.
\]
As this argument is valid $\forall t\in\left[0,1\right]$ and $\forall x_{1},y_{1},z\in\mathbb{R}$,
we conclude that $\lim_{n\to\infty}S^{\left(n\right)}=\frac{\partial\rho}{\partial x_{1}}$.
\end{proof}

\begin{prop}
\label{prop:bounds gradu}Let $\left(\rho,\omega,f_{\rho},f_{\omega}\right)$
be the solution of the forced Boussinesq system presented in \cite{Articulo Boussinesq}.
We have
\[
\widetilde{\nabla}^{n}\widetilde{u^{\left(n\right)}}^{n}\left(t,x\right)=\widetilde{\theta^{\left(n\right)}}^{n}\left(t,x\right)+\lambda_{n}\widetilde{\Theta^{\left(n\right)}}^{n}\left(t,x\right),
\]
where
\[
\widetilde{\theta^{\left(n\right)}}^{n}\left(t,x\right)=B_{n}\left(t\right)\varphi\left(\lambda_{n}x_{1}\right)\varphi\left(\lambda_{n}x_{2}\right)\left(\begin{array}{cc}
a_{n}\left(t\right)b_{n}\left(t\right)\cos\left(x_{1}\right)\cos\left(x_{2}\right) & a_{n}\left(t\right)^{2}\sin\left(x_{1}\right)\sin\left(x_{2}\right)\\
-b_{n}\left(t\right)^{2}\sin\left(x_{1}\right)\sin\left(x_{2}\right) & -a_{n}\left(t\right)b_{n}\left(t\right)\cos\left(x_{1}\right)\cos\left(x_{2}\right)
\end{array}\right).
\]
Furthermore,
\[
\begin{aligned}\left|\left|\theta^{\left(n\right)}\left(t,\cdot\right)\right|\right|_{L^{\infty}\left(\mathbb{R}^{2};\mathbb{R}^{2\times2}\right)} & \lesssim_{\varphi}YC^{\delta\left(\frac{1}{1-\gamma}\right)^{n}},\\
\left|\left|\theta^{\left(n\right)}\left(t,\cdot\right)\right|\right|_{\dot{C}^{\alpha}\left(\mathbb{R}^{2};\mathbb{R}^{2\times2}\right)} & \lesssim_{\varphi}Yb_{n}\left(t\right)^{\alpha}C^{\left(\delta+2\mu\right)\left(\frac{1}{1-\gamma}\right)^{n}}\leq YC^{\left(\alpha\left(1+k_{\max}\right)+\delta+3\mu\right)\left(\frac{1}{1-\gamma}\right)^{n}},\\
\left|\left|\Theta^{\left(n\right)}\left(t,\cdot\right)\right|\right|_{L^{\infty}\left(\mathbb{R}^{2};\mathbb{R}^{2\times2}\right)} & \lesssim_{\varphi}YC^{\delta\left(\frac{1}{1-\gamma}\right)^{n}},\\
\left|\left|\Theta^{\left(n\right)}\left(t,\cdot\right)\right|\right|_{\dot{C}^{\alpha}\left(\mathbb{R}^{2};\mathbb{R}^{2\times2}\right)} & \lesssim_{\varphi}Yb_{n}\left(t\right)^{\alpha}C^{\left(\delta+2\mu\right)\left(\frac{1}{1-\gamma}\right)^{n}}\leq YC^{\left(\alpha\left(1+k_{\max}\right)+\delta+3\mu\right)\left(\frac{1}{1-\gamma}\right)^{n}}.
\end{aligned}
\]
\end{prop}
\begin{proof}
By differentiating in the expression given for the velocity in Proposition
\ref{prop:computations vorticity}, taking into account the definition
of $\tilde{\nabla}^{n}$ given in equation \eqref{eq:Boussinesq def nabla_tilde},
we obtain the first part of the statement. Clearly, by equations \eqref{eq:property Calpha multiplication}, \eqref{eq:property Calpha composition}, and \eqref{eq:Boussinesq lambda_n},
we have
\[
\begin{aligned}\left|\left|\widetilde{\theta^{\left(n\right)}}^{n}\left(t,\cdot\right)\right|\right|_{L^{\infty}\left(\mathbb{R}^{2};\mathbb{R}^{2\times2}\right)} & \lesssim_{\varphi}B_{n}\left(t\right)\max\left\{ a_{n}\left(t\right),b_{n}\left(t\right)\right\} ^{2}\leq B_{n}\left(t\right)\left(a_{n}\left(t\right)^{2}+b_{n}\left(t\right)^{2}\right),\\
\left|\left|\widetilde{\theta^{\left(n\right)}}^{n}\left(t,\cdot\right)\right|\right|_{\dot{C}^{\alpha}\left(\mathbb{R}^{2};\mathbb{R}^{2\times2}\right)} & \lesssim_{\varphi}B_{n}\left(t\right)\left(1+\lambda^{\alpha}_{n}\right)\max\left\{ a_{n}\left(t\right),b_{n}\left(t\right)\right\} ^{2}\lesssim B_{n}\left(t\right)\left(a_{n}\left(t\right)^{2}+b_{n}\left(t\right)^{2}\right).
\end{aligned}
\]
Taking into account equations \eqref{eq:Boussinesq Bn(t)} and \eqref{eq:Boussinesq Mn},
we deduce that
\[
\left|\left|\widetilde{\theta^{\left(n\right)}}^{n}\left(t,\cdot\right)\right|\right|_{L^{\infty}\left(\mathbb{R}^{2};\mathbb{R}^{2\times2}\right)}\lesssim_{\varphi}YC^{\delta\left(\frac{1}{1-\gamma}\right)^{n}},\quad\left|\left|\widetilde{\theta^{\left(n\right)}}^{n}\left(t,\cdot\right)\right|\right|_{\dot{C}^{\alpha}\left(\mathbb{R}^{2};\mathbb{R}^{2\times2}\right)}\lesssim_{\varphi}YC^{\delta\left(\frac{1}{1-\gamma}\right)^{n}}.
\]
Since $\left|\left|\cdot\right|\right|_{L^{\infty}\left(\mathbb{R}^{2}\right)}$
is invariant under diffeomorphisms of the domain, we conclude the
result given in the statement for $\left|\left|\theta^{\left(n\right)}\left(t,\cdot\right)\right|\right|_{L^{\infty}\left(\mathbb{R}^{2}\right)}$.
Concerning $\left|\left|\theta^{\left(n\right)}\left(t,\cdot\right)\right|\right|_{\dot{C}^{\alpha}\left(\mathbb{R}^{2}\right)}$,
by equations \eqref{eq:property Calpha composition} and \eqref{eq:Boussinesq jacobian inverse},
it must be
\[
\begin{aligned}\left|\left|\theta^{\left(n\right)}\left(t,\cdot\right)\right|\right|_{\dot{C}^{\alpha}\left(\mathbb{R}^{2};\mathbb{R}^{2\times2}\right)} & =\left|\left|\widetilde{\theta^{\left(n\right)}}^{n}\left(t,\left(\phi^{\left(n\right)}\right)^{-1}\left(t,\cdot\right)\right)\right|\right|_{\dot{C}^{\alpha}\left(\mathbb{R}^{2};\mathbb{R}^{2\times2}\right)}\lesssim\\
 & \lesssim\left|\left|\left(\phi^{\left(n\right)}\right)^{-1}\left(t,\cdot\right)\right|\right|^{\alpha}_{\dot{C}^{1}\left(\mathbb{R}^{2};\mathbb{R}^{2}\right)}\left|\left|\widetilde{\theta^{\left(n\right)}}^{n}\left(t,\cdot\right)\right|\right|_{\dot{C}^{\alpha}\left(\mathbb{R}^{2};\mathbb{R}^{2\times2}\right)}\leq\\
 & \lesssim\max\left\{ a_{n}\left(t\right),b_{n}\left(t\right)\right\} ^{\alpha}YC^{\delta\left(\frac{1}{1-\gamma}\right)^{n}}.
\end{aligned}
\]
Finally, Corollary \ref{cor:bound for an(t). bn(t)} and equation
\eqref{eq:Boussinesq def kmax}, along with the bound $\alpha\leq1$
applied where $\alpha$ is multiplied by $\mu$, ensure that
\[
\begin{aligned}\left|\left|\theta^{\left(n\right)}\left(t,\cdot\right)\right|\right|_{\dot{C}^{\alpha}\left(\mathbb{R}^{2};\mathbb{R}^{2\times2}\right)} & \lesssim_{\varphi}Yb_{n}\left(t\right)^{\alpha}C^{\left(\delta+2\alpha\mu\right)\left(\frac{1}{1-\gamma}\right)^{n}}\leq Yb_{n}\left(t\right)^{\alpha}C^{\left(\delta+2\mu\right)\left(\frac{1}{1-\gamma}\right)^{n}}\leq\\
 & \lesssim_{\varphi}YC^{\left(\alpha\left(1+k_{\max}\right)+\delta+3\mu\right)\left(\frac{1}{1-\gamma}\right)^{n}}.
\end{aligned}
\]

Now, let us focus on the bounds of $\Theta^{\left(n\right)}\left(t,\cdot\right)$.
Looking at Proposition \ref{prop:computations vorticity}, we see
that any component of the matrix $\widetilde{\Theta^{\left(n\right)}}^{n}\left(t,x\right)$
is given by a sum of terms of the form
\[
\begin{aligned} & B_{n}\left(t\right)\overbrace{\left(\begin{array}{c}
\text{non-negative}\\
\text{power of }\lambda_{n}
\end{array}\right)}^{\leq1}\cdot\overbrace{\left(\begin{array}{c}
a_{n}\left(t\right)\\
\text{or}\\
b_{n}\left(t\right)
\end{array}\right)\cdot\left(\begin{array}{c}
a_{n}\left(t\right)\\
\text{or}\\
b_{n}\left(t\right)
\end{array}\right)}^{\leq\max\left\{ a_{n}\left(t\right)^{2},b_{n}\left(t\right)^{2}\right\} }\cdot\\
\cdot & \left(\begin{array}{c}
\text{derivative}\\
\text{of }\varphi\text{ of}\\
\text{order }\le2
\end{array}\right)\left(\lambda_{n}x_{1}\right)\left(\begin{array}{c}
\text{derivative}\\
\text{of }\varphi\text{ of}\\
\text{order }\le2
\end{array}\right)\left(\lambda_{n}x_{2}\right)\left(\begin{array}{c}
\sin\\
\text{or}\\
\cos
\end{array}\right)\left(x_{1}\right)\left(\begin{array}{c}
\sin\\
\text{or}\\
\cos
\end{array}\right)\left(x_{2}\right).
\end{aligned}
\]
Thereby, the procedure we have followed to obtain bounds for $\theta^{\left(n\right)}$
is also valid for $\Theta^{\left(n\right)}$, which provides the result.
\end{proof}

\begin{lem}
\label{lem:temporal derivatives velocity}Let $\left(\rho,\omega,f_{\rho},f_{\omega}\right)$
be the solution of the forced Boussinesq system presented in \cite{Articulo Boussinesq}.
Let $n\in\mathbb{N}$ with $n\ge2$. Then, $\forall t\in\left[t_{n},1\right]$,
we have
\[
\begin{aligned}\left|\frac{\mathrm{d}}{\mathrm{d}t}\left(B_{n}\left(t\right)a_{n}\left(t\right)\right)\right| & \lesssim_{\delta,\mu}\frac{Y^{2}C^{\left(2\mu+2\delta\right)\left(\frac{1}{1-\gamma}\right)^{n}}}{b_{n}\left(t\right)},\\
\left|\frac{\mathrm{d}}{\mathrm{d}t}\left(B_{n}\left(t\right)b_{n}\left(t\right)\right)\right| & \lesssim_{\delta,\mu}\frac{Y^{2}C^{\left(2\mu+2\delta\right)\left(\frac{1}{1-\gamma}\right)^{n}}}{b_{n}\left(t\right)}.
\end{aligned}
\]
\end{lem}
\begin{proof}
On the one hand, equation \eqref{eq:Boussinesq dBndt} ensures us
that
\begin{equation}
\begin{aligned}\left|\frac{\mathrm{d}B_{n}}{\mathrm{d}t}\left(t\right)\right| & \lesssim_{\delta,\mu}\frac{2M_{n}}{a_{n}\left(t\right)^{2}+b_{n}\left(t\right)^{2}}YC^{\left(2\mu+\delta\left(1-\gamma\right)\right)\left(\frac{1}{1-\gamma}\right)^{n}}.\end{aligned}
\label{eq:bound dBndt}
\end{equation}
On the other hand, equation \eqref{eq:Boussinesq dynamics} provides,
$\forall t\in\left[t_{n},1\right]$,
\[
\left|\frac{\mathrm{d}b_{n}}{\mathrm{d}t}\left(t\right)\right|\leq b_{n}\left(t\right)\left|\frac{\mathrm{d}}{\mathrm{d}t}\left(\ln\left(b_{n}\left(t\right)\right)\right)\right|\lesssim b_{n}\left(t\right)\sum^{n-1}_{m=1}B_{m}\left(t\right)a_{m}\left(t\right)b_{m}\left(t\right)=b_{n}\left(t\right)\sum^{n-1}_{m=1}B_{m}\left(t\right)a_{m}\left(1\right)b_{m}\left(1\right).
\]
Proposition \ref{prop:time convergence}, Lemma \ref{lem:estimate sum superexponential}
and equation \eqref{eq:Boussinesq Mn} lead us to
\begin{equation}
\left|\frac{\mathrm{d}b_{n}}{\mathrm{d}t}\left(t\right)\right|\lesssim b_{n}\left(t\right)\sum^{n-1}_{m=1}M_{m}=b_{n}\left(t\right)\sum^{n-1}_{m=1}YC^{\delta\left(\frac{1}{1-\gamma}\right)^{m}}\lesssim_{\delta}b_{n}\left(t\right)YC^{\delta\left(\frac{1}{1-\gamma}\right)^{n-1}}.\label{eq:bound dbndt}
\end{equation}
Furthermore, since $\frac{\mathrm{d}}{\mathrm{d}t}\left(a_{n}\left(t\right)b_{n}\left(t\right)\right)=0$
by equation \eqref{eq:Boussinesq dynamics}, we must have
\begin{equation}
\left|\frac{\mathrm{d}a_{n}}{\mathrm{d}t}\left(t\right)\right|=\frac{a_{n}\left(t\right)}{b_{n}\left(t\right)}\left|\frac{\mathrm{d}b_{n}}{\mathrm{d}t}\left(t\right)\right|\lesssim_{\delta}a_{n}\left(t\right)YC^{\delta\left(\frac{1}{1-\gamma}\right)^{n-1}}.\label{eq:bound dandt}
\end{equation}

Then, Leibniz' rule, combined with equations \eqref{eq:bound dBndt}, \eqref{eq:bound dbndt}, \eqref{eq:bound dandt}, and \eqref{eq:Boussinesq Bn(t)}
and the bounds $1-\gamma\leq1$ and $h^{\left(n\right)}\left(t\right)\leq1$,
provides
\[
\begin{aligned}\left|\frac{\mathrm{d}}{\mathrm{d}t}\left(B_{n}\left(t\right)a_{n}\left(t\right)\right)\right| & \lesssim_{\delta,\mu}\frac{2M_{n}}{a_{n}\left(t\right)^{2}+b_{n}\left(t\right)^{2}}YC^{\left(2\mu+\delta\left(1-\gamma\right)\right)\left(\frac{1}{1-\gamma}\right)^{n}}a_{n}\left(t\right)+\frac{2M_{n}}{a_{n}\left(t\right)^{2}+b_{n}\left(t\right)^{2}}a_{n}\left(t\right)YC^{\delta\left(\frac{1}{1-\gamma}\right)^{n-1}}\lesssim\\
 & \lesssim_{\delta,\mu}2M_{n}YC^{\left(2\mu+\delta\right)\left(\frac{1}{1-\gamma}\right)^{n-1}}\underbrace{\frac{a_{n}\left(t\right)b_{n}\left(t\right)}{a_{n}\left(t\right)^{2}+b_{n}\left(t\right)^{2}}}_{\lesssim1}\frac{1}{b_{n}\left(t\right)},\\
\left|\frac{\mathrm{d}}{\mathrm{d}t}\left(B_{n}\left(t\right)b_{n}\left(t\right)\right)\right| & \lesssim_{\delta,\mu}\frac{2M_{n}}{a_{n}\left(t\right)^{2}+b_{n}\left(t\right)^{2}}YC^{\left(2\mu+\delta\left(1-\gamma\right)\right)\left(\frac{1}{1-\gamma}\right)^{n}}b_{n}\left(t\right)+\frac{2M_{n}}{a_{n}\left(t\right)^{2}+b_{n}\left(t\right)^{2}}b_{n}\left(t\right)YC^{\delta\left(\frac{1}{1-\gamma}\right)^{n-1}}\lesssim\\
 & \lesssim_{\delta,\mu}2M_{n}YC^{\left(2\mu+\delta\right)\left(\frac{1}{1-\gamma}\right)^{n-1}}\underbrace{\frac{b_{n}\left(t\right)^{2}}{a_{n}\left(t\right)^{2}+b_{n}\left(t\right)^{2}}}_{\leq1}\frac{1}{b_{n}\left(t\right)}.
\end{aligned}
\]
Using equation \eqref{eq:Boussinesq Mn} to get rid of $M_{n}$ leads
to the results of the statement.
\end{proof}

\begin{prop}
\label{prop:bound material derivative u}Let $\left(\rho,\omega,f_{\rho},f_{\omega}\right)$
be the solution of the forced Boussinesq system presented in \cite{Articulo Boussinesq}.
Then, $\forall\alpha\in\left(0,1\right)$, as long as $\delta\le\Upsilon_{1}\left(\alpha\right)$,
$\mu\leq\Upsilon_{2}\left(\alpha\right)$ and $1-\gamma\leq\Upsilon_{3}$,
$\left(\frac{\partial u}{\partial t}+u\cdot\nabla u\right)\left(t,\cdot\right)$
admits a (unique) extension to time $t=1$ such that $\frac{\partial u}{\partial t}+u\cdot\nabla u\in C^{0}_{t}C^{\alpha}_{x}\left(\left[0,1\right]\times\mathbb{R}^{2}\right)$.
\end{prop}
\begin{proof}
To tackle this proof, it proves useful to consider
\begin{equation}
\mathcal{U}^{\left(n\right)}\coloneqq\frac{\partial U^{\left(n\right)}}{\partial t}+U^{\left(n\right)}\cdot\nabla U^{\left(n\right)},\label{eq:caligraphic capital U}
\end{equation}
where we recall that $U^{\left(n\right)}=\sum^{n}_{m=1}u^{\left(m\right)}$,
and to write the equation for the difference $\mathcal{U}^{\left(n\right)}-\mathcal{U}^{\left(n-1\right)}$:
\[
\mathcal{U}^{\left(n\right)}-\mathcal{U}^{\left(n-1\right)}=\frac{\partial u^{\left(n\right)}}{\partial t}+U^{\left(n-1\right)}\cdot\nabla u^{\left(n\right)}+u^{\left(n\right)}\cdot\nabla U^{\left(n-1\right)}+u^{\left(n\right)}\cdot\nabla u^{\left(n\right)}.
\]
To make the computations easier, we should change variables to move
with layer $n$. Notice that, by equations \eqref{eq:Boussinesq phi}, \eqref{eq:Boussinesq jacobian inverse}, and \eqref{eq:Boussinesq def nabla_tilde},
\[
\begin{aligned}\left(u^{\left(n\right)}\cdot\nabla u^{\left(n\right)}\right)\left(t,\phi^{\left(n\right)}\left(t,x\right)\right) & =\widetilde{u^{\left(n\right)}}^{n}\left(t,x\right)\cdot\tilde{\nabla}^{n}\widetilde{u^{\left(n\right)}}^{n}\left(t,x\right),\\
\left(u^{\left(n\right)}\cdot\nabla U^{\left(n-1\right)}\right)\left(t,\phi^{\left(n\right)}\left(t,x\right)\right) & =\widetilde{u^{\left(n\right)}}^{n}\left(t,x\right)\cdot\tilde{\nabla}^{n}\widetilde{U^{\left(n-1\right)}}^{n}\left(t,x\right),\\
\left(U^{\left(n-1\right)}\cdot\nabla u^{\left(n\right)}\right)\left(t,\phi^{\left(n\right)}\left(t,x\right)\right) & =\widetilde{U^{\left(n-1\right)}}^{n}\left(t,x\right)\cdot\tilde{\nabla}^{n}\widetilde{u^{\left(n\right)}}^{n}\left(t,x\right),\\
\frac{\partial u^{\left(n\right)}}{\partial t}\left(t,\phi^{\left(n\right)}\left(t,x\right)\right) & =\frac{\partial\widetilde{u^{\left(n\right)}}^{n}}{\partial t}\left(t,x\right)-\frac{\partial\phi^{\left(n\right)}}{\partial t}\left(t,x\right)\cdot\widetilde{\nabla}^{n}\widetilde{u^{\left(n\right)}}^{n}\left(t,x\right).
\end{aligned}
\]
Combining the above with equation \eqref{eq:Boussinesq transport is Taylor of velocity},
we deduce that
\begin{equation}
\begin{aligned}\widetilde{\mathcal{U}^{\left(n\right)}}^{n}-\widetilde{\mathcal{U}^{\left(n-1\right)}}^{n} & =\frac{\partial\widetilde{u^{\left(n\right)}}^{n}}{\partial t}\left(t,x\right)+\left(\widetilde{U^{\left(n-1\right)}}^{n}\left(t,x\right)-\widetilde{U^{\left(n-1\right)}}^{n}\left(t,0\right)-\mathrm{J}\widetilde{U^{\left(n-1\right)}}^{n}\left(t,0\right)\cdot\left(\begin{array}{c}
x_{1}\\
x_{2}
\end{array}\right)\right)\cdot\tilde{\nabla}^{n}\widetilde{u^{\left(n\right)}}^{n}\left(t,x\right)+\\
 & \quad+\widetilde{u^{\left(n\right)}}^{n}\left(t,x\right)\cdot\tilde{\nabla}^{n}\widetilde{U^{\left(n-1\right)}}^{n}\left(t,x\right)+\widetilde{u^{\left(n\right)}}^{n}\left(t,x\right)\cdot\tilde{\nabla}^{n}\widetilde{u^{\left(n\right)}}^{n}\left(t,x\right).
\end{aligned}
\label{eq:decomposition velocity}
\end{equation}

Let us begin bounding the pure quadratic term
\[
\widetilde{Q^{\left(n\right)}_{u}}^{n}\left(t,x\right)\coloneqq\widetilde{u^{\left(n\right)}}^{n}\left(t,x\right)\cdot\tilde{\nabla}^{n}\widetilde{u^{\left(n\right)}}^{n}\left(t,x\right).
\]
Bounding $\left|\left|Q^{\left(n\right)}_{u}\left(t,\cdot\right)\right|\right|_{L^{\infty}\left(\mathbb{R}^{2};\mathbb{R}^{2}\right)}$
is immediate. Using Propositions \ref{prop:bounds gradu} and \ref{prop:form of the velocity},
we infer that
\[
\left|\left|Q^{\left(n\right)}_{u}\left(t,\cdot\right)\right|\right|_{L^{\infty}\left(\mathbb{R}^{2};\mathbb{R}^{2}\right)}\lesssim_{\varphi}B_{n}\left(t\right)\max\left\{ a_{n}\left(t\right),b_{n}\left(t\right)\right\} YC^{\delta\left(\frac{1}{1-\gamma}\right)^{n}}.
\]
Equations \eqref{eq:Boussinesq Bn(t)} and \eqref{eq:Boussinesq an(t) bn(t)}
and Corollary \ref{cor:bound for an(t). bn(t)} provide
\[
\left|\left|Q^{\left(n\right)}_{u}\left(t,\cdot\right)\right|\right|_{L^{\infty}\left(\mathbb{R}^{2};\mathbb{R}^{2}\right)}\lesssim_{\varphi}\frac{2M_{n}}{a_{n}\left(t\right)^{2}+b_{n}\left(t\right)^{2}}b_{n}\left(t\right)YC^{\left(2\mu+\delta\right)\left(\frac{1}{1-\gamma}\right)^{n}}.
\]
Employing equation \eqref{eq:Boussinesq Mn} and Corollary \ref{cor:bound for an(t). bn(t)}
again, we deduce that
\begin{equation}
\left|\left|Q^{\left(n\right)}_{u}\left(t,\cdot\right)\right|\right|_{L^{\infty}\left(\mathbb{R}^{2};\mathbb{R}^{2}\right)}\lesssim_{\varphi}\frac{2M_{n}YC^{\left(2\mu+\delta\right)\left(\frac{1}{1-\gamma}\right)^{n}}}{b_{n}\left(t\right)}\underbrace{\frac{b_{n}\left(t\right)^{2}}{a_{n}\left(t\right)^{2}+b_{n}\left(t\right)^{2}}}_{\leq1}\leq\frac{Y^{2}C^{\left(2\mu+2\delta\right)\left(\frac{1}{1-\gamma}\right)^{n}}}{b_{n}\left(t\right)}\leq Y^{2}C^{\left(-1+3\mu+2\delta\right)\left(\frac{1}{1-\gamma}\right)^{n}}.\label{eq:U bound quadratic Linfty}
\end{equation}
Now, we shall continue with $\left|\left|Q^{\left(n\right)}_{u}\left(t,\cdot\right)\right|\right|_{\dot{C}^{\alpha}\left(\mathbb{R}^{2}\right)}$.
Equation \eqref{eq:property Calpha multiplication}, Propositions
\ref{prop:bounds gradu} and \ref{prop:form of the velocity} ensure
us that
\[
\begin{aligned}\left|\left|Q^{\left(n\right)}_{u}\left(t,\cdot\right)\right|\right|_{\dot{C}^{\alpha}\left(\mathbb{R}^{2};\mathbb{R}^{2}\right)} & \lesssim_{\varphi}B_{n}\left(t\right)\max\left\{ a_{n}\left(t\right)^{1+\alpha},b_{n}\left(t\right)^{1+\alpha}\right\} YC^{\delta\left(\frac{1}{1-\gamma}\right)^{n}}+\\
 & \quad+B_{n}\left(t\right)\max\left\{ a_{n}\left(t\right),b_{n}\left(t\right)\right\} Yb_{n}\left(t\right)^{\alpha}C^{\left(\delta+2\mu\right)\left(\frac{1}{1-\gamma}\right)^{n}}.
\end{aligned}
\]
Employing the same procedure as before, we arrive at
\begin{equation}
\begin{aligned}\left|\left|Q^{\left(n\right)}_{u}\left(t,\cdot\right)\right|\right|_{\dot{C}^{\alpha}\left(\mathbb{R}^{2};\mathbb{R}^{2}\right)} & \lesssim_{\varphi}\frac{2M_{n}}{a_{n}\left(t\right)^{2}+b_{n}\left(t\right)^{2}}\frac{b_{n}\left(t\right)^{2}}{b_{n}\left(t\right)^{1-\alpha}}YC^{\left(4\mu+\delta\right)\left(\frac{1}{1-\gamma}\right)^{n}}+\\
 & \quad+\frac{2M_{n}}{a_{n}\left(t\right)^{2}+b_{n}\left(t\right)^{2}}\frac{b_{n}\left(t\right)^{2}}{b_{n}\left(t\right)}Yb_{n}\left(t\right)^{\alpha}C^{\left(\delta+4\mu\right)\left(\frac{1}{1-\gamma}\right)^{n}}\lesssim\\
 & \lesssim_{\varphi}Y^{2}\frac{C^{\left(4\mu+2\delta\right)\left(\frac{1}{1-\gamma}\right)^{n}}}{b_{n}\left(t\right)^{1-\alpha}}+Y^{2}\frac{C^{\left(2\delta+4\mu\right)\left(\frac{1}{1-\gamma}\right)^{n}}}{b_{n}\left(t\right)^{1-\alpha}}\leq\\
 & \lesssim_{\varphi}Y^{2}C^{\left(-\left(1-\alpha\right)+5\mu+2\delta\right)\left(\frac{1}{1-\gamma}\right)^{n}}.
\end{aligned}
\label{eq:U bound quadratic Calpha}
\end{equation}

Let us continue with the temporal derivative $\frac{\partial\widetilde{u^{\left(n\right)}}^{n}}{\partial t}\left(t,x\right)$.
Using Proposition \ref{prop:computations vorticity}, we deduce that,
$\forall i\in\left\{ 1,2\right\} $,
\[
\begin{aligned}\frac{\partial\widetilde{u^{\left(n\right)}_{i}}^{n}}{\partial t}\left(t,x\right) & =\sum_{2}\underbrace{\lambda^{0\text{ or }1}_{n}}_{\leq1}\left(\begin{matrix}\frac{\mathrm{d}}{\mathrm{d}t}\left(B_{n}\left(t\right)b_{n}\left(t\right)\right)\\
\text{or}\\
\frac{\mathrm{d}}{\mathrm{d}t}\left(B_{n}\left(t\right)a_{n}\left(t\right)\right)
\end{matrix}\right)\left(\begin{array}{c}
\text{derivative}\\
\text{of }\varphi\text{ of}\\
\text{order }\le1
\end{array}\right)\left(\lambda_{n}x_{1}\right)\left(\begin{array}{c}
\text{derivative}\\
\text{of }\varphi\text{ of}\\
\text{order }\le1
\end{array}\right)\left(\lambda_{n}x_{2}\right)\cdot\\
 & \quad\cdot\left(\begin{array}{c}
\sin\\
\text{or}\\
\cos
\end{array}\right)\left(x_{1}\right)\left(\begin{array}{c}
\sin\\
\text{or}\\
\cos
\end{array}\right)\left(x_{2}\right),
\end{aligned}
\]
where the subscript of the sum only indicates the number of summands.
The invariance of $\left|\left|\cdot\right|\right|_{L^{\infty}\left(\mathbb{R}^{2}\right)}$
under diffeomorphisms of the domain, Lemma \ref{lem:temporal derivatives velocity},
equations \eqref{eq:property Calpha composition} and \eqref{eq:Boussinesq jacobian inverse}
and Corollary \ref{cor:bound for an(t). bn(t)} allow us to conclude
that
\begin{equation}
\begin{aligned}\left|\left|\frac{\partial\widetilde{u^{\left(n\right)}}^{n}}{\partial t}\left(t,\left(\phi^{\left(n\right)}\right)^{-1}\left(t,\cdot\right)\right)\right|\right|_{L^{\infty}\left(\mathbb{R}^{2};\mathbb{R}^{2}\right)} & \lesssim_{\delta,\mu,\varphi}\frac{Y^{2}C^{\left(2\mu+2\delta\right)\left(\frac{1}{1-\gamma}\right)^{n}}}{b_{n}\left(t\right)}\leq Y^{2}C^{\left(-1+3\mu+2\delta\right)\left(\frac{1}{1-\gamma}\right)^{n}},\\
\left|\left|\frac{\partial\widetilde{u^{\left(n\right)}}^{n}}{\partial t}\left(t,\left(\phi^{\left(n\right)}\right)^{-1}\left(t,\cdot\right)\right)\right|\right|_{\dot{C}^{\alpha}\left(\mathbb{R}^{2};\mathbb{R}^{2}\right)} & \lesssim\left|\left|\left(\phi^{\left(n\right)}\right)^{-1}\left(t,\cdot\right)\right|\right|^{\alpha}_{\dot{C}^{1}\left(\mathbb{R}^{2};\mathbb{R}^{2}\right)}\left|\left|\frac{\partial\widetilde{u^{\left(n\right)}}^{n}}{\partial t}\left(t,\cdot\right)\right|\right|_{\dot{C}^{\alpha}\left(\mathbb{R}^{2}\right)}\lesssim\\
 & \lesssim_{\delta,\mu,\varphi}\frac{\max\left\{ a_{n}\left(t\right)^{\alpha},b_{n}\left(t\right)^{\alpha}\right\} }{b_{n}\left(t\right)}Y^{2}C^{\left(2\mu+2\delta\right)\left(\frac{1}{1-\gamma}\right)^{n}}\lesssim\\
 & \lesssim_{\delta,\mu,\varphi}\frac{b_{n}\left(t\right)^{\alpha}C^{2\mu\left(\frac{1}{1-\gamma}\right)^{n}}}{b_{n}\left(t\right)}Y^{2}C^{\left(2\mu+2\delta\right)\left(\frac{1}{1-\gamma}\right)^{n}}\lesssim\\
 & \lesssim_{\delta,\mu,\varphi}\frac{Y^{2}}{b_{n}\left(t\right)^{1-\alpha}}C^{\left(4\mu+2\delta\right)\left(\frac{1}{1-\gamma}\right)^{n}}\leq Y^{2}C^{\left(-\left(1-\alpha\right)+5\mu+2\delta\right)\left(\frac{1}{1-\gamma}\right)^{n}}.
\end{aligned}
\label{eq:U bound dudt}
\end{equation}

Now, we shall focus our attention on the term
\[
\widetilde{\Pi^{\left(n\right)}_{u}}^{n}\left(t,x\right)\coloneqq\widetilde{u^{\left(n\right)}}^{n}\left(t,x\right)\cdot\tilde{\nabla}^{n}\widetilde{U^{\left(n-1\right)}}^{n}\left(t,x\right).
\]
The chain rule, the definition of $\widetilde{\nabla}^{n}$ (see equation
\eqref{eq:Boussinesq def nabla_tilde}) and equation \eqref{eq:Boussinesq phi}
show that
\[
\begin{aligned}\widetilde{\Pi^{\left(n\right)}_{u}}^{n}\left(t,x\right) & =\widetilde{u^{\left(n\right)}}^{n}\left(t,x\right)\cdot\left[\nabla\phi^{\left(n\right)}\left(t,0\right)\right]^{-1}\cdot\nabla\left[\sum^{n-1}_{m=1}u^{\left(m\right)}\left(t,\phi^{\left(n\right)}\left(t,x\right)\right)\right]=\\
 & =\widetilde{u^{\left(n\right)}}^{n}\left(t,x\right)\cdot\left[\nabla\phi^{\left(n\right)}\left(t,0\right)\right]^{-1}\cdot\left[\nabla\phi^{\left(n\right)}\left(t,0\right)\right]\cdot\left(\sum^{n-1}_{m=1}\nabla u^{\left(m\right)}\left(t,\phi^{\left(n\right)}\left(t,x\right)\right)\right)=\\
 & =\widetilde{u^{\left(n\right)}}^{n}\left(t,x\right)\cdot\left(\sum^{n-1}_{m=1}\nabla u^{\left(m\right)}\left(t,\phi^{\left(m\right)}\left(t,\left(\phi^{\left(m\right)}\right)^{-1}\left(t,\phi^{\left(n\right)}\left(t,x\right)\right)\right)\right)\right)=\\
 & =\widetilde{u^{\left(n\right)}}^{n}\left(t,x\right)\cdot\left(\sum^{n-1}_{m=1}\widetilde{\nabla u^{\left(m\right)}}^{m}\left(t,\left(\phi^{\left(m\right)}\right)^{-1}\left(t,\phi^{\left(n\right)}\left(t,x\right)\right)\right)\right)=\\
 & =\widetilde{u^{\left(n\right)}}^{n}\left(t,x\right)\cdot\sum^{n-1}_{m=1}\tilde{\nabla}^{m}\widetilde{u^{\left(m\right)}}^{m}\left(t,\left(\phi^{\left(m\right)}\right)^{-1}\left(t,\phi^{\left(n\right)}\left(t,x\right)\right)\right).
\end{aligned}
\]
In view of the above, notice that
\[
\Pi^{\left(n\right)}_{u}\left(t,x\right)=u^{\left(n\right)}\left(t,x\right)\cdot\left(\sum^{n-1}_{m=1}\tilde{\nabla}^{m}\widetilde{u^{\left(m\right)}}^{m}\left(t,\left(\phi^{\left(m\right)}\right)^{-1}\left(t,x\right)\right)\right).
\]
Using Propositions \ref{prop:bounds gradu} and \ref{prop:form of the velocity},
equations \eqref{eq:Boussinesq Bn(t)}, \eqref{eq:Boussinesq an(t) bn(t)}, and \eqref{eq:Boussinesq Mn},
Lemma \ref{lem:estimate sum superexponential} and Corollary \ref{cor:bound for an(t). bn(t)},
one can deduce that
\begin{equation}
\begin{aligned}\left|\left|\Pi^{\left(n\right)}_{u}\left(t,\cdot\right)\right|\right|_{L^{\infty}\left(\mathbb{R}^{2};\mathbb{R}^{2}\right)} & \lesssim_{\varphi}B_{n}\left(t\right)\max\left\{ a_{n}\left(t\right),b_{n}\left(t\right)\right\} \sum^{n-1}_{m=1}YC^{\delta\left(\frac{1}{1-\gamma}\right)^{m}}\lesssim_{\delta}\\
 & \lesssim_{\varphi,\delta}\frac{2M_{n}}{a_{n}\left(t\right)^{2}+b_{n}\left(t\right)^{2}}b_{n}\left(t\right)C^{2\mu\left(\frac{1}{1-\gamma}\right)^{n}}YC^{\delta\left(\frac{1}{1-\gamma}\right)^{n-1}}\lesssim\\
 & \lesssim_{\varphi,\delta}\frac{Y^{2}}{b_{n}\left(t\right)}C^{\left(2\delta+2\mu\right)\left(\frac{1}{1-\gamma}\right)^{n}}\leq Y^{2}C^{\left(-1+2\delta+3\mu\right)\left(\frac{1}{1-\gamma}\right)^{n}}.
\end{aligned}
\label{eq:U bound PI Lifnty}
\end{equation}
To compute $\left|\left|\Pi^{\left(n\right)}_{u}\left(t,\cdot\right)\right|\right|_{\dot{C}^{\alpha}\left(\mathbb{R}^{2};\mathbb{R}^{2}\right)}$,
we will resort to equations \eqref{eq:property Calpha composition} and \eqref{eq:Boussinesq jacobian inverse}.
Thereby, applying equation \eqref{eq:property Calpha multiplication},
Propositions \ref{prop:bounds gradu} and \ref{prop:form of the velocity},
equations \eqref{eq:Boussinesq Bn(t)}, \eqref{eq:Boussinesq an(t) bn(t)}, \eqref{eq:Boussinesq def kmax}, and \eqref{eq:Boussinesq Mn},
Lemma \ref{lem:estimate sum superexponential} and Corollary \ref{cor:bound for an(t). bn(t)},
\begin{equation}
\begin{aligned}\left|\left|\Pi^{\left(n\right)}_{u}\left(t,\cdot\right)\right|\right|_{\dot{C}^{\alpha}\left(\mathbb{R}^{2};\mathbb{R}^{2}\right)} & \lesssim_{\varphi}B_{n}\left(t\right)\max\left\{ a_{n}\left(t\right)^{1+\alpha},b_{n}\left(t\right)^{1+\alpha}\right\} \sum^{n-1}_{m=1}YC^{\delta\left(\frac{1}{1-\gamma}\right)^{m}}+B_{n}\left(t\right)\max\left\{ a_{n}\left(t\right),b_{n}\left(t\right)\right\} \cdot\\
 & \quad\cdot\sum^{n-1}_{m=1}\left|\left|\left(\phi^{\left(m\right)}\right)^{-1}\left(t,\cdot\right)\right|\right|^{\alpha}_{\dot{C}^{1}\left(\mathbb{R}^{2};\mathbb{R}^{2}\right)}YC^{\left(\alpha\left(1+k_{\max}\right)+\delta+3\mu\right)\left(\frac{1}{1-\gamma}\right)^{m}}\lesssim_{\delta}\\
 & \lesssim_{\varphi,\delta}B_{n}\left(t\right)\max\left\{ a_{n}\left(t\right)^{1+\alpha},b_{n}\left(t\right)^{1+\alpha}\right\} YC^{\delta\left(\frac{1}{1-\gamma}\right)^{n-1}}+\\
 & \quad+B_{n}\left(t\right)\max\left\{ a_{n}\left(t\right),b_{n}\left(t\right)\right\} \sum^{n-1}_{m=1}\max\left\{ a_{m}\left(t\right)^{\alpha},b_{m}\left(t\right)^{\alpha}\right\} YC^{\left(\alpha\left(1+k_{\max}\right)+\delta+3\mu\right)\left(\frac{1}{1-\gamma}\right)^{m}}\lesssim\\
 & \lesssim_{\varphi,\delta}\frac{2M_{n}}{a_{n}\left(t\right)^{2}+b_{n}\left(t\right)^{2}}b_{n}\left(t\right)^{1+\alpha}YC^{\left(4\mu+\delta\right)\left(\frac{1}{1-\gamma}\right)^{n}}+\\
 & \quad+\frac{2M_{n}}{a_{n}\left(t\right)^{2}+b_{n}\left(t\right)^{2}}b_{n}\left(t\right)C^{2\mu\left(\frac{1}{1-\gamma}\right)^{n}}\sum^{n-1}_{m=1}YC^{\left[2\alpha\left(1+k_{\max}\right)+4\mu+\delta\right]\left(\frac{1}{1-\gamma}\right)^{m}}\lesssim\\
 & \lesssim_{\varphi,\delta}Y^{2}\frac{C^{\left(4\mu+2\delta\right)\left(\frac{1}{1-\gamma}\right)^{n}}}{b_{n}\left(t\right)^{1-\alpha}}+Y^{2}\frac{C^{\left[\delta+2\mu+\left(2\alpha\left(1+k_{\max}\right)+4\mu+\delta\right)\left(1-\gamma\right)\right]\left(\frac{1}{1-\gamma}\right)^{n}}}{b_{n}\left(t\right)}\lesssim\\
 & \lesssim_{\varphi,\delta}Y^{2}C^{\left[-\left(1-\alpha\right)+5\mu+2\delta\right]\left(\frac{1}{1-\gamma}\right)^{n}}+Y^{2}C^{\left[-1+\delta+3\mu+\left(2\alpha\left(1+k_{\max}\right)+4\mu+\delta\right)\left(1-\gamma\right)\right]\left(\frac{1}{1-\gamma}\right)^{n}}.
\end{aligned}
\label{eq:U bound PI Calpha}
\end{equation}

Finally, we shall bound the transport term
\[
\widetilde{T^{\left(n\right)}_{u}}^{n}\left(t,x\right)\coloneqq\left(\widetilde{U^{\left(n-1\right)}}^{n}\left(t,x\right)-\widetilde{U^{\left(n-1\right)}}^{n}\left(t,0\right)-\mathrm{J}\widetilde{U^{\left(n-1\right)}}^{n}\left(t,0\right)\cdot\left(\begin{array}{c}
x_{1}\\
x_{2}
\end{array}\right)\right)\cdot\tilde{\nabla}^{n}\widetilde{u^{\left(n\right)}}^{n}\left(t,x\right).
\]
 To accomplish this, we shall make use of Propositions \ref{prop:bounds gradu}
and \ref{prop:bounds for Taylor development of transport}. Clearly,
\[
\begin{aligned}\left|\left|T^{\left(n\right)}_{u}\left(t,\cdot\right)\right|\right|_{L^{\infty}\left(\mathbb{R}^{2};\mathbb{R}^{2}\right)} & \lesssim_{\delta,\varphi}YC^{\left[-2\left(1-\Lambda-k_{\max}\right)+2\mu+\left(1+\delta\right)\left(1-\gamma\right)\right]\left(\frac{1}{1-\gamma}\right)^{n}}YC^{\delta\left(\frac{1}{1-\gamma}\right)^{n}}=\\
 & \lesssim_{\delta,\varphi}Y^{2}C^{\left[-2\left(1-\Lambda-k_{\max}\right)+2\mu+2\delta+\left(1-\gamma\right)\right]\left(\frac{1}{1-\gamma}\right)^{n}}.
\end{aligned}
\]
Moreover, applying equation \eqref{eq:property Calpha multiplication},
Propositions \ref{prop:bounds gradu} and \ref{prop:bounds for Taylor development of transport},
we obtain
\[
\begin{aligned}\left|\left|T^{\left(n\right)}_{u}\left(t,\cdot\right)\right|\right|_{\dot{C}^{\alpha}\left(\mathbb{R}^{2};\mathbb{R}^{2}\right)} & \lesssim_{\delta,\varphi}YC^{\left[-2\left(1-\Lambda-k_{\max}\right)+2\mu+\left(1+\delta\right)\left(1-\gamma\right)\right]\left(\frac{1}{1-\gamma}\right)^{n}}YC^{\left(\alpha\left(1+k_{\max}\right)+\delta+\mu\right)\left(\frac{1}{1-\gamma}\right)^{n}}+\\
 & \quad+YC^{\left[-2\left(1-\Lambda-k_{\max}\right)+\max\left\{ \alpha\left(1-k_{\max}\right)-\Lambda,0\right\} +2\mu+\left(2+\delta\right)\left(1-\gamma\right)\right]\left(\frac{1}{1-\gamma}\right)^{n}}YC^{\delta\left(\frac{1}{1-\gamma}\right)^{n}}.
\end{aligned}
\]
Using equations \eqref{eq:Boussinesq identity Lambda kmax} and the
fact that $\alpha\leq1$, we may bound
\begin{equation}
\begin{aligned}\left|\left|T^{\left(n\right)}_{u}\left(t,\cdot\right)\right|\right|_{L^{\infty}\left(\mathbb{R}^{2};\mathbb{R}^{2}\right)} & \lesssim_{\delta,\varphi}Y^{2}C^{\left[-\frac{4}{3}+2\mu+2\delta+\left(1-\gamma\right)\right]\left(\frac{1}{1-\gamma}\right)^{n}},\\
\left|\left|T^{\left(n\right)}_{u}\left(t,\cdot\right)\right|\right|_{\dot{C}^{\alpha}\left(\mathbb{R}^{2};\mathbb{R}^{2}\right)} & \lesssim_{\delta,\varphi}Y^{2}C^{\left[2\Lambda+3k_{\max}-1+3\mu+2\delta+\left(1-\gamma\right)\right]\left(\frac{1}{1-\gamma}\right)^{n}}+\\
 & \quad+Y^{2}C^{\left[\Lambda+k_{\max}-1+2\mu+2\delta+2\left(1-\gamma\right)\right]\left(\frac{1}{1-\gamma}\right)^{n}}\leq\\
 & \lesssim_{\delta,\varphi}Y^{2}C^{\left[-\Lambda+3\mu+2\delta+\left(1-\gamma\right)\right]\left(\frac{1}{1-\gamma}\right)^{n}}+Y^{2}C^{\left[-\frac{2}{3}+2\mu+2\delta+2\left(1-\gamma\right)\right]\left(\frac{1}{1-\gamma}\right)^{n}}.
\end{aligned}
\label{eq:U bound transport}
\end{equation}

Combining equations \eqref{eq:U bound quadratic Linfty}, \eqref{eq:U bound quadratic Calpha},
\eqref{eq:U bound dudt}, \eqref{eq:U bound PI Lifnty}, \eqref{eq:U bound PI Calpha}
and \eqref{eq:U bound transport}, we deduce that
\[
\sup_{t\in\left[0,1\right]}\left|\left|\mathcal{U}^{\left(n\right)}\left(t,\cdot\right)-\mathcal{U}^{\left(n-1\right)}\left(t,\cdot\right)\right|\right|_{C^{\alpha}\left(\mathbb{R}^{2}\right)}
\]
decreases superexponentially in $n\in\mathbb{N}$ as long as $\mu$
and $\delta$ are chosen sufficiently small with respect to $1-\alpha$
and provided that $\gamma$ is close enough to one (the choice of
$\gamma$ can be made independent of $\alpha$). Next, we will show
that $\mathcal{U}^{\left(n\right)}$ is convergent in $C^{0}_{t}C^{\alpha}_{x}\left(\left[0,1\right]\times\mathbb{R}^{2}\right)$.
As $C^{0}_{t}C^{\alpha}_{x}\left(\left[0,1\right]\times\mathbb{R}^{2}\right)$
is complete, it will be enough to show that it is a Cauchy sequence.
First of all, notice that, $\forall n\in\mathbb{N}$, $\mathcal{U}^{\left(n\right)}$
is nothing but a finite sum of compositions of smooth functions in
space and time. Consequently, $\mathcal{U}^{\left(n\right)}\in C^{0}_{t}C^{\alpha}_{x}\left(\left[0,1\right]\times\mathbb{R}^{2}\right)$
$\forall n\in\mathbb{N}$. Moreover, by the telescopic nature of $\mathcal{U}^{\left(n\right)}$,
$\forall l,m\in\mathbb{N}$ with $l\ge m$, we have
\[
\begin{aligned}\left|\left|\mathcal{U}^{\left(l\right)}-\mathcal{U}^{\left(m\right)}\right|\right|_{C^{0}_{t}C^{\alpha}_{x}\left(\left[0,1\right]\times\mathbb{R}^{2}\right)} & =\left|\left|\sum^{l}_{n=m+1}\left(\mathcal{U}^{\left(n\right)}-\mathcal{U}^{\left(n-1\right)}\right)\right|\right|_{C^{0}_{t}C^{\alpha}_{x}\left(\left[0,1\right]\times\mathbb{R}^{2}\right)}\leq\\
 & \leq\sum^{\infty}_{n=m+1}\underbrace{\sup_{t\in\left[0,1\right]}\left|\left|\mathcal{U}^{\left(n\right)}\left(t,\cdot\right)-\mathcal{U}^{\left(n-1\right)}\left(t,\cdot\right)\right|\right|_{C^{\alpha}\left(\mathbb{R}^{2}\right)}}_{\left(*\right)}
\end{aligned}
\]
and we have seen that the marked term is superexponentially decreasing
in $n\in\mathbb{N}$. In this way, in the bound above, we have the
tail of a convergent series. Consequently, $\left|\left|\mathcal{U}^{\left(l\right)}-\mathcal{U}^{\left(m\right)}\right|\right|_{C^{0}_{t}C^{\alpha}_{x}\left(\left[0,1\right]\times\mathbb{R}^{2}\right)}\xrightarrow[m\to\infty]{}0$
and, therefore, $\left(\mathcal{U}^{\left(n\right)}\right)_{n\in\mathbb{N}}$
is a Cauchy sequence in $C^{0}_{t}C^{\alpha}_{x}\left(\left[0,1\right]\times\mathbb{R}^{2}\right)$.

With the above, we have already proven that $\left(\frac{\partial u}{\partial t}+u\cdot\nabla u\right)\left(t,\cdot\right)$
admits an extension $\mathcal{U}$ to time $t=1$ such that $\frac{\partial u}{\partial t}+u\cdot\nabla u\in C^{0}_{t}C^{\alpha}_{x}\left(\left[0,1\right]\times\mathbb{R}^{2}\right)$.
We just need to prove the uniqueness. Let $\mathcal{V}\left(t,x\right)$
be any other extension, i.e., such that $\mathcal{V}\left(t,\cdot\right)=\left(\frac{\partial u}{\partial t}+u\cdot\nabla u\right)\left(t,\cdot\right)=\mathcal{U}\left(t,\cdot\right)$
$\forall t\in\left[0,1\right)$. Then, by the uniqueness of the limit
in $C^{\alpha}\left(\mathbb{R}^{2}\right)$, we conclude that $\lim_{t\to1^{-}}\mathcal{V}\left(t,\cdot\right)=\lim_{t\to1^{-}}\mathcal{U}\left(t,\cdot\right)$
in $C^{\alpha}\left(\mathbb{R}^{2}\right)$. Consequently, $\mathcal{U}=\mathcal{V}$
in $C^{0}_{t}C^{\alpha}_{x}\left(\left[0,1\right]\times\mathbb{R}^{2}\right)$.
\end{proof}

\begin{prop}
\label{prop:oddness and evenness U}Let $\left(\rho,\omega,f_{\rho},f_{\omega}\right)$
be the solution of the forced Boussinesq system presented in \cite{Articulo Boussinesq}.
$u_{2}\left(t,x\right)$ is odd in $x_{2}$ and $u_{1}\left(t,x\right)$
is even in $x_{2}$. Furthermore, $\left(\frac{\partial u}{\partial t}+u\cdot\nabla u\right)_{2}\left(t,x\right)$
is odd in $x_{2}$ and $\left(\frac{\partial u}{\partial t}+u\cdot\nabla u\right)_{1}\left(t,x\right)$
is even in $x_{2}$.
\end{prop}
\begin{proof}
Looking at the expression for the velocity given in Proposition \ref{prop:computations vorticity},
it is clear that $\widetilde{u^{\left(n\right)}_{2}}^{n}\left(t,x\right)$
is odd in $x_{2}$ and that $\widetilde{u^{\left(n\right)}_{1}}^{n}\left(t,x\right)$
is even in $x_{2}$. Since the change of variables $\phi^{\left(n\right)}_{2}\left(t,x\right)$
is linear in $x_{2}$ (see equation \eqref{eq:Boussinesq phi}), $u^{\left(n\right)}_{2}\left(t,x\right)$
must also be odd in $x_{2}$ and $u^{\left(n\right)}_{1}\left(t,x\right)$
must be even in $x_{2}$. As the sum of odd functions is odd, the
limit of odd functions is odd, the sum of even functions is even and
the limit of even functions is even, we conclude that $u_{2}\left(t,x\right)$
is odd in $x_{2}$ and that $u_{1}\left(t,x\right)$ is even in $x_{2}$.
Then, $\frac{\partial u_{2}}{\partial t}\left(t,x\right)$, $\frac{\partial u_{2}}{\partial x_{1}}\left(t,x\right)$
and $\frac{\partial u_{1}}{\partial x_{2}}\left(t,x\right)$ must
also be odd in $x_{2}$, while $\frac{\partial u_{2}}{\partial x_{2}}\left(t,x\right)$,
$\frac{\partial u_{1}}{\partial t}\left(t,x\right)$ and $\frac{\partial u_{1}}{\partial x_{1}}\left(t,x\right)$
will be even in $x_{2}$. In this way,
\[
\left(\frac{\partial u}{\partial t}+u\cdot\nabla u\right)_{2}\left(t,x\right)=\underbrace{\frac{\partial u_{2}}{\partial t}\left(t,x\right)}_{\text{odd in }x_{2}}+\underbrace{u_{1}\left(t,x\right)}_{\text{even in }x_{2}}\underbrace{\frac{\partial u_{2}}{\partial x_{1}}\left(t,x\right)}_{\text{odd in }x_{2}}+\underbrace{u_{2}\left(t,x\right)}_{\text{odd in }x_{2}}\underbrace{\frac{\partial u_{2}}{\partial x_{2}}\left(t,x\right)}_{\text{even in }x_{2}}
\]
will be odd in $x_{2}$, whereas
\[
\left(\frac{\partial u}{\partial t}+u\cdot\nabla u\right)_{1}\left(t,x\right)=\underbrace{\frac{\partial u_{1}}{\partial t}\left(t,x\right)}_{\text{even in }x_{2}}+\underbrace{u_{1}\left(t,x\right)}_{\text{even in }x_{2}}\underbrace{\frac{\partial u_{1}}{\partial x_{1}}\left(t,x\right)}_{\text{even in }x_{2}}+\underbrace{u_{2}\left(t,x\right)}_{\text{odd in }x_{2}}\underbrace{\frac{\partial u_{1}}{\partial x_{2}}\left(t,x\right)}_{\text{odd in }x_{2}}
\]
will be even in $x_{2}$.
\end{proof}

\subsection{Bounds for the bad terms}

In this subsection, we will prove that, even though $\frac{\partial\rho}{\partial x_{2}}$
and $\frac{\partial\rho}{\partial x_{1}}$ blow up in $\left|\left|\cdot\right|\right|_{\dot{C}^{\alpha}\left(\mathbb{R}^{2}\right)}$
(for high enough $\alpha$ in the case of $\frac{\partial\rho}{\partial x_{1}}$),
they retain some $L^{p}\left(\mathbb{R}^{2}\right)$ regularity at
the time of the blow-up. This is why we refer to these terms as ``bad''
terms.
\begin{prop}[Bounds for $\frac{\partial\rho}{\partial x_{2}}\left(t,\cdot\right)$]
\label{prop:bad bound drhodx2}$\forall\alpha\in\left(0,1\right)$,
uniformly in time $t\in\left[0,1\right]$, we have
\[
\begin{aligned}\left|\left|\frac{\partial\rho}{\partial x_{2}}\left(t,\cdot\right)\right|\right|_{L^{\infty}\left(\mathbb{R}^{2}\setminus B\left(0;r\right)\right)} & \lesssim_{\delta,\varphi,\mu,Y}\frac{1}{r^{\frac{2\delta+2\mu}{\frac{2}{3}-\mu}}},\\
\left|\left|\frac{\partial\rho}{\partial x_{2}}\left(t,\cdot\right)\right|\right|_{\dot{C}^{\alpha}\left(\mathbb{R}^{2}\setminus B\left(0;r\right)\right)} & \lesssim_{\delta,\varphi,\mu,Y}\frac{1}{r^{\frac{\alpha\left(1+k_{\max}\right)+2\delta+3\mu}{\frac{2}{3}-\mu}}}.
\end{aligned}
\]

As a consequence, as long as $\delta$ and $\mu$ are taken sufficiently
small (the smallness depends on the value of $p_{0}$), we can ensure
that $\frac{\partial\rho}{\partial x_{2}}\in C^{0}_{t}L^{p}_{x}\left(\left[0,1\right]\times\mathbb{R}^{2}\right)$
$\forall p\in\left[1,p_{0}\right]$. Furthermore, $\frac{\partial\rho}{\partial x_{2}}\left(1,\cdot\right)\equiv0$
in $L^{p}\left(\mathbb{R}^{2}\right)$ $\forall p\in\left[1,p_{0}\right]$.
\end{prop}
\begin{proof}
Eliminating a ball of radius $r$ around the blow-up point amounts
to keeping a finite number of layers. Since, by equations \eqref{eq:Boussinesq psi rho} and \eqref{eq:Boussinesq phi},
every density layer $\rho^{\left(n\right)}$ has support
\[
\text{supp}\rho^{\left(n\right)}\left(t,\cdot\right)\subseteq\left[\phi^{\left(n\right)}_{1}\left(t,0\right)-\frac{16\pi}{\lambda_{n}a_{n}\left(t\right)},\phi^{\left(n\right)}_{1}\left(t,0\right)+\frac{16\pi}{\lambda_{n}a_{n}\left(t\right)}\right]\times\left[-\frac{16\pi}{\lambda_{n}b_{n}\left(t\right)},\frac{16\pi}{\lambda_{n}b_{n}\left(t\right)}\right],
\]
using equations \eqref{eq:Boussinesq lambda_n} and \eqref{eq:Boussinesq def kmax}
and Corollary \ref{cor:bound for an(t). bn(t)}, we may bound 
\[
\begin{aligned}\text{supp}\rho^{\left(n\right)}\left(t,\cdot\right) & \subseteq\left[\phi^{\left(n\right)}_{1}\left(t,0\right)-\frac{16\pi}{C^{\left(1-\Lambda-k_{\max}-\mu\right)\left(\frac{1}{1-\gamma}\right)^{n}}},\phi^{\left(n\right)}_{1}\left(t,0\right)+\frac{16\pi}{C^{\left(1-\Lambda-k_{\max}-\mu\right)\left(\frac{1}{1-\gamma}\right)^{n}}}\right]\times\\
 & \quad\times\left[-\frac{16\pi}{C^{\left(1-\Lambda-\mu\right)\left(\frac{1}{1-\gamma}\right)^{n}}},\frac{16\pi}{C^{\left(1-\Lambda-\mu\right)\left(\frac{1}{1-\gamma}\right)^{n}}}\right].
\end{aligned}
\]
Using Lemma \ref{lem:blow-up point} and Choice \ref{choice:blow-up origin},
we can write
\[
\begin{aligned}\text{supp}\rho^{\left(n\right)}\left(t,\cdot\right) & \subseteq\left[-\frac{16\pi}{C^{\left(1-\Lambda-k_{\max}-\mu\right)\left(\frac{1}{1-\gamma}\right)^{n}}}-8\pi C^{-\left(\frac{1}{1-\gamma}\right)^{n}},\frac{16\pi}{C^{\left(1-\Lambda-k_{\max}-\mu\right)\left(\frac{1}{1-\gamma}\right)^{n}}}+8\pi C^{-\left(\frac{1}{1-\gamma}\right)^{n}}\right]\times\\
 & \quad\times\left[-\frac{16\pi}{C^{\left(1-\Lambda-\mu\right)\left(\frac{1}{1-\gamma}\right)^{n}}},\frac{16\pi}{C^{\left(1-\Lambda-\mu\right)\left(\frac{1}{1-\gamma}\right)^{n}}}\right].
\end{aligned}
\]
Clearly, the slowest decreasing component is the summand $\frac{16\pi}{C^{\left(1-\Lambda-k_{\max}-\mu\right)\left(\frac{1}{1-\gamma}\right)^{n}}}$.
Thus, by equation \eqref{eq:Boussinesq identity Lambda kmax}, if
we take
\begin{equation}
r\geq17\pi C^{-\left(\frac{2}{3}-\mu\right)\left(\frac{1}{1-\gamma}\right)^{n}}\label{eq:Lp choice r}
\end{equation}
we can guarantee that the support of layer $n$ does not intersect
$\mathbb{R}^{2}\setminus B\left(0;r\right)$. Then, with this condition
on $r$, we have
\begin{equation}
\left.\frac{\partial\rho}{\partial x_{2}}\left(t,\cdot\right)\right|_{\mathbb{R}^{2}\setminus B\left(0;r\right)}=\sum^{n-1}_{m=1}\frac{\partial\rho^{\left(m\right)}}{\partial x_{2}}\left(t,\cdot\right).\label{eq:Lp r reduction}
\end{equation}
On the one hand, by equations \eqref{eq:Boussinesq psi rho} and \eqref{eq:Boussinesq def nabla_tilde},
\begin{equation}
\begin{aligned}\left(0,1\right)\cdot\tilde{\nabla}^{n}\widetilde{\rho^{\left(n\right)}}^{n}\left(t,x\right) & =\frac{\mathrm{d}}{\mathrm{d}t}\left[B_{n}\left(t\right)\left(a_{n}\left(t\right)^{2}+b_{n}\left(t\right)^{2}\right)\right]\varphi\left(\lambda_{n}x_{1}\right)\varphi\left(\lambda_{n}x_{2}\right)\sin\left(x_{1}\right)\sin\left(x_{2}\right)+\\
 & \quad-\lambda_{n}\frac{\mathrm{d}}{\mathrm{d}t}\left[B_{n}\left(t\right)\left(a_{n}\left(t\right)^{2}+b_{n}\left(t\right)^{2}\right)\right]\varphi\left(\lambda_{n}x_{1}\right)\varphi'\left(\lambda_{n}x_{2}\right)\sin\left(x_{1}\right)\cos\left(x_{2}\right).
\end{aligned}
\label{eq:Lp one hand}
\end{equation}
On the other hand, by the definition of $\tilde{\nabla}^{n}$ (see
equation \eqref{eq:Boussinesq def nabla_tilde}), we know that
\begin{equation}
\begin{aligned}\left(0,1\right)\cdot\tilde{\nabla}^{n}\widetilde{\rho^{\left(n\right)}}^{n}\left(t,x\right) & =b_{n}\left(t\right)\frac{\partial\widetilde{\rho^{\left(n\right)}}^{n}}{\partial x_{2}}\left(t,x\right)=b_{n}\left(t\right)\frac{\partial}{\partial x_{2}}\left(\rho^{\left(n\right)}\left(t,\phi^{\left(n\right)}\left(t,x\right)\right)\right)=\\
 & =b_{n}\left(t\right)\frac{\partial\phi^{\left(n\right)}}{\partial x_{2}}\cdot\nabla\rho^{\left(n\right)}\left(t,\phi^{\left(n\right)}\left(t,x\right)\right)=\frac{\partial\rho^{\left(n\right)}}{\partial x_{2}}\left(t,\phi^{\left(n\right)}\left(t,x\right)\right).
\end{aligned}
\label{eq:Lp other hand}
\end{equation}
Combining equations \eqref{eq:Lp r reduction}, \eqref{eq:Lp one hand}
and \eqref{eq:Lp other hand} with equations \eqref{eq:property Calpha multiplication}
and \eqref{eq:property Calpha composition}, equations \eqref{eq:Boussinesq psi rho}, \eqref{eq:Boussinesq Bn(t)}, \eqref{eq:Boussinesq an(t) bn(t)}, \eqref{eq:Boussinesq Mn}, and \eqref{eq:Boussinesq jacobian inverse},
we deduce that
\[
\begin{aligned}\left|\left|\frac{\partial\rho}{\partial x_{2}}\left(t,\cdot\right)\right|\right|_{L^{\infty}\left(\mathbb{R}^{2}\setminus B\left(0;r\right)\right)} & \lesssim_{\delta,\varphi}\sum^{n-1}_{m=1}YC^{\delta\left(\frac{1}{1-\gamma}\right)^{m}}\frac{h^{\left(m\right)}\left(t\right)b_{m}\left(t\right)}{\int^{1}_{t_{m}}h^{\left(m\right)}\left(s\right)b_{m}\left(s\right)\mathrm{d}s},\\
\left|\left|\frac{\partial\rho}{\partial x_{2}}\left(t,\cdot\right)\right|\right|_{\dot{C}^{\alpha}\left(\mathbb{R}^{2}\setminus B\left(0;r\right)\right)} & \lesssim_{\delta,\varphi}\sum^{n-1}_{m=1}YC^{\delta\left(\frac{1}{1-\gamma}\right)^{m}}\max\left\{ a_{m}\left(t\right),b_{m}\left(t\right)\right\} ^{\alpha}\frac{h^{\left(m\right)}\left(t\right)b_{m}\left(t\right)}{\int^{1}_{t_{m}}h^{\left(m\right)}\left(s\right)b_{m}\left(s\right)\mathrm{d}s}.
\end{aligned}
\]
Lemma \ref{lem:estimate of integral hn bn}, equations \eqref{eq:Boussinesq an(t) bn(t)} and \eqref{eq:Boussinesq def kmax}
and Corollary \ref{cor:bound for an(t). bn(t)} provide
\[
\begin{aligned}\left|\left|\frac{\partial\rho}{\partial x_{2}}\left(t,\cdot\right)\right|\right|_{L^{\infty}\left(\mathbb{R}^{2}\setminus B\left(0;r\right)\right)} & \lesssim_{\delta,\varphi,\mu}\sum^{n-1}_{m=1}Y^{2}C^{\delta\left(\frac{1}{1-\gamma}\right)^{m}}\frac{C^{\left(1+k_{\max}+\mu\right)\left(\frac{1}{1-\gamma}\right)^{m}}}{C^{\left(1+k_{\max}-\mu-\delta\left(1-\gamma\right)\right)\left(\frac{1}{1-\gamma}\right)^{m}}},\\
\left|\left|\frac{\partial\rho}{\partial x_{2}}\left(t,\cdot\right)\right|\right|_{\dot{C}^{\alpha}\left(\mathbb{R}^{2}\setminus B\left(0;r\right)\right)} & \lesssim_{\delta,\varphi,\mu}\sum^{n-1}_{m=1}Y^{2}C^{\delta\left(\frac{1}{1-\gamma}\right)^{m}}\frac{C^{\left(1+\alpha\right)\left(1+k_{\max}+\mu\right)\left(\frac{1}{1-\gamma}\right)^{m}}}{C^{\left(1+k_{\max}-\mu-\delta\left(1-\gamma\right)\right)\left(\frac{1}{1-\gamma}\right)^{m}}}.
\end{aligned}
\]
Finally, Lemma \ref{lem:estimate sum superexponential} lets us write
\[
\begin{aligned}\left|\left|\frac{\partial\rho}{\partial x_{2}}\left(t,\cdot\right)\right|\right|_{L^{\infty}\left(\mathbb{R}^{2}\setminus B\left(0;r\right)\right)} & \lesssim_{\delta,\varphi,\mu}Y^{2}C^{\left(2\mu+2\delta\right)\left(\frac{1}{1-\gamma}\right)^{n-1}},\\
\left|\left|\frac{\partial\rho}{\partial x_{2}}\left(t,\cdot\right)\right|\right|_{\dot{C}^{\alpha}\left(\mathbb{R}^{2}\setminus B\left(0;r\right)\right)} & \lesssim_{\delta,\varphi,\mu}Y^{2}C^{\left(\alpha\left(1+k_{\max}\right)+2\delta+3\mu\right)\left(\frac{1}{1-\gamma}\right)^{n-1}}.
\end{aligned}
\]
From equation \eqref{eq:Lp choice r} follows that the bounds above
are valid $\forall r\geq17\pi C^{-\left(\frac{2}{3}-\mu\right)\left(\frac{1}{1-\gamma}\right)^{n}}$.
Now, for which values of $r$ is the bound above optimal in $n\in\mathbb{N}$?
Again, by equation \eqref{eq:Lp choice r}, we know that, if we had
$r\geq17\pi C^{-\left(\frac{2}{3}-\mu\right)\left(\frac{1}{1-\gamma}\right)^{n-1}}$,
we could change $n-1$ by $n-2$ in the bounds above, so the optimal
interval for the bounds above is
\[
17\pi C^{-\left(\frac{2}{3}-\mu\right)\left(\frac{1}{1-\gamma}\right)^{n}}\leq r\leq17\pi C^{-\left(\frac{2}{3}-\mu\right)\left(\frac{1}{1-\gamma}\right)^{n-1}}.
\]
Thereby, we can write $C^{\left(\frac{1}{1-\gamma}\right)^{n-1}}$
in terms of $r$ as follows
\begin{equation}
r\leq17\pi\frac{1}{\left[C^{\left(\frac{1}{1-\gamma}\right)^{n-1}}\right]^{\left(\frac{2}{3}-\mu\right)}}\iff C^{\left(\frac{1}{1-\gamma}\right)^{n-1}}\leq\left(17\pi\right)^{\frac{1}{\frac{2}{3}-\mu}}\frac{1}{r^{\frac{1}{\frac{2}{3}-\mu}}},\label{eq:relation r C}
\end{equation}
leading to
\[
\begin{aligned}\left|\left|\frac{\partial\rho}{\partial x_{2}}\left(t,\cdot\right)\right|\right|_{L^{\infty}\left(\mathbb{R}^{2}\setminus B\left(0;r\right)\right)} & \lesssim_{\delta,\varphi,\mu,Y}\frac{1}{r^{\frac{2\delta+2\mu}{\frac{2}{3}-\mu}}},\\
\left|\left|\frac{\partial\rho}{\partial x_{2}}\left(t,\cdot\right)\right|\right|_{\dot{C}^{\alpha}\left(\mathbb{R}^{2}\setminus B\left(0;r\right)\right)} & \lesssim_{\delta,\varphi,\mu,Y}\frac{1}{r^{\frac{\alpha\left(1+k_{\max}\right)+2\delta+3\mu}{\frac{2}{3}-\mu}}}.
\end{aligned}
\]

Let us proceed with the $L^{p}\left(\mathbb{R}^{2}\right)$ statement.
First of all, notice that $\forall t\in\left[0,1\right)$, $\frac{\partial\rho}{\partial x_{2}}\left(t,\cdot\right)$
is nothing but a finite sum of compositions of smooth functions in
space and time. Thus, $\frac{\partial\rho}{\partial x_{2}}\in C^{0}_{t}L^{p}_{x}\left(\left[0,1-\varepsilon\right]\times\mathbb{R}^{2}\right)$
$\forall\varepsilon>0$. Let us study $\left|\left|\frac{\partial\rho}{\partial x_{2}}\left(t,\cdot\right)\right|\right|^{p}_{L^{p}\left(\mathbb{R}^{2}\right)}$
as $t\to1$. To do this, we first have to compute how the support
of $\rho^{\left(n\right)}$ diminishes in time. For this, let us try
to solve for $C^{-\left(\frac{1}{1-\gamma}\right)^{n}}$ as a function
of $1-t_{n+1}$ using equation \eqref{eq:Boussinesq time scale}:
\[
1-t_{n+1}=\frac{1}{Y}C^{-\delta\left(\frac{1}{1-\gamma}\right)^{n}}\mathrm{arccosh}\left(C^{k_{\max}\left(\frac{1}{1-\gamma}\right)^{n+1}}\right)\iff C^{-\delta\left(\frac{1}{1-\gamma}\right)^{n}}=\frac{Y\left(1-t_{n+1}\right)}{\mathrm{arccosh}\left(C^{k_{\max}\left(\frac{1}{1-\gamma}\right)^{n+1}}\right)}.
\]
Since $k_{\max}\ge\frac{1}{100}$ and $1-\gamma\leq\frac{1}{2}$ by
subsection \ref{subsec:summary Boussinesq}, $n\ge1$ and $\text{arccosh}$
is an increasing function, we may bound
\[
C^{-\delta\left(\frac{1}{1-\gamma}\right)^{n}}\lesssim_{Y}\left(1-t_{n+1}\right)\implies C^{-\left(\frac{1}{1-\gamma}\right)^{n}}\lesssim_{Y,\delta}\left(1-t_{n+1}\right)^{\frac{1}{\delta}}.
\]
Recall that, from our choice of $r$ in \eqref{eq:Lp choice r}, we
know that
\[
\text{supp}\rho^{\left(n\right)}\subseteq B\left(0;r\right),
\]
and we can bound this $r$ by
\[
r\lesssim C^{-\left(\frac{2}{3}-\mu\right)\left(\frac{1}{1-\gamma}\right)^{n}}\lesssim_{Y,\delta}\left(1-t_{n+1}\right)^{\frac{1}{\delta}\left(\frac{2}{3}-\mu\right)}.
\]
Since there is only one active density layer for each given time (see
item \ref{item:Boussinesq active layer} of the summary), we may ensure
that, $\forall t\in\left[t_{n},t_{n+1}\right]$,
\[
\text{supp}\rho\left(t,\cdot\right)\subseteq B\left(0;r\right),\quad r\lesssim_{Y,\delta}\left(1-t_{n+1}\right)^{\frac{1}{\delta}\left(\frac{2}{3}-\mu\right)}\leq\left(1-t\right)^{\frac{1}{\delta}\left(\frac{2}{3}-\mu\right)}.
\]
Thus, actually
\begin{equation}
\text{supp}\rho\left(t,\cdot\right)\subseteq B\left(0;r\left(t\right)\right),\quad r\left(t\right)\lesssim_{Y,\delta}\left(1-t\right)^{\frac{1}{\delta}\left(\frac{2}{3}-\mu\right)}\quad\forall t\in\left[0,1\right].\label{eq:time-bound r}
\end{equation}
Thereby, the first part of the statement shows that
\[
\begin{aligned}\left|\left|\frac{\partial\rho}{\partial x_{2}}\left(t,\cdot\right)\right|\right|^{p}_{L^{p}\left(\mathbb{R}^{2}\right)} & =\int_{\mathbb{R}^{2}}\left|\frac{\partial\rho}{\partial x_{2}}\left(t,x\right)\right|^{p}\mathrm{d}x=\int_{\text{supp}\rho\left(t,\cdot\right)}\left|\frac{\partial\rho}{\partial x_{2}}\left(t,x\right)\right|^{p}\mathrm{d}x\lesssim_{\delta,\varphi,\mu,Y}\\
 & \lesssim_{\delta,\varphi,\mu,Y}\int_{\text{supp}\rho\left(t,\cdot\right)}\frac{1}{\left|\left|x\right|\right|^{p\frac{2\delta+2\mu}{\frac{2}{3}-\mu}}_{2}}\mathrm{d}x.
\end{aligned}
\]
Changing into polar coordinates, we arrive at
\[
\left|\left|\frac{\partial\rho}{\partial x_{2}}\left(t,\cdot\right)\right|\right|^{p}_{L^{p}\left(\mathbb{R}^{2}\right)}\lesssim_{\delta,\varphi,\mu,Y}\int^{r\left(t\right)}_{0}R^{1-p\frac{2\delta+2\mu}{\frac{2}{3}-\mu}}\mathrm{d}R=\frac{r\left(t\right)^{2-p\frac{2\delta+2\mu}{\frac{2}{3}-\mu}}}{2-p\frac{2\delta+2\mu}{\frac{2}{3}-\mu}}
\]
provided that $2-p\frac{\left(2\delta+2\mu\right)}{\frac{2}{3}-\mu}>0$.
It is clear that we can take $\delta$ and $\mu$ sufficiently small
(depending on $p_{0}$) so that $2-p\frac{\left(2\delta+2\mu\right)}{\frac{2}{3}-\mu}>0$
$\forall p\in\left[1,p_{0}\right]$. In view of equation \eqref{eq:time-bound r},
this ensures that
\[
\left|\left|\frac{\partial\rho}{\partial x_{2}}\left(t,\cdot\right)\right|\right|^{p}_{L^{p}\left(\mathbb{R}^{2}\right)}\lesssim_{\delta,\varphi,\mu,Y}r\left(t\right)^{\beta}\lesssim_{Y,\delta}\left(1-t\right)^{\frac{1}{\delta}\left(\frac{2}{3}-\mu\right)\beta}
\]
for some $\beta>0$. In this manner,
\[
\left|\left|\frac{\partial\rho}{\partial x_{2}}\left(t,\cdot\right)\right|\right|_{L^{p}\left(\mathbb{R}^{2}\right)}\xrightarrow[t\to1^{-}]{}0\quad\forall p\in\left[1,p_{0}\right].
\]
This limit simultaneously proves that $\frac{\partial\rho}{\partial x_{2}}\left(1,\cdot\right)\equiv0$
in $L^{p}\left(\mathbb{R}^{2}\right)$ and that $\frac{\partial\rho}{\partial x_{2}}:\left[0,1\right]\to L^{p}\left(\mathbb{R}^{2}\right)$
is continuous at $t=1$. Hence, we may finally conclude that $\frac{\partial\rho}{\partial x_{2}}\in C^{0}_{t}L^{p}_{x}\left(\left[0,1\right]\times\mathbb{R}^{2}\right)$
$\forall p\in\left[1,p_{0}\right]$.
\end{proof}

\begin{prop}
\label{prop:bad bound drhodx1}Let $\alpha\in\left[\frac{\alpha_{*}}{1+\alpha_{*}},1\right)$.
Uniformly in time $t\in\left[0,1\right]$, we have
\[
\left|\left|\frac{\partial\rho}{\partial x_{1}}\left(t,\cdot\right)\right|\right|_{\dot{C}^{\alpha}\left(\mathbb{R}^{2}\setminus B\left(0;r\right)\right)}\lesssim_{\varphi,\mu,Y,\delta}\frac{1}{r^{\frac{-k_{\max}+\alpha\left(1+k_{\max}\right)+2\delta+3\mu}{\frac{2}{3}-\mu}}}.
\]
\end{prop}
\begin{proof}
As we argued in the proof of Proposition \ref{prop:bad bound drhodx2},
the $n$-th density layer is contained in the ball $B\left(0;r\right)$
with
\[
r\geq17\pi C^{-\left(\frac{2}{3}-\mu\right)\left(\frac{1}{1-\gamma}\right)^{n}}.
\]
Then, with this condition on $r$, we have
\[
\left.\frac{\partial\rho}{\partial x_{1}}\left(t,\cdot\right)\right|_{\mathbb{R}^{2}\setminus B\left(0;r\right)}=\sum^{n-1}_{m=1}\frac{\partial\rho^{\left(m\right)}}{\partial x_{1}}\left(t,\cdot\right).
\]
Making use of equation \eqref{eq:bound drhodx1 n}, we arrive at
\[
\left|\left|\frac{\partial\rho}{\partial x_{1}}\left(t,\cdot\right)\right|\right|_{\dot{C}^{\alpha}\left(\mathbb{R}^{2}\setminus B\left(0;r\right)\right)}\lesssim_{\varphi,\mu,Y}\sum^{n-1}_{m=1}C^{\left(-k_{\max}+\alpha\left(1+k_{\max}\right)+2\delta+3\mu\right)\left(\frac{1}{1-\gamma}\right)^{m}}.
\]
Since $\alpha\geq\frac{\alpha_{*}}{1+\alpha_{*}}$ and $k_{\max}=\alpha_{*}$,
we have
\[
-k_{\max}+\alpha\left(1+k_{\max}\right)+2\delta+3\mu\geq2\delta+3\mu>0.
\]
Then, Lemma \ref{lem:estimate sum superexponential} ensures that
\[
\left|\left|\frac{\partial\rho}{\partial x_{1}}\left(t,\cdot\right)\right|\right|_{\dot{C}^{\alpha}\left(\mathbb{R}^{2}\setminus B\left(0;r\right)\right)}\lesssim_{\varphi,\mu,Y,\delta}C^{\left(-k_{\max}+\alpha\left(1+k_{\max}\right)+2\delta+3\mu\right)\left(\frac{1}{1-\gamma}\right)^{n-1}}.
\]
Using the relationship between $r$ and $C^{\left(\frac{1}{1-\gamma}\right)^{n-1}}$
that we found in the proof of Proposition \ref{prop:bad bound drhodx2}
(see equation \eqref{eq:relation r C}), we infer that
\[
\left|\left|\frac{\partial\rho}{\partial x_{1}}\left(t,\cdot\right)\right|\right|_{\dot{C}^{\alpha}\left(\mathbb{R}^{2}\setminus B\left(0;r\right)\right)}\lesssim_{\varphi,\mu,Y,\delta}\frac{1}{r^{\frac{-k_{\max}+\alpha\left(1+k_{\max}\right)+2\delta+3\mu}{\frac{2}{3}-\mu}}}.
\]
\end{proof}

\subsection{Lower bounds for the blow-up}

In Lemma \ref{lem:monotone convergence kn} and Corollary 1 of \cite{Articulo Boussinesq},
we partially studied the convergence of $\overline{k}_{n}\left(t_{n}+\left(1-t_{n}\right)\hat{\hat{t}}\right)$
to the profile $k_{\max}\left(1-\left|1-2\hat{\hat{t}}\right|\right)$.
We showed that this convergence was monotonically decreasing, but
we did not study the rate of convergence. To prove that the Boussinesq
solution given in \cite{Articulo Boussinesq} also satisfies the blow-up
criterion for non-homogeneous Euler, we will need to study this rate
of convergence. This is the objective of the next Lemma.
\begin{lem}
\label{lem:convergence of knbar to limit}Let $\left(\rho,\omega,f_{\rho},f_{\omega}\right)$
be the solution of the forced Boussinesq system presented in \cite{Articulo Boussinesq}.
$\forall\hat{\hat{t}}\in\left[0,1\right]$,
\[
\overline{k}_{n}\left(t_{n}+\left(1-t_{n}\right)\hat{\hat{t}}\right)-k_{\max}\left(1-\left|1-2\hat{\hat{t}}\right|\right)\leq\frac{2\ln\left(2\right)\min\left\{ \hat{\hat{t}},1-\hat{\hat{t}}\right\} +C^{-2\left|1-2\hat{\hat{t}}\right|k_{\max}\left(\frac{1}{1-\gamma}\right)^{n}}}{\ln\left(C\right)\left(\frac{1}{1-\gamma}\right)^{n}}.
\]
In particular,
\[
\overline{k}_{n}\left(t_{n}+\left(1-t_{n}\right)\hat{\hat{t}}\right)-k_{\max}\left(1-\left|1-2\hat{\hat{t}}\right|\right)\leq\frac{1+\ln\left(2\right)}{\ln\left(C\right)}\left(1-\gamma\right)^{n}.
\]
\end{lem}
\begin{proof}
On the one hand, by equation \eqref{eq:Boussinesq kn exact}, we have
\[
\overline{k}_{n}\left(t_{n}+\left(1-t_{n}\right)\hat{\hat{t}}\right)=k_{\max}+\frac{1}{\ln\left(C\right)\left(\frac{1}{1-\gamma}\right)^{n}}\ln\left(\frac{1}{\cosh\left(\mathrm{arccosh}\left(C^{k_{\max}\left(\frac{1}{1-\gamma}\right)^{n}}\right)\left(1-2\hat{\hat{t}}\right)\right)}\right).
\]
Differentiating with respect to $\hat{\hat{t}}$, we obtain
\begin{equation}
\frac{\partial}{\partial\hat{\hat{t}}}\left(\overline{k}_{n}\left(t_{n}+\left(1-t_{n}\right)\hat{\hat{t}}\right)\right)=2\frac{\mathrm{arccosh}\left(C^{k_{\max}\left(\frac{1}{1-\gamma}\right)^{n}}\right)}{\ln\left(C\right)\left(\frac{1}{1-\gamma}\right)^{n}}\tanh\left(\mathrm{arccosh}\left(C^{k_{\max}\left(\frac{1}{1-\gamma}\right)^{n}}\right)\left(1-2\hat{\hat{t}}\right)\right).\label{eq:knbar real derivative}
\end{equation}
On the other hand,
\begin{equation}
\frac{\partial}{\partial\hat{\hat{t}}}\left(k_{\max}\left(1-\left|1-2\hat{\hat{t}}\right|\right)\right)=2k_{\max}\text{sign}\left(1-2\hat{\hat{t}}\right).\label{eq:knbar limite derivative}
\end{equation}
Both functions $\overline{k}_{n}\left(t_{n}+\left(1-t_{n}\right)\hat{\hat{t}}\right)$
and $k_{\max}\left(1-\left|1-2\hat{\hat{t}}\right|\right)$ are even
with respect to $\hat{\hat{t}}=\frac{1}{2}$ (see equation \eqref{eq:Boussinesq kn exact}),
so we can restrict our analysis to the interval $\hat{\hat{t}}\in\left[0,\frac{1}{2}\right]$.
Recalling that $\mathrm{arccosh}\left(x\right)=\ln\left(x+\sqrt{x^{2}-1}\right)$,
it is not difficult to see that
\begin{equation}
\begin{aligned}\frac{\mathrm{arccosh}\left(C^{k_{\max}\left(\frac{1}{1-\gamma}\right)^{n}}\right)}{\ln\left(C\right)\left(\frac{1}{1-\gamma}\right)^{n}} & \leq\frac{\ln\left(2C^{k_{\max}\left(\frac{1}{1-\gamma}\right)^{n}}\right)}{\ln\left(C\right)\left(\frac{1}{1-\gamma}\right)^{n}}=\frac{\ln\left(2\right)+k_{\max}\ln\left(C\right)\left(\frac{1}{1-\gamma}\right)^{n}}{\ln\left(C\right)\left(\frac{1}{1-\gamma}\right)^{n}}=\\
 & \leq k_{\max}+\frac{\ln\left(2\right)}{\ln\left(C\right)\left(\frac{1}{1-\gamma}\right)^{n}}.\\
\frac{\mathrm{arccosh}\left(C^{k_{\max}\left(\frac{1}{1-\gamma}\right)^{n}}\right)}{\ln\left(C\right)\left(\frac{1}{1-\gamma}\right)^{n}} & \geq\frac{\ln\left(C^{k_{\max}\left(\frac{1}{1-\gamma}\right)^{n}}\right)}{\ln\left(C\right)\left(\frac{1}{1-\gamma}\right)^{n}}=k_{\max}.
\end{aligned}
\label{eq:bounds arccosh}
\end{equation}
Furthermore, since
\[
1-\tanh\left(x\right)=1-\frac{\mathrm{e}^{x}-\mathrm{e}^{-x}}{\mathrm{e}^{x}+\mathrm{e}^{-x}}=\frac{2\mathrm{e}^{-x}}{\mathrm{e}^{x}+\mathrm{e}^{-x}}\leq2\mathrm{e}^{-2x}\quad\forall x\in\left[0,\infty\right),
\]
we deduce that
\begin{equation}
\begin{aligned} & \left|1-\tanh\left(\mathrm{arccosh}\left(C^{k_{\max}\left(\frac{1}{1-\gamma}\right)^{n}}\right)\left(1-2\hat{\hat{t}}\right)\right)\right|\leq\\
\leq & 2\exp\left(-2\mathrm{arccosh}\left(C^{k_{\max}\left(\frac{1}{1-\gamma}\right)^{n}}\right)\left(1-2\hat{\hat{t}}\right)\right)\leq2\exp\left(-2\ln\left(C^{k_{\max}\left(\frac{1}{1-\gamma}\right)^{n}}\right)\left(1-2\hat{\hat{t}}\right)\right)=\\
\leq & 2C^{-2\left(1-2\hat{\hat{t}}\right)k_{\max}\left(\frac{1}{1-\gamma}\right)^{n}}.
\end{aligned}
\label{eq:bounds tanh}
\end{equation}
Coming back to equations \eqref{eq:knbar real derivative}, \eqref{eq:knbar limite derivative}, \eqref{eq:bounds arccosh}, and \eqref{eq:bounds tanh},
we can ensure that, $\forall\hat{\hat{t}}\in\left[0,\frac{1}{2}\right]$,
\[
\begin{aligned} & \left|\frac{\partial}{\partial\hat{\hat{t}}}\left(\overline{k}_{n}\left(t_{n}+\left(1-t_{n}\right)\hat{\hat{t}}\right)\right)-\frac{\partial}{\partial\hat{\hat{t}}}\left(k_{\max}\left(1-\left|1-2\hat{\hat{t}}\right|\right)\right)\right|\leq\\
\leq & \left|2\frac{\mathrm{arccosh}\left(C^{k_{\max}\left(\frac{1}{1-\gamma}\right)^{n}}\right)}{\ln\left(C\right)\left(\frac{1}{1-\gamma}\right)^{n}}\tanh\left(\mathrm{arccosh}\left(C^{k_{\max}\left(\frac{1}{1-\gamma}\right)^{n}}\right)\left(1-2\hat{\hat{t}}\right)\right)-2k_{\max}\right|\leq\\
\leq & 2\left|\frac{\mathrm{arccosh}\left(C^{k_{\max}\left(\frac{1}{1-\gamma}\right)^{n}}\right)}{\ln\left(C\right)\left(\frac{1}{1-\gamma}\right)^{n}}-k_{\max}\right|\tanh\left(\mathrm{arccosh}\left(C^{k_{\max}\left(\frac{1}{1-\gamma}\right)^{n}}\right)\left(1-2\hat{\hat{t}}\right)\right)+\\
 & +2k_{\max}\left|\tanh\left(\mathrm{arccosh}\left(C^{k_{\max}\left(\frac{1}{1-\gamma}\right)^{n}}\right)\left(1-2\hat{\hat{t}}\right)\right)-1\right|\leq\\
\leq & \frac{2\ln\left(2\right)}{\ln\left(C\right)\left(\frac{1}{1-\gamma}\right)^{n}}+4k_{\max}C^{-2\left(1-2\hat{\hat{t}}\right)k_{\max}\left(\frac{1}{1-\gamma}\right)^{n}}.
\end{aligned}
\]
Using that $\frac{\partial\left|f\right|}{\partial\hat{\hat{t}}}\leq\left|\frac{\partial f}{\partial\hat{\hat{t}}}\right|$,
integrating from $\hat{\hat{t}}=0$ to a generic $\hat{\hat{t}}\in\left[0,\frac{1}{2}\right]$
and bearing in mind that $\overline{k}_{n}\left(t_{n}\right)=0=k_{\max}\left(1-\left|1-2\cdot0\right|\right)$,
we arrive at
\[
\begin{aligned} & \left|\overline{k}_{n}\left(t_{n}+\left(1-t_{n}\right)\hat{\hat{t}}\right)-k_{\max}\left(1-\left|1-2\hat{\hat{t}}\right|\right)\right|\leq\\
\leq & \frac{2\ln\left(2\right)}{\ln\left(C\right)\left(\frac{1}{1-\gamma}\right)^{n}}\hat{\hat{t}}+\frac{1}{\ln\left(C\right)\left(\frac{1}{1-\gamma}\right)^{n}}\left(C^{-2\left(1-2\hat{\hat{t}}\right)k_{\max}\left(\frac{1}{1-\gamma}\right)^{n}}-C^{-2k_{\max}\left(\frac{1}{1-\gamma}\right)^{n}}\right).
\end{aligned}
\]
Neglecting the negative factor in the second summand, recalling that
the convergence of $\overline{k}_{n}\left(t_{n}+\left(1-t_{n}\right)\hat{\hat{t}}\right)$
to $k_{\max}\left(1-\left|1-2\hat{\hat{t}}\right|\right)$ is monotonically
decreasing (see Lemma \ref{lem:monotone convergence kn}) and adapting
the equation above to the general case $\hat{\hat{t}}\in\left[0,1\right]$
by symmetry, we arrive at the first statement of the Lemma. Lastly,
uniformly in $\hat{\hat{t}}\in\left[0,\frac{1}{2}\right]$ we can
bound
\[
\left|\overline{k}_{n}\left(t_{n}+\left(1-t_{n}\right)\hat{\hat{t}}\right)-k_{\max}\left(1-\left|1-2\hat{\hat{t}}\right|\right)\right|\leq\frac{1+\ln\left(2\right)}{\ln\left(C\right)\left(\frac{1}{1-\gamma}\right)^{n}}.
\]
Again, the fact that the convergence of $\overline{k}_{n}\left(t_{n}+\left(1-t_{n}\right)\hat{\hat{t}}\right)$
to $k_{\max}\left(1-\left|1-2\hat{\hat{t}}\right|\right)$ is monotonically
decreasing (see Lemma \ref{lem:monotone convergence kn}) allows us
to remove the absolute value, leading to the second statement.
\end{proof}

\begin{prop}
\label{prop:blow-up criterion}Let $\left(\rho,\omega,f_{\rho},f_{\omega}\right)$
be the solution of the forced Boussinesq system presented in \cite{Articulo Boussinesq}.
Provided that $\delta$ and $\mu$ are sufficiently small, let us
say $\delta\leq\Upsilon_{1}$ and $\mu\leq\Upsilon_{2}$,
\[
\lim_{T\to1^{-}}\int^{T}_{0}\left|\left|\omega\left(t,\cdot\right)\right|\right|_{L^{\infty}\left(\mathbb{R}^{2}\right)}\mathrm{d}t=\infty.
\]
\end{prop}
\begin{proof}
First of all, $\forall n\in\mathbb{N}$, clearly,
\[
I\coloneqq\int^{1}_{0}\left|\left|\omega\left(t,\cdot\right)\right|\right|_{L^{\infty}\left(\mathbb{R}^{2}\right)}\mathrm{d}t\geq\int^{t_{n+1}}_{t_{n}}\left|\left|\omega\left(t,\cdot\right)\right|\right|_{L^{\infty}\left(\mathbb{R}^{2}\right)}\mathrm{d}t,
\]
where $t_{n}$ is the time instant when layer $n$ is introduced (see
equation \eqref{eq:Boussinesq time scale}). Since no layer with index
bigger than $n$ exists for times $t\in\left[t_{n},t_{n+1}\right]$,
we have
\begin{equation}
I\geq\underbrace{\int^{t_{n+1}}_{t_{n}}\left|\left|\omega^{\left(n\right)}\left(t,\cdot\right)\right|\right|_{L^{\infty}\left(\mathbb{R}^{2}\right)}\mathrm{d}t}_{\eqqcolon I_{+}}-\underbrace{\int^{t_{n+1}}_{t_{n}}\left|\left|\Omega^{\left(n-1\right)}\left(t,\cdot\right)\right|\right|_{L^{\infty}\left(\mathbb{R}^{2}\right)}\mathrm{d}t}_{\eqqcolon I_{-}},\label{eq:decomposition I}
\end{equation}
where $\Omega^{\left(n-1\right)}=\sum^{n-1}_{m=1}\omega^{\left(m\right)}$,
as in \cite{Articulo Boussinesq}.

We will begin bounding $I_{-}$. Since $\left|\left|\cdot\right|\right|_{L^{\infty}\left(\mathbb{R}^{2}\right)}$
is invariant under diffeomorphisms of the domain, in view of Proposition
\ref{prop:computations vorticity} and recalling that $\lambda_{n}\le1$
by equation \eqref{eq:Boussinesq lambda_n}, we deduce that
\[
\begin{aligned}I_{-} & \leq\sum^{n-1}_{m=1}\int^{1}_{t_{n}}\left|\left|\omega^{\left(m\right)}\left(t,\cdot\right)\right|\right|_{L^{\infty}\left(\mathbb{R}^{2}\right)}\mathrm{d}t\lesssim_{\varphi}\sum^{n-1}_{m=1}\int^{1}_{t_{n}}B_{m}\left(t\right)\max\left\{ a_{m}\left(t\right)^{2},b_{m}\left(t\right)^{2}\right\} \mathrm{d}t\leq\\
 & \lesssim_{\varphi}\sum^{n-1}_{m=1}\int^{1}_{t_{n}}B_{m}\left(t\right)\left(a_{m}\left(t\right)^{2}+b_{m}\left(t\right)^{2}\right)\mathrm{d}t.
\end{aligned}
\]
Equations \eqref{eq:Boussinesq Bn(t)}, \eqref{eq:Boussinesq Mn}, and \eqref{eq:Boussinesq time scale},
Lemmas \ref{lem:estimate sum superexponential} and \ref{lem:arccosh by ln}
provide
\begin{equation}
\begin{aligned}I_{-} & \lesssim_{\varphi}\sum^{n-1}_{m=1}\int^{1}_{t_{n}}2YC^{\delta\left(\frac{1}{1-\gamma}\right)^{m}}\mathrm{d}t\leq\left(1-t_{n}\right)\sum^{n-1}_{m=1}2YC^{\delta\left(\frac{1}{1-\gamma}\right)^{m}}\lesssim_{\varphi,\delta}\left(1-t_{n}\right)2YC^{\delta\left(\frac{1}{1-\gamma}\right)^{n-1}}\lesssim\\
 & \lesssim_{\varphi,\delta}\mathrm{arccosh}\left(C^{k_{\max}\left(\frac{1}{1-\gamma}\right)^{n}}\right)\le\ln\left(2\right)+\ln\left(C\right)k_{\max}\left(\frac{1}{1-\gamma}\right)^{n}.
\end{aligned}
\label{eq:bound I-}
\end{equation}

On the other hand, since $\left|\left|\cdot\right|\right|_{L^{\infty}\left(\mathbb{R}^{2}\right)}$
is invariant under diffeomorphisms of the domain,
\[
I_{+}\geq\int^{t_{n+1}}_{t_{n}}\left|\left|\widetilde{\omega^{\left(n\right)}}^{n}\left(t,\cdot\right)\right|\right|_{L^{\infty}\left(\mathbb{R}^{2}\right)}\mathrm{d}t.
\]
Looking at Proposition \ref{prop:computations vorticity} and evaluating
$\widetilde{\omega^{\left(n\right)}}^{n}\left(t,x\right)$ at $x=\left(\frac{\pi}{2},\frac{\pi}{2}\right)$,
which is part of the zone where the cut-off functions are identically
1 according to equations \eqref{eq:Boussinesq lambda_n} and \eqref{eq:Boussinesq psi rho},
we get
\[
I_{+}\geq\int^{t_{n+1}}_{t_{n}}B_{n}\left(t\right)\left(a_{n}\left(t\right)^{2}+b_{n}\left(t\right)^{2}\right)\mathrm{d}t.
\]
Equations \eqref{eq:Boussinesq Bn(t)} and \eqref{eq:Boussinesq Mn}
provide
\begin{equation}
I_{+}\geq\int^{t_{n+1}}_{t_{n}}2YC^{\delta\left(\frac{1}{1-\gamma}\right)^{n}}\frac{\int^{t}_{t_{n}}h^{\left(n\right)}\left(s\right)b_{n}\left(s\right)\mathrm{d}s}{\int^{1}_{t_{n}}h^{\left(n\right)}\left(s\right)b_{n}\left(s\right)\mathrm{d}s}\mathrm{d}t.\label{eq:bound I+ v0}
\end{equation}
To find a lower bound for the expression above, it will prove useful
to show that
\begin{equation}
\frac{\int^{t_{n}+\frac{3}{4}\left(1-t_{n}\right)}_{t_{n}}h^{\left(n\right)}\left(s\right)b_{n}\left(s\right)\mathrm{d}s}{\int^{1}_{t_{n}}h^{\left(n\right)}\left(s\right)b_{n}\left(s\right)\mathrm{d}s}\geq\frac{1}{2}\label{eq:needed for I+}
\end{equation}
as long as $n$ is large enough. In order to accomplish this task,
we will bound
\[
Q\coloneqq\frac{\int^{1}_{t_{n}+\frac{3}{4}\left(1-t_{n}\right)}h^{\left(n\right)}\left(s\right)b_{n}\left(s\right)\mathrm{d}s}{\int^{1}_{t_{n}}h^{\left(n\right)}\left(s\right)b_{n}\left(s\right)\mathrm{d}s}.
\]
Lemma \ref{lem:estimate of integral hn bn} and equation \eqref{eq:Boussinesq an(t) bn(t)}
imply that
\[
Q\lesssim_{\mu}Y\frac{\int^{1}_{t_{n}+\frac{3}{4}\left(1-t_{n}\right)}C^{\left(1+k_{n}\left(s\right)\right)\left(\frac{1}{1-\gamma}\right)^{n}}\mathrm{d}s}{C^{\left(1+k_{\max}-\mu-\delta\left(1-\gamma\right)\right)\left(\frac{1}{1-\gamma}\right)^{n}}}.
\]
Using Proposition \ref{prop:convergence kn to ideal model} (with
$\beta=\beta'=\frac{1}{2}$) and Lemma \ref{lem:convergence of knbar to limit},
we deduce that $\forall s\in\left[t_{n}+\frac{3}{4}\left(1-t_{n}\right),1\right]$,
we have
\[
\begin{aligned}k_{n}\left(s\right) & \leq k_{\max}\left(1-\left|1-2\frac{3}{4}\right|\right)+\frac{1+\ln\left(2\right)}{\ln\left(C\right)}\left(1-\gamma\right)^{n}+L\left(\delta\right)C^{-\frac{1}{8}\delta\gamma\left(\frac{1}{1-\gamma}\right)^{n-1}}+\\
 & \leq\frac{1}{2}k_{\max}+\frac{1+\ln\left(2\right)}{\ln\left(C\right)}\left(1-\gamma\right)^{n}+L\left(\delta\right)C^{-\frac{1}{8}\delta\gamma\left(\frac{1}{1-\gamma}\right)^{n-1}}
\end{aligned}
\]
for some $L\left(\delta\right)>0$. Multiplying by $\left(\frac{1}{1-\gamma}\right)^{n}$
and making use of Lemma \ref{lem:exponential superexponential bound}
with $a\leftarrow\frac{1}{1-\gamma}$, $b\leftarrow\frac{1}{8}\delta\gamma$,
$\beta\leftarrow\frac{1}{2}$ and $n\leftarrow n-1$ we arrive at
\[
k_{n}\left(s\right)\left(\frac{1}{1-\gamma}\right)^{n}\leq\frac{1}{2}k_{\max}\left(\frac{1}{1-\gamma}\right)^{n}+\frac{1+\ln\left(2\right)}{\ln\left(C\right)}+L_{1}\left(\delta,C\right)C^{-\frac{1}{16}\delta\gamma\left(\frac{1}{1-\gamma}\right)^{n-1}}
\]
for some $L_{1}\left(\delta,C\right)>0$. As long as $n$ is large
enough, we may write
\[
k_{n}\left(s\right)\left(\frac{1}{1-\gamma}\right)^{n}\leq\frac{1}{2}k_{\max}\left(\frac{1}{1-\gamma}\right)^{n}+\frac{1+\ln\left(2\right)}{\ln\left(C\right)}+\frac{1}{\ln\left(C\right)}.
\]
Substituting above, we obtain
\[
Q\lesssim_{\mu}Y\frac{\int^{1}_{t_{n}+\frac{3}{4}\left(1-t_{n}\right)}C^{\frac{1}{2}k_{\max}\left(\frac{1}{1-\gamma}\right)^{n}}\mathrm{e}^{\left(2+\ln\left(2\right)\right)}\mathrm{d}s}{C^{\left(k_{\max}-\mu-\delta\left(1-\gamma\right)\right)\left(\frac{1}{1-\gamma}\right)^{n}}}\lesssim_{\mu}\frac{Y\overbrace{\left(1-t_{n}\right)}^{\leq1}}{C^{\left(\frac{1}{2}k_{\max}-\mu-\delta\left(1-\gamma\right)\right)\left(\frac{1}{1-\gamma}\right)^{n}}}\leq\frac{Y}{C^{\left(\frac{1}{2}k_{\max}-\mu-\delta\left(1-\gamma\right)\right)\left(\frac{1}{1-\gamma}\right)^{n}}}.
\]
As long as $\mu$ and $\delta$ are sufficiently small, $\frac{1}{2}k_{\max}-\mu-\delta\left(1-\gamma\right)>0$
and, consequently, $Q$ can be made as small as we wish by taking
$n$ large enough. This, in particular, implies \eqref{eq:needed for I+}
for $n$ large enough. Substituting back into \eqref{eq:bound I+ v0},
as $\int^{t}_{t_{n}}h^{\left(n\right)}\left(s\right)b_{n}\left(s\right)\mathrm{d}s\geq\int^{t_{n}+\frac{3}{4}\left(1-t_{n}\right)}_{t_{n}}h^{\left(n\right)}\left(s\right)b_{n}\left(s\right)\mathrm{d}s$
$\forall t\in\left[t_{n}+\frac{3}{4}\left(1-t_{n}\right),t_{n+1}\right]$,
we find
\begin{equation}
I_{+}\geq\int^{t_{n+1}}_{t_{n}+\frac{3}{4}\left(1-t_{n}\right)}2YC^{\delta\left(\frac{1}{1-\gamma}\right)^{n}}\frac{\int^{t}_{t_{n}}h^{\left(n\right)}\left(s\right)b_{n}\left(s\right)\mathrm{d}s}{\int^{1}_{t_{n}}h^{\left(n\right)}\left(s\right)b_{n}\left(s\right)\mathrm{d}s}\mathrm{d}t\geq YC^{\delta\left(\frac{1}{1-\gamma}\right)^{n}}\int^{t_{n+1}}_{t_{n}+\frac{3}{4}\left(1-t_{n}\right)}\mathrm{d}t.\label{eq:bound I+ v1}
\end{equation}
Now,
\[
\begin{aligned} & t_{n+1}-\left(t_{n}+\frac{3}{4}\left(1-t_{n}\right)\right)=\\
= & t_{n+1}-1+1-\left(t_{n}+\frac{3}{4}\left(1-t_{n}\right)\right)=\frac{1}{4}\left(1-t_{n}\right)-\left(1-t_{n+1}\right).
\end{aligned}
\]
Equation \eqref{eq:Boussinesq time scale} provides
\[
t_{n+1}-\left(t_{n}+\frac{3}{4}\left(1-t_{n}\right)\right)=\frac{1}{Y}\left[\frac{1}{4}C^{-\delta\left(\frac{1}{1-\gamma}\right)^{n-1}}\mathrm{arccosh}\left(C^{k_{\max}\left(\frac{1}{1-\gamma}\right)^{n}}\right)-C^{-\delta\left(\frac{1}{1-\gamma}\right)^{n}}\mathrm{arccosh}\left(C^{k_{\max}\left(\frac{1}{1-\gamma}\right)^{n+1}}\right)\right].
\]
It is clear that the absolute value of the second summand decreases
much faster with $n$ than the first summand. This means that, provided
that $n$ is large enough we may ensure that
\[
t_{n+1}-\left(t_{n}+\frac{3}{4}\left(1-t_{n}\right)\right)\geq\frac{1}{8Y}C^{-\delta\left(\frac{1}{1-\gamma}\right)^{n-1}}\mathrm{arccosh}\left(C^{k_{\max}\left(\frac{1}{1-\gamma}\right)^{n}}\right).
\]
Going back to equation \eqref{eq:bound I+ v1}, we obtain
\begin{equation}
I_{+}\geq\frac{1}{8}C^{\delta\gamma\left(\frac{1}{1-\gamma}\right)^{n}}\mathrm{arccosh}\left(C^{k_{\max}\left(\frac{1}{1-\gamma}\right)^{n}}\right).\label{eq:bound I+ v2}
\end{equation}

Uniting equations \eqref{eq:decomposition I}, \eqref{eq:bound I-}, and \eqref{eq:bound I+ v2},
we infer that
\[
I\geq\frac{1}{8}C^{\delta\gamma\left(\frac{1}{1-\gamma}\right)^{n}}\mathrm{arccosh}\left(C^{k_{\max}\left(\frac{1}{1-\gamma}\right)^{n}}\right)-L_{2}\left(\varphi,\mu\right)\left[\ln\left(2\right)+\ln\left(C\right)k_{\max}\left(\frac{1}{1-\gamma}\right)^{n}\right]
\]
for some $L_{2}\left(\varphi,\mu\right)>0$. As the first factor grows
much more rapidly in $n\in\mathbb{N}$ than the second factor, taking
$n$ large enough guarantees that
\[
\int^{1}_{0}\left|\left|\omega\left(t,\cdot\right)\right|\right|_{L^{\infty}\left(\mathbb{R}^{2}\right)}\mathrm{d}t=I\geq\frac{1}{16}C^{\delta\gamma\left(\frac{1}{1-\gamma}\right)^{n}}\mathrm{arccosh}\left(C^{k_{\max}\left(\frac{1}{1-\gamma}\right)^{n}}\right).
\]
Taking the limit as $n\to\infty$ provides the result.
\end{proof}

\section{Known estimates for the Laplacian}

In this section, we will review some well-known estimates for the
Laplacian, focusing on the key differences in $\mathbb{R}^{2}$ that
make the theory a little trickier sometimes.
\begin{prop}
\label{prop:Laplacian estimate L^p}Consider the equation $\Delta u=f$,
where $f:\mathbb{R}^{2}\to\mathbb{R}$ and $u$ is computed via the
fundamental solution of the Laplace equation. Provided that
\begin{enumerate}
\item $\int_{\mathbb{R}^{2}}f\left(x\right)\mathrm{d}x=0$,
\item $f$ has compact support,
\item $f\in L^{p}\left(\mathbb{R}^{2}\right)$ with $p>2$,
\end{enumerate}
then $u\in L^{p}\left(\mathbb{R}^{2}\right)$ and, furthermore,
\[
\left|\left|u\right|\right|_{L^{p}\left(\mathbb{R}^{2}\right)}\lesssim_{p}\left[\mathrm{diam}\left(\mathrm{supp}f\right)\right]^{2}\left|\left|f\right|\right|_{L^{p}\left(\mathbb{R}^{2}\right)},
\]
where $\mathrm{diam}\left(\mathrm{supp}f\right)$ denotes the diameter
of the support of $f$.
\end{prop}
\begin{proof}
We may restrict the proof to the case $0\in\text{supp}f$. Indeed,
if $0\notin\text{supp}f$, take $x_{0}\in\text{supp}f$ and consider
$\Delta v=g$, where $g\left(x\right)=f\left(x-x_{0}\right)$. In
this way, $v\left(x\right)=u\left(x-x_{0}\right)$. As $0\in\text{supp}g$,
we can apply the statement to obtain the desired bound for $v$ and
$g$. Since $\left|\left|u\right|\right|_{L^{p}\left(\mathbb{R}^{2}\right)}=\left|\left|v\right|\right|_{L^{p}\left(\mathbb{R}^{2}\right)}$,
$\left|\left|f\right|\right|_{L^{p}\left(\mathbb{R}^{2}\right)}=\left|\left|g\right|\right|_{L^{p}\left(\mathbb{R}^{2}\right)}$
and $\text{diam}\left(\text{supp}f\right)=\text{diam}\left(\text{supp}g\right)$,
the result must also be true for $u$ and $f$.

We know that $u$ is given by
\begin{equation}
u\left(x\right)=\frac{1}{2\pi}\int_{\mathbb{R}^{2}}\ln\left(\left|\left|x-y\right|\right|_{2}\right)f\left(y\right)\mathrm{d}y.\label{eq:u Laplace}
\end{equation}
Since $\int_{\mathbb{R}^{2}}f\left(x\right)\mathrm{d}x=0$, we can
write
\[
u\left(x\right)=\frac{1}{2\pi}\int_{\mathbb{R}^{2}}\left[\ln\left(\left|\left|x-y\right|\right|_{2}\right)-\ln\left(\left|\left|x\right|\right|_{2}\right)\right]f\left(y\right)\mathrm{d}y.
\]
Consequently,
\[
\left|\left|u\right|\right|^{p}_{L^{p}\left(\mathbb{R}^{2}\right)}=\frac{1}{\left(2\pi\right)^{p}}\int_{\mathbb{R}^{2}}\left|\int_{\mathbb{R}^{2}}\left[\ln\left(\left|\left|x-y\right|\right|_{2}\right)-\ln\left(\left|\left|x\right|\right|_{2}\right)\right]f\left(y\right)\mathrm{d}y\right|^{p}\mathrm{d}x.
\]
As $f$ has compact support, we have
\[
\left|\left|u\right|\right|^{p}_{L^{p}\left(\mathbb{R}^{2}\right)}=\frac{1}{\left(2\pi\right)^{p}}\int_{\mathbb{R}^{2}}\left|\int_{\text{supp}f}\left[\ln\left(\left|\left|x-y\right|\right|_{2}\right)-\ln\left(\left|\left|x\right|\right|_{2}\right)\right]f\left(y\right)\mathrm{d}y\right|^{p}\mathrm{d}x.
\]
Multiplying and dividing by the Lebesgue measure of the support of
$f$ (denoted by $\left|\text{supp}f\right|$) in the inner integral,
we obtain
\[
\left|\left|u\right|\right|^{p}_{L^{p}\left(\mathbb{R}^{2}\right)}=\frac{\left|\text{supp}f\right|^{p}}{\left(2\pi\right)^{p}}\int_{\mathbb{R}^{2}}\left|\int_{\text{supp}f}\left[\ln\left(\left|\left|x-y\right|\right|_{2}\right)-\ln\left(\left|\left|x\right|\right|_{2}\right)\right]f\left(y\right)\frac{\mathrm{d}y}{\left|\mathrm{supp}f\right|}\right|^{p}\mathrm{d}x.
\]
Clearly, $\frac{\mathrm{d}y}{\left|\mathrm{supp}f\right|}$ is a probability
measure over $\mathrm{supp}f$. Thereby, as $z\to\left|z\right|^{p}$
is convex $\forall p>1$, Jensen's inequality allows us to write
\[
\left|\left|u\right|\right|^{p}_{L^{p}\left(\mathbb{R}^{2}\right)}\leq\frac{\left|\text{supp}f\right|^{p}}{\left(2\pi\right)^{p}}\int_{\mathbb{R}^{2}}\int_{\mathrm{supp}f}\left|\ln\left(\left|\left|x-y\right|\right|_{2}\right)-\ln\left(\left|\left|x\right|\right|_{2}\right)\right|^{p}\left|f\left(y\right)\right|^{p}\frac{\mathrm{d}y}{\left|\mathrm{supp}f\right|}\mathrm{d}x.
\]
Now, Tonelli's Theorem guarantees that
\begin{equation}
\left|\left|u\right|\right|^{p}_{L^{p}\left(\mathbb{R}^{2}\right)}\leq\frac{\left|\text{supp}f\right|^{p-1}}{\left(2\pi\right)^{p}}\int_{\mathrm{supp}f}\left|f\left(y\right)\right|^{p}\underbrace{\int_{\mathbb{R}^{2}}\left|\ln\left(\left|\left|x-y\right|\right|_{2}\right)-\ln\left(\left|\left|x\right|\right|_{2}\right)\right|^{p}\mathrm{d}x}_{\eqqcolon I}\mathrm{d}y.\label{eq:Lp norm of u}
\end{equation}
In the inner integral $I$, we undertake the change of variables
\[
z=\frac{x}{\left|\left|y\right|\right|_{2}}\implies\left|\det\mathrm{J}_{x}z\right|=\frac{1}{\left|\left|y\right|\right|^{2}_{2}}.
\]
In this way,
\begin{equation}
I=\left|\left|y\right|\right|^{2}_{2}\int_{\mathbb{R}^{2}}\left|\ln\left(\frac{\left|\left|z\left|\left|y\right|\right|_{2}-y\right|\right|_{2}}{\left|\left|z\right|\right|_{2}\left|\left|y\right|\right|_{2}}\right)\right|^{p}\mathrm{d}z.\label{eq:I after change of variables}
\end{equation}
The triangular and reverse triangular inequalities provide us the
bounds
\[
\begin{aligned}\left|\left|z\left|\left|y\right|\right|_{2}-y\right|\right|_{2} & \leq\left|\left|z\right|\right|_{2}\left|\left|y\right|\right|_{2}+\left|\left|y\right|\right|_{2}\leq\left|\left|y\right|\right|_{2}\left(1+\left|\left|z\right|\right|_{2}\right),\\
\left|\left|z\left|\left|y\right|\right|_{2}-y\right|\right|_{2} & \geq\left|\left|\left|z\right|\right|_{2}\left|\left|y\right|\right|_{2}-\left|\left|y\right|\right|_{2}\right|=\left|\left|y\right|\right|_{2}\left|\left|\left|z\right|\right|_{2}-1\right|.
\end{aligned}
\]
Since $\ln$ is an increasing function, these bounds above ensure
that
\[
\ln\left|1-\frac{1}{\left|\left|z\right|\right|_{2}}\right|=\ln\left|\frac{\left|\left|z\right|\right|_{2}-1}{\left|\left|z\right|\right|_{2}}\right|\leq\ln\left(\frac{\left|\left|z\left|\left|y\right|\right|_{2}-y\right|\right|_{2}}{\left|\left|z\right|\right|_{2}\left|\left|y\right|\right|_{2}}\right)\leq\ln\left(\frac{1+\left|\left|z\right|\right|_{2}}{\left|\left|z\right|\right|_{2}}\right)=\ln\left(1+\frac{1}{\left|\left|z\right|\right|_{2}}\right).
\]
Thus,
\[
\left|\ln\left(\frac{\left|\left|z\left|\left|y\right|\right|_{2}-y\right|\right|_{2}}{\left|\left|z\right|\right|_{2}\left|\left|y\right|\right|_{2}}\right)\right|\leq\max\left\{ \left|\ln\left|1-\frac{1}{\left|\left|z\right|\right|_{2}}\right|\right|,\left|\ln\left(1+\frac{1}{\left|\left|z\right|\right|_{2}}\right)\right|\right\} 
\]
and, as a consequence, coming back to equation \eqref{eq:I after change of variables},
we deduce that
\[
I\leq\left|\left|y\right|\right|^{2}_{2}\left[\int_{\mathbb{R}^{2}}\left|\ln\left|1-\frac{1}{\left|\left|z\right|\right|_{2}}\right|\right|^{p}\mathrm{d}z+\int_{\mathbb{R}^{2}}\left|\ln\left(1+\frac{1}{\left|\left|z\right|\right|_{2}}\right)\right|^{p}\mathrm{d}z\right].
\]
To continue, we split the integration domain of $I$ in two parts.
For $I_{1}$, we will integrate in $\left\{ \left|\left|z\right|\right|_{2}\geq2\right\} $
and, for $I_{2}$, we will integrate in $\left\{ \left|\left|z\right|\right|_{2}<2\right\} $.

For $I_{1}$, using that $\left|\ln\left|1-v\right|\right|\leq2v$
$\forall v\in\left[0,\frac{1}{2}\right]$ and that $\ln\left(1+v\right)\leq v$
$\forall v\in\left[0,\infty\right)$, we can easily bound
\[
I_{1}\leq\left(2^{p}+1\right)\left|\left|y\right|\right|^{2}_{2}\int_{\left\{ \left|\left|z\right|\right|_{2}\geq2\right\} }\frac{1}{\left|\left|z\right|\right|^{p}_{2}}\mathrm{d}z\lesssim_{p}\left|\left|y\right|\right|^{2}_{2}\int^{\infty}_{2}\frac{1}{\rho^{p-1}}\mathrm{d}\rho.
\]
As long as $p>2$, the integral above is bounded, which means that
\begin{equation}
I_{1}\lesssim_{p}\left|\left|y\right|\right|^{2}_{2}.\label{eq:Laplace Lp I1}
\end{equation}

Now, for $I_{2}$, we treat each integral separately, i.e.,
\[
I_{2}=I_{2,1}+I_{2,2},\quad I_{2,1}\coloneqq\left|\left|y\right|\right|^{2}_{2}\int_{\left\{ \left|\left|z\right|\right|_{2}<2\right\} }\left|\ln\left|1-\frac{1}{\left|\left|z\right|\right|_{2}}\right|\right|^{p}\mathrm{d}z,\quad I_{2,2}\coloneqq\left|\left|y\right|\right|^{2}_{2}\int_{\left\{ \left|\left|z\right|\right|_{2}<2\right\} }\left|\ln\left(1+\frac{1}{\left|\left|z\right|\right|_{2}}\right)\right|^{p}\mathrm{d}z.
\]
For the first integral $I_{2,1}$, we pass to polar coordinates $\rho=\left|\left|z\right|\right|_{2}$,
$\theta=\text{atan2}\left(z_{2},z_{1}\right)$ and we do the change
of variables
\[
w=1-\frac{1}{\rho}\iff\rho=\frac{1}{1-w}\implies\mathrm{d}\rho=\frac{\mathrm{d}w}{\left(1-w\right)^{2}},\quad\left|\left|z\right|\right|_{2}\leq2\iff0\leq\rho\leq2\iff w\in\left(-\infty,\frac{1}{2}\right].
\]
In this manner, we arrive at
\[
I_{2,1}\lesssim\left|\left|y\right|\right|^{2}_{2}\int^{\frac{1}{2}}_{-\infty}\left|\ln\left|w\right|\right|^{p}\frac{\mathrm{d}w}{\left(1-w\right)^{3}}.
\]
Since any power of a logarithm is integrable at the origin and $\frac{1}{\left(1-w\right)^{3}}$
dominates at $-\infty$ and is integrable there, we conclude that
the integral above is bounded (with value depending on $p$), which
leads to
\begin{equation}
I_{2,1}\lesssim_{p}\left|\left|y\right|\right|^{2}_{2}.\label{eq:Laplace Lp I21}
\end{equation}
Similarly, for the second integral $I_{2,2}$, we also pass to polar
coordinates and undertake the change of variables
\[
w=1+\frac{1}{\rho}\iff\rho=\frac{1}{w-1}\implies\mathrm{d}\rho=-\frac{\mathrm{d}w}{\left(w-1\right)^{2}},\quad\left|\left|z\right|\right|_{2}\leq2\iff0\leq\rho\leq2\iff w\in\left[\frac{3}{2},\infty\right).
\]
This leads to
\begin{equation}
I_{2,2}\lesssim\left|\left|y\right|\right|^{2}_{2}\int^{\infty}_{\frac{3}{2}}\left|\ln w\right|^{p}\frac{\mathrm{d}w}{\left(w-1\right)^{3}}.\label{eq:Laplace Lp I22}
\end{equation}
As $\frac{1}{\left(w-1\right)^{3}}$ dominates at $\infty$ and is
integrable there, we infer that the integral above is bounded (with
value depending on $p$). Consequently,
\[
I_{2}=I_{2,1}+I_{2,2}\lesssim_{p}\left|\left|y\right|\right|^{2}_{2}.
\]
Substituting back into equation \eqref{eq:Lp norm of u}, taking into
account equations \eqref{eq:Laplace Lp I1}, \eqref{eq:Laplace Lp I21}, and \eqref{eq:Laplace Lp I22},
we obtain
\[
\left|\left|u\right|\right|^{p}_{L^{p}\left(\mathbb{R}^{2}\right)}\lesssim_{p}\left|\text{supp}f\right|^{p-1}\int_{\mathrm{supp}f}\left|f\left(y\right)\right|^{p}\left|\left|y\right|\right|^{2}_{2}\mathrm{d}y.
\]
Since $0\in\text{supp}f$, we may bound
\[
\left|\left|y\right|\right|_{2}\leq\text{diam}\left(\text{supp}f\right),\quad\left|\text{supp}f\right|\leq\pi\left[\text{diam}\left(\text{supp}f\right)\right]^{2}
\]
which provides
\[
\left|\left|u\right|\right|^{p}_{L^{p}\left(\mathbb{R}^{2}\right)}\lesssim_{p}\left[\mathrm{diam}\left(\mathrm{supp}f\right)\right]^{2p}\left|\left|f\right|\right|^{p}_{L^{p}\left(\mathbb{R}^{2}\right)}.
\]
This leads directly to the statement.
\end{proof}

\begin{prop}[Theorem 3 of Chapter 2 of \cite{Stein}]
\label{prop:Stein}Let $\Omega$ be homogeneous of degree $0$ and
suppose that $\Omega$ satisfies the cancellation property
\[
\int_{\mathbb{S}^{n-1}}\Omega\left(x\right)\mathrm{d}\sigma=0,
\]
and the smoothness property
\[
\int^{1}_{0}\frac{\omega\left(\delta\right)}{\delta}\mathrm{d}\delta<\infty,\quad\omega\left(\delta\right)\coloneqq\sup_{\substack{\left|\left|x-y\right|\right|_{2}\leq\delta\\
\left|\left|x\right|\right|_{2}=\left|\left|y\right|\right|_{2}=1
}
}\left|\Omega\left(x\right)-\Omega\left(y\right)\right|.
\]
For $1<p<\infty$ and $f\in L^{p}\left(\mathbb{R}^{n}\right)$, let
\[
T_{\varepsilon}\left(f\right)\left(x\right)=\int_{\left\{ \left|\left|y\right|\right|_{2}\geq\varepsilon\right\} }\frac{\Omega\left(y\right)}{\left|\left|y\right|\right|^{n}_{2}}f\left(x-y\right)\mathrm{d}y.
\]
Then,
\begin{enumerate}
\item ~
\[
\left|\left|T_{\varepsilon}\left(f\right)\right|\right|_{L^{p}\left(\mathbb{R}^{n}\right)}\lesssim_{n,p}\left|\left|f\right|\right|_{L^{p}\left(\mathbb{R}^{n}\right)},
\]
\item $T\left(f\right)\coloneqq\lim_{\varepsilon\to0}T_{\varepsilon}\left(f\right)$
exists in $L^{p}\left(\mathbb{R}^{n}\right)$ and
\[
\left|\left|T\left(f\right)\right|\right|_{L^{p}\left(\mathbb{R}^{n}\right)}\lesssim_{n,p}\left|\left|f\right|\right|_{L^{p}\left(\mathbb{R}^{n}\right)}.
\]
\end{enumerate}
\end{prop}
\begin{cor}
\label{cor:Stein}Consider the equation $\Delta u=f$, where $f:\mathbb{R}^{2}\to\mathbb{R}$
and $u$ is computed via the fundamental solution of the Laplace equation.
Let $p\in\left(1,\infty\right)$. Then,
\[
\left|\left|\mathrm{D}^{2}u\right|\right|_{L^{p}\left(\mathbb{R}^{2}\right)}\lesssim_{p}\left|\left|f\right|\right|_{L^{p}\left(\mathbb{R}^{2}\right)}.
\]
\end{cor}
\begin{proof}
Differentiating under the integral sign twice in \eqref{eq:u Laplace},
we arrive at
\[
\mathrm{D}^{2}u\left(x\right)=\frac{1}{2\pi}\int_{\mathbb{R}^{2}}\frac{\text{id}_{2\times2}-2\frac{y\otimes y}{\left|\left|y\right|\right|^{2}_{2}}}{\left|\left|y\right|\right|^{2}_{2}}f\left(x-y\right)\mathrm{d}y.
\]
Clearly, $\Omega_{ij}\left(y\right)=\left(\text{id}_{2\times2}-2\frac{y\otimes y}{\left|\left|y\right|\right|^{2}_{2}}\right)_{ij}$
is smooth except at the origin $\forall\left(i,j\right)\in\left\{ 1,2\right\} ^{2}$.
Since the straight line that joins two points in the definition of
$\omega\left(\delta\right)$ never crosses the origin if $\delta\leq1$,
we will have $\omega\left(\delta\right)\lesssim\delta$, which makes
the smoothness property trivial. Furthermore, an easy trigonometric
exercise shows that $\int_{\mathbb{S}^{1}}\Omega_{ij}\left(y\right)\mathrm{d}\sigma=0$
$\forall\left(i,j\right)\in\left\{ 1,2\right\} ^{2}$. Thus, Proposition
\ref{prop:Stein} is applicable, which provides the result.
\end{proof}

\begin{prop}[Gagliardo-Nirenberg]
\label{prop:Gagliardo-Nirenberg} \cite{Gagliardo-Nirenberg} Let
$r,q\in\left[1,\infty\right]$, $p\in\left[1,\infty\right)$, $j,m\in\mathbb{N}_{0}$
such that $j<m$ and $\theta\in\left[0,1\right]$ so that
\[
\frac{1}{p}=\frac{j}{n}+\theta\left(\frac{1}{r}-\frac{m}{n}\right)+\frac{1-\theta}{q}.
\]
Then,
\[
\left|\left|\mathrm{D}^{j}u\right|\right|_{L^{p}\left(\mathbb{R}^{n}\right)}\lesssim_{j,m,n,q,r,\theta}\left|\left|\mathrm{D}^{m}u\right|\right|^{\theta}_{L^{r}\left(\mathbb{R}^{n}\right)}\left|\left|u\right|\right|^{1-\theta}_{L^{q}\left(\mathbb{R}^{n}\right)}
\]
for any $u\in L^{q}\left(\mathbb{R}^{n}\right)$ such that $\mathrm{D}^{m}u\in L^{r}\left(\mathbb{R}^{n}\right)$,
with two exceptional cases:
\begin{enumerate}
\item if $j=0$ (with the understanding that $\mathrm{D}^{0}u=u$), $q=\infty$
and $rm<n$, then an additional assumption is needed: either $u$
tends to $0$ at infinity or $u\in L^{s}\left(\mathbb{R}^{n}\right)$
for some finite value of $s$;
\item if $r>1$ and $m-j-\frac{n}{r}\geq0$, then the additional assumption
$\frac{j}{m}\leq\theta<1$ (notice the strict inequality) is needed. 
\end{enumerate}
\end{prop}
\begin{cor}
\label{cor:estimates Laplacian}Consider the same assumptions as in
Proposition \ref{prop:Laplacian estimate L^p}. Then,
\begin{enumerate}
\item ~
\[
\left|\left|\nabla u\right|\right|_{L^{p}\left(\mathbb{R}^{2}\right)}\lesssim_{p}\mathrm{diam}\left(\mathrm{supp}f\right)\left|\left|f\right|\right|_{L^{p}\left(\mathbb{R}^{2}\right)},
\]
\item ~
\[
\left|\left|\nabla u\right|\right|_{W^{1,p}\left(\mathbb{R}^{2}\right)}\lesssim_{p}\left[1+\mathrm{diam}\left(\mathrm{supp}f\right)\right]\left|\left|f\right|\right|_{L^{p}\left(\mathbb{R}^{2}\right)},
\]
\item ~
\[
\left|\left|\nabla u\right|\right|_{C^{1-\frac{2}{p}}\left(\mathbb{R}^{2}\right)}\lesssim_{p}\left[1+\mathrm{diam}\left(\mathrm{supp}f\right)\right]\left|\left|f\right|\right|_{L^{p}\left(\mathbb{R}^{2}\right)},
\]
\item ~
\[
\left|\left|\nabla u\right|\right|_{L^{2}\left(\mathbb{R}^{2}\right)}\lesssim_{r,q}\left[\mathrm{diam}\left(\mathrm{supp}f\right)\right]^{2\frac{\left(r-1\right)\frac{q}{r}+\frac{1}{r}-\frac{1}{q}}{\left(r-1\right)\frac{q}{r}+1}}\left|\left|f\right|\right|_{L^{q}\left(\mathbb{R}^{2}\right)}
\]
$\forall q>2$ and $\forall r\in\left(1,2\right)$.
\end{enumerate}
\end{cor}
\begin{proof}
~
\begin{enumerate}
\item Take $p\leftarrow p$, $r\leftarrow p$, $q\leftarrow p$, $n\leftarrow2$,
$j\leftarrow1$ and $m\leftarrow2$ in Proposition \ref{prop:Gagliardo-Nirenberg}.
These choices lead to
\[
\frac{1}{p}=\frac{1}{2}+\theta\left(\frac{1}{p}-1\right)+\frac{1-\theta}{p}\iff0=\frac{1}{2}-\theta\iff\theta=\frac{1}{2},
\]
\[
\left|\left|\nabla u\right|\right|_{L^{p}\left(\mathbb{R}^{2}\right)}\lesssim_{p}\left|\left|\mathrm{D}^{2}u\right|\right|^{\frac{1}{2}}_{L^{p}\left(\mathbb{R}^{2}\right)}\left|\left|u\right|\right|^{\frac{1}{2}}_{L^{p}\left(\mathbb{R}^{2}\right)}.
\]
Notice that $m-j-\frac{n}{r}\leftarrow2-1-\frac{2}{p}=1-\frac{2}{p}>0$
since $p>2$, so the additional assumption $\frac{j}{m}\leq\theta<1$
is needed. In our case, we have $\frac{1}{2}\leq\frac{1}{2}<1$, which
is true. Using Proposition \ref{prop:Laplacian estimate L^p} and
Corollary \ref{cor:Stein}, we arrive at the statement.
\item By Point 1,
\[
\left|\left|\nabla u\right|\right|_{L^{p}\left(\mathbb{R}^{2}\right)}\lesssim_{p}\mathrm{diam}\left(\mathrm{supp}f\right)\left|\left|f\right|\right|_{L^{p}\left(\mathbb{R}^{2}\right)}.
\]
And, by Corollary \ref{cor:Stein},
\[
\left|\left|\mathrm{D}^{2}u\right|\right|_{L^{p}\left(\mathbb{R}^{2}\right)}\lesssim_{p}\left|\left|f\right|\right|_{L^{p}\left(\mathbb{R}^{2}\right)}.
\]
From these, the statement clearly follows.
\item This follows immediately from Point 2 by applying Morrey's inequality.
\item Take $p\leftarrow2$, $r\leftarrow r$, $q\leftarrow q$, $n\leftarrow2$,
$j\leftarrow1$ and $m\leftarrow2$ in Proposition \ref{prop:Gagliardo-Nirenberg}.
These choices lead to
\[
\frac{1}{2}=\frac{1}{2}+\theta\left(\frac{1}{r}-1\right)+\frac{1-\theta}{q}\iff\theta\left(\frac{1}{r}-\frac{1}{q}-1\right)+\frac{1}{q}=0\iff\theta=\frac{\frac{1}{q}}{1+\frac{1}{q}-\frac{1}{r}}=\frac{1}{\left(r-1\right)\frac{q}{r}+1},
\]
\[
\left|\left|\nabla u\right|\right|_{L^{2}\left(\mathbb{R}^{2}\right)}\lesssim_{r,q}\left|\left|\mathrm{D}^{2}u\right|\right|^{\frac{1}{\left(r-1\right)\frac{q}{r}+1}}_{L^{r}\left(\mathbb{R}^{2}\right)}\left|\left|u\right|\right|^{\frac{\left(r-1\right)\frac{q}{r}}{\left(r-1\right)\frac{q}{r}+1}}_{L^{q}\left(\mathbb{R}^{2}\right)}.
\]
Notice that $m-j-\frac{n}{r}\leftarrow2-1-\frac{2}{r}<0$ because
$r\in\left(1,2\right)$, so no additional assumption is needed. Using
Proposition \ref{prop:Laplacian estimate L^p} and Corollary \ref{cor:Stein},
we arrive at
\[
\left|\left|\nabla u\right|\right|_{L^{2}\left(\mathbb{R}^{2}\right)}\lesssim_{r,q}\left[\mathrm{diam}\left(\mathrm{supp}f\right)\right]^{2\frac{\left(r-1\right)\frac{q}{r}}{\left(r-1\right)\frac{q}{r}+1}}\left|\left|f\right|\right|^{\frac{1}{\left(r-1\right)\frac{q}{r}+1}}_{L^{r}\left(\mathbb{R}^{2}\right)}\left|\left|f\right|\right|^{\frac{\left(r-1\right)\frac{q}{r}}{\left(r-1\right)\frac{q}{r}+1}}_{L^{q}\left(\mathbb{R}^{2}\right)}.
\]
Since $f$ has compact support (see assumptions of Proposition \ref{prop:Laplacian estimate L^p}),
Hölder's inequality provides
\[
\left|\left|f\right|\right|_{L^{r}\left(\mathbb{R}^{2}\right)}\leq\left|\text{supp}f\right|^{\left(1-\frac{r}{q}\right)\frac{1}{r}}\left|\left|f\right|\right|_{L^{q}\left(\mathbb{R}^{2}\right)}\lesssim_{r,q}\left[\mathrm{diam}\left(\mathrm{supp}f\right)\right]^{2\left(\frac{1}{r}-\frac{1}{q}\right)}\left|\left|f\right|\right|_{L^{q}\left(\mathbb{R}^{2}\right)}.
\]
And, consequently,
\[
\left|\left|\nabla u\right|\right|_{L^{2}\left(\mathbb{R}^{2}\right)}\lesssim_{r,q}\left[\text{diam}\left(\text{supp}f\right)\right]^{2\frac{\left(r-1\right)\frac{q}{r}+\frac{1}{r}-\frac{1}{q}}{\left(r-1\right)\frac{q}{r}+1}}\left|\left|f\right|\right|_{L^{q}\left(\mathbb{R}^{2}\right)}.
\]
\end{enumerate}
\end{proof}

\begin{prop}
\label{prop:Laplace D2u Holder}Let $\alpha\in\left(0,1\right)$.
Consider the equation $\Delta u=f$, where $f\in C^{\alpha}\left(\mathbb{R}^{2}\right)$
and $u$ is computed via the fundamental solution of the Laplace equation.
Then,
\[
\left|\left|\mathrm{D}^{2}u\right|\right|_{\dot{C}^{\alpha}\left(\mathbb{R}^{2}\right)}\lesssim_{\alpha}\left|\left|f\right|\right|_{\dot{C}^{\alpha}\left(\mathbb{R}^{2}\right)}.
\]
Furthermore, if $f\in L^{p}\left(\mathbb{R}^{2}\right)$ for some
$p>1$,
\[
\left|\left|\mathrm{D}^{2}u\right|\right|_{L^{\infty}\left(\mathbb{R}^{2}\right)}\lesssim\frac{1}{\alpha}\left|\left|f\right|\right|_{\dot{C}^{\alpha}\left(\mathbb{R}^{2}\right)}+\left(\frac{p-1}{2}\right)^{\frac{p-1}{p}}\left|\left|f\right|\right|_{L^{p}\left(\mathbb{R}^{2}\right)}.
\]
\end{prop}
\begin{proof}
The first statement is exactly Theorem 3.4.1 of \cite{krylov}.

Now, for the proof of the second statement, differentiating under
the integral sign twice in \eqref{eq:u Laplace}, we arrive at
\[
\mathrm{D}^{2}u\left(x\right)=\int_{\mathbb{R}^{2}}\frac{\Omega\left(x-z\right)}{\left|\left|x-z\right|\right|^{2}_{2}}f\left(z\right)\mathrm{d}z,\quad\Omega\left(w\right)\coloneqq\frac{1}{2\pi}\left[\text{id}_{2\times2}-2\frac{w\otimes w}{\left|\left|w\right|\right|^{2}_{2}}\right].
\]
As argued in the proof of Corollary \ref{cor:Stein}, we have $\int_{\mathbb{S}^{1}}\Omega\left(w\right)\mathrm{d}\sigma=0$.
Thus,
\[
\mathrm{D}^{2}u\left(x\right)=\int_{\mathbb{R}^{2}}\frac{\Omega\left(x-z\right)}{\left|\left|x-z\right|\right|^{2}_{2}}\left(f\left(z\right)-f\left(x\right)\right)\mathrm{d}z.
\]
We split the domain of integration in $\mathbb{R}^{2}=\left\{ \left|\left|x-z\right|\right|_{2}\leq1\right\} \cup\left\{ \left|\left|x-z\right|\right|_{2}>1\right\} .$
In this way,
\[
\mathrm{D}^{2}u\left(x\right)=\underbrace{\int_{\left\{ \left|\left|x-z\right|\right|_{2}\leq1\right\} }\frac{\Omega\left(x-z\right)}{\left|\left|x-z\right|\right|^{2}_{2}}\left(f\left(z\right)-f\left(x\right)\right)\mathrm{d}z}_{\eqqcolon I_{1}}+\underbrace{\int_{\left\{ \left|\left|x-z\right|\right|_{2}>1\right\} }\frac{\Omega\left(x-z\right)}{\left|\left|x-z\right|\right|^{2}_{2}}\left(f\left(z\right)-f\left(x\right)\right)\mathrm{d}z}_{\eqqcolon I_{2}}.
\]
On the one hand, to bound $I_{1}$, we use that $f$ is Hölder-continuous.
Since $\left|\left|\Omega\right|\right|_{L^{\infty}\left(\mathbb{R}^{2}\right)}\lesssim1$,
\[
\left|\left|I_{1}\right|\right|_{\infty}\lesssim\left|\left|f\right|\right|_{\dot{C}^{\alpha}\left(\mathbb{R}^{2}\right)}\int_{\left\{ \left|\left|x-z\right|\right|_{2}\leq1\right\} }\frac{\mathrm{d}z}{\left|\left|x-z\right|\right|^{2-\alpha}_{2}}\lesssim\left|\left|f\right|\right|_{\dot{C}^{\alpha}\left(\mathbb{R}^{2}\right)}\int^{1}_{0}\frac{\mathrm{d}\rho}{\rho^{1-\alpha}}=\frac{1}{\alpha}\left|\left|f\right|\right|_{\dot{C}^{\alpha}\left(\mathbb{R}^{2}\right)}.
\]
On the other hand, to bound $I_{2}$, we use
\begin{enumerate}
\item $\int_{\mathbb{S}^{1}}\Omega\left(w\right)\mathrm{d}\sigma=0$ to
justify that $\int_{\left\{ \left|\left|x-z\right|\right|_{2}>1\right\} }\frac{\Omega\left(x-z\right)}{\left|\left|x-z\right|\right|^{2}_{2}}f\left(x\right)\mathrm{d}z=0$,
\item $\left|\left|\Omega\left(z\right)\right|\right|_{\infty}\lesssim1$,
\item the fact that $f\in L^{p}\left(\mathbb{R}^{2}\right)$ and Hölder's
inequality.
\end{enumerate}
In this way, we obtain
\[
\begin{aligned}\left|\left|I_{2}\right|\right|_{\infty} & \lesssim\left(\int_{\left\{ \left|\left|x-z\right|\right|_{2}>1\right\} }\left(\frac{1}{\left|\left|x-z\right|\right|^{2}_{2}}\right)^{\frac{p}{p-1}}\mathrm{d}z\right)^{1-\frac{1}{p}}\left|\left|f\right|\right|_{L^{p}\left(\mathbb{R}^{2}\right)}\lesssim\\
 & \lesssim\left(\int^{\infty}_{1}\frac{\mathrm{d}\rho}{\rho^{2\frac{p}{p-1}-1}}\right)^{1-\frac{1}{p}}\left|\left|f\right|\right|_{L^{p}\left(\mathbb{R}^{2}\right)}=\left(\frac{1}{\frac{2p}{p-1}-2}\right)^{1-\frac{1}{p}}\left|\left|f\right|\right|_{L^{p}\left(\mathbb{R}^{2}\right)}=\left(\frac{p-1}{2}\right)^{\frac{p-1}{p}}\left|\left|f\right|\right|_{L^{p}\left(\mathbb{R}^{2}\right)}.
\end{aligned}
\]
These bounds directly lead to the statement.
\end{proof}

\section{\label{sec:fixed-point argument}The fixed-point argument}

\subsection{Definition of the map}

To carry out the fixed-point argument, we will define a map where
we replace $\nabla^{\perp}\varPsi$ in \eqref{eq:vorticity equation final first attempt} and \eqref{eq:g1 first attempt}
by a function $a\left(t,x\right)$.
\begin{defn}
\label{def:THE MAP}Let $\beta\in\left[0,1\right)$. We consider the
map
\[
\begin{matrix}T: & C^{0}_{t}C^{\beta}_{x}\left(\left[0,1\right]\times\mathbb{R}^{2};\mathbb{R}^{2}\right) & \longrightarrow & T\left(C^{0}_{t}C^{\beta}_{x}\left(\left[0,1\right]\times\mathbb{R}^{2};\mathbb{R}^{2}\right)\right)\\
 & a & \longrightarrow & \nabla^{\perp}\varPsi,
\end{matrix}
\]
where $\varPsi$ is obtained from solving the equation
\begin{equation}
\begin{aligned}\frac{\partial\rho_{B}}{\partial x_{2}}\left(1+\frac{1}{\overline{\rho}+\rho_{B}}\left(g_{1}\left(t\right)\varPhi\left(x\right)+a_{1}\left(t,x\right)-\left(\frac{\partial u_{B}}{\partial t}+u_{B}\cdot\nabla u_{B}\right)_{1}\right)\right)+\\
-\frac{1}{\overline{\rho}+\rho_{B}}\frac{\partial\rho_{B}}{\partial x_{1}}\left(g_{2}\left(t\right)\varPhi\left(x\right)+a_{2}\left(t,x\right)-\left(\frac{\partial u_{B}}{\partial t}+u_{B}\cdot\nabla u_{B}\right)_{2}\right)+\\
+f_{\omega_{B}}-g\left(t\right)\cdot\nabla^{\perp}\varPhi\left(x\right) & =\Delta\varPsi.
\end{aligned}
\label{eq:varpsi}
\end{equation}
via a convolution with the fundamental solution of the Laplace equation
and where $g\left(t\right)$ is given by
\begin{equation}
g\left(t\right)=\left(\frac{\partial u_{B}}{\partial t}+u_{B}\cdot\nabla u_{B}\right)\left(t,0\right)-a\left(t,0\right)-\left(\overline{\rho}+\rho_{B}\left(t,0\right)\right)\hat{e}_{1}.\label{eq:g}
\end{equation}
\end{defn}

\subsection{Regularity improvements}

In this section, we will see that $T$ actually improves the regularity,
i.e., $T\left(a\right)$ will be more regular than $a$.
\begin{prop}
\label{prop:MAP wins regularity 1}Provided that $\overline{\rho}>\left|\left|\rho_{B}\right|\right|_{C^{0}_{t}C^{0}_{x}\left(\left[0,1\right]\times\mathbb{R}^{2}\right)}$,
$\delta\leq\Upsilon_{\delta},$ $\mu\leq\Upsilon_{\mu}$ and $1-\gamma\leq\Upsilon_{\gamma}$,
where $\delta$, $\mu$ and $\gamma$ are the parameters of the solution
of Boussinesq defined in \cite{Articulo Boussinesq}, we have
\[
\begin{matrix}T: &  & \left\{ a\in C^{0}_{t}C^{0}_{x}\left(\left[0,1\right]\times\mathbb{R}^{2};\mathbb{R}^{2}\right):a_{1}\text{ even in }x_{2},\;a_{2}\text{ odd in }x_{2}\right\}  & \longrightarrow\\
 & \longrightarrow & \left\{ a\in C^{0}_{t}C^{\frac{3}{4}}_{x}\left(\left[0,1\right]\times\mathbb{R}^{2};\mathbb{R}^{2}\right):a_{1}\text{ even in }x_{2},\;a_{2}\text{ odd in }x_{2}\right\} .
\end{matrix}
\]
Moreover, the function $g\left(t\right)$ given by equation \eqref{eq:g}
is $C\left(\left[0,1\right]\right)$.
\end{prop}
\begin{proof}
Notice that, in view of Corollary \ref{cor:estimates Laplacian} and
the linearity of the Laplace equation, to obtain the regularity improvement
it is enough to prove that the left-hand-side of \eqref{eq:varpsi}
is $C^{0}_{t}L^{8}_{x}\left(\left[0,1\right]\times\mathbb{R}^{2}\right)$,
is compactly supported uniformly in time and has zero mean. We will
choose $\delta\leq\Upsilon_{\delta,1}$, $\mu\leq\Upsilon_{\mu,1}$
and $1-\gamma\leq\Upsilon_{\gamma,1}$ sufficiently small so that
$\rho\in C^{0}_{t}C^{0}_{x}\left(\left[0,1\right]\times\mathbb{R}^{2}\right)$
(see Proposition \ref{prop:bounded density}), $\frac{\partial u_{B}}{\partial t}+u_{B}\cdot\nabla u_{B}\in C^{0}_{t}C^{0}_{x}\left(\left[0,1\right]\times\mathbb{R}^{2}\right)$
(see Proposition \ref{prop:bound material derivative u}), $\frac{\partial\rho_{B}}{\partial x_{1}}\in C^{0}_{t}C^{0}_{x}\left(\left[0,1\right]\times\mathbb{R}^{2}\right)$
(see Proposition \ref{prop:bounded drhodx1}) and $\frac{\partial\rho_{B}}{\partial x_{2}}\in C^{0}_{t}L^{8}_{x}\left(\left[0,1\right]\times\mathbb{R}^{2}\right)$
(see Proposition \ref{prop:bad bound drhodx2}).
\begin{enumerate}
\item \label{enu:LHS odd}The left-hand-side of \eqref{eq:varpsi} is odd
in $x_{2}$. First, notice that, since $a_{2}\left(t,x\right)$ is
odd in $x_{2}$ and $\left(\frac{\partial u_{B}}{\partial t}+u_{B}\cdot\nabla u_{B}\right)_{2}\left(t,x\right)$
is also odd in $x_{2}$ by Proposition \ref{prop:oddness and evenness U},
equation \eqref{eq:g} tells us that $g_{2}\left(t\right)\equiv0$.
Then,
\[
-g\left(t\right)\cdot\nabla^{\perp}\varPhi\left(x\right)=g_{1}\left(t\right)\frac{\partial\varPhi}{\partial x_{2}}\left(x\right),
\]
which is odd in $x_{2}$ because $\varPhi$ is even in $x_{2}$ (see
Choice \ref{choice: fu}). Moreover, $f_{\omega_{B}}$ is odd in $x_{2}$
because of the Main Theorem of \cite{Articulo Boussinesq} (see Theorem
\ref{thm:MAIN THM Boussinesq}). Besides, since $g_{2}\left(t\right)\equiv0$,
$a_{2}\left(t,x\right)$ is odd in $x_{2}$ and $\left(\frac{\partial u_{B}}{\partial t}+u_{B}\cdot\nabla u_{B}\right)_{2}$
is also odd in $x_{2}$ by Proposition \ref{prop:oddness and evenness U},
the multiplicative factor to the right of $\frac{\partial\rho_{B}}{\partial x_{1}}$
is odd in $x_{2}$. $\frac{\partial\rho_{B}}{\partial x_{1}}$ itself
is even in $x_{2}$ because $\rho_{B}$ is even in $x_{2}$ according
to Proposition \ref{prop:evenness rho}. By this result, we also deduce
that $\frac{1}{\overline{\rho}+\rho_{B}}$ is even in $x_{2}$. Consequently,
the second summand of the left-hand-side of \eqref{eq:varpsi} is
odd in $x_{2}$. Concerning the first summand, $\frac{\partial\rho_{B}}{\partial x_{2}}$
is odd in $x_{2}$ because $\rho_{B}$ is even in $x_{2}$ according
to Proposition \ref{prop:evenness rho}. In addition, the factor that
multiplies $\frac{\partial\rho_{B}}{\partial x_{2}}$ is even in $x_{2}$
because $1$ is even in $x_{2}$, $\rho_{B}$ is even in $x_{2}$
(see Proposition \ref{prop:evenness rho}), $\varPhi$ is even in
$x_{2}$ (see Choice \ref{choice: fu}), $a_{1}\left(t,x\right)$
is even in $x_{2}$ and $\left(\frac{\partial u_{B}}{\partial t}+u_{B}\cdot\nabla u_{B}\right)_{1}$
is even in $x_{2}$ thanks to Proposition \ref{prop:oddness and evenness U}.
In conclusion, the first summand is odd in $x_{2}$, as are all other
summands, which makes the left-hand-side of \eqref{eq:varpsi} odd
in $x_{2}$.
\item \label{enu:LHS compact support}The left-hand-side of \eqref{eq:varpsi}
is compactly supported uniformly in time. By the Main Theorem of \cite{Articulo Boussinesq}
(see Theorem \ref{thm:MAIN THM Boussinesq}), $f_{\omega_{B}}$ is
compactly supported uniformly in time. Likewise, by subsection \ref{subsec:summary Boussinesq}
(or Proposition 24 of \cite{Articulo Boussinesq}), $\rho_{B}$ is
compactly supported uniformly in time (and so are its derivatives).
Since $\varPhi$ is also compactly supported and $\varPhi$ does not
depend on time, the left-hand-side of \eqref{eq:varpsi} must be compactly
supported uniformly in time as well.
\item \label{enu:LHS mean-free}The left-hand-side of \eqref{eq:varpsi}
has zero mean, because it is odd in $x_{2}$ (see point \ref{enu:LHS odd})
and it is compactly supported (see point \ref{enu:LHS compact support}).
\item \label{enu:evenness and oddness T}$T\left(a\right)_{1}$ is even
in $x_{2}$ and $T\left(a\right)_{2}$ is odd in $x_{2}$. Since the
left-hand-side of \eqref{eq:varpsi} is odd in $x_{2}$ by point \ref{enu:LHS odd},
so is $\varPsi\left(t,x\right)$. As a consequence, $T\left(a\right)_{1}\left(t,x\right)=-\frac{\partial\varPsi}{\partial x_{2}}\left(t,x\right)$
is even in $x_{2}$ and $T\left(a\right)_{2}\left(t,x\right)=\frac{\partial\varPsi}{\partial x_{1}}\left(t,x\right)$
is odd in $x_{2}$, as required.
\item \label{enu:g continuous in time}$g\in C\left(\left[0,1\right]\right)$.
By Proposition \ref{prop:bound material derivative u}, $\left(\frac{\partial u_{B}}{\partial t}+u_{B}\cdot\nabla u_{B}\right)\left(t,0\right)\in C\left(\left[0,1\right]\right)$;
by Proposition \ref{prop:bounded density}, $\rho_{B}\left(t,0\right)\in C\left(\left[0,1\right]\right)$;
and, by hypothesis, $a\left(t,0\right)\in C\left(\left[0,1\right]\right)$.
As a result, looking at equation \eqref{eq:g}, $g\in C\left(\left[0,1\right]\right)$.
\item \label{enu:LHS L8}The left-hand-side of \eqref{eq:varpsi} is $C^{0}_{t}L^{8}_{x}\left(\left[0,1\right]\times\mathbb{R}^{2}\right)$. 
\begin{itemize}
\item $f_{\omega_{B}}\in C^{0}_{t}C^{\alpha}_{x,c}\left(\left[0,1\right]\times\mathbb{R}^{2}\right)$
by the Main Theorem of \cite{Articulo Boussinesq} (see Theorem \ref{thm:MAIN THM Boussinesq}),
so $f_{\omega_{B}}\in C^{0}_{t}L^{8}_{x}\left(\left[0,1\right]\times\mathbb{R}^{2}\right)$.
\item $g\in C\left(\left[0,1\right]\right)$ by point \ref{enu:g continuous in time}
and $\varPhi\in C^{\infty}_{c}\left(\mathbb{R}^{2}\right)$ by Choice
\ref{choice: fu}, so $g\left(t\right)\cdot\nabla^{\perp}\varPhi\left(x\right)\in C^{0}_{t}L^{8}_{x}\left(\left[0,1\right]\times\mathbb{R}^{2}\right)$.
\item The factor that multiplies $\frac{\partial\rho_{B}}{\partial x_{2}}$
is $C^{0}_{t}C^{0}_{x}\left(\left[0,1\right]\times\mathbb{R}^{2}\right)$
because $\frac{1}{\overline{\rho}+\rho_{B}}\in C^{0}_{t}C^{0}_{x}\left(\left[0,1\right]\times\mathbb{R}^{2}\right)$
by Lemma \ref{lem:inverse Holder} since $\rho_{B}\in C^{0}_{t}C^{0}_{x}\left(\left[0,1\right]\times\mathbb{R}^{2}\right)$
by Proposition \ref{prop:bounded density} and $\overline{\rho}>\left|\left|\rho_{B}\right|\right|_{C^{0}_{t}C^{0}_{x}\left(\left[0,1\right]\times\mathbb{R}^{2}\right)}$
by hypothesis, $g_{1}\left(t\right)\varPhi\left(x\right)\in C^{0}_{t}C^{0}_{x}\left(\left[0,1\right]\times\mathbb{R}^{2}\right)$
as $g\in C\left(\left[0,1\right]\right)$ by point \ref{enu:g continuous in time}
and $\varPhi\in C^{\infty}\left(\mathbb{R}^{2}\right)$ (see Choice
\ref{choice: fu}), $a\left(t,x\right)\in C^{0}_{t}C^{0}_{x}\left(\left[0,1\right]\times\mathbb{R}^{2}\right)$
by hypothesis and $\left(\frac{\partial u_{B}}{\partial t}+u_{B}\cdot\nabla u_{B}\right)_{1}\in C^{0}_{t}C^{0}_{x}\left(\left[0,1\right]\times\mathbb{R}^{2}\right)$
by Proposition \ref{prop:bound material derivative u}.
\item Likewise, the factor that multiplies $\frac{\partial\rho_{B}}{\partial x_{1}}\left(t,x\right)$
is $C^{0}_{t}C^{0}_{x}\left(\left[0,1\right]\times\mathbb{R}^{2}\right)$
because $\frac{1}{\overline{\rho}+\rho_{B}}\in C^{0}_{t}C^{0}_{x}\left(\left[0,1\right]\times\mathbb{R}^{2}\right)$
by Lemma \ref{lem:inverse Holder} since $\rho_{B}\in C^{0}_{t}C^{0}_{x}\left(\left[0,1\right]\times\mathbb{R}^{2}\right)$
by Proposition \ref{prop:bounded density} and $\overline{\rho}>\left|\left|\rho_{B}\right|\right|_{C^{0}_{t}C^{0}_{x}\left(\left[0,1\right]\times\mathbb{R}^{2}\right)}$
by hypothesis, $g_{2}\left(t\right)\varPhi\left(x\right)\in C^{0}_{t}C^{0}_{x}\left(\left[0,1\right]\times\mathbb{R}^{2}\right)$
as $g\in C\left(\left[0,1\right]\right)$ by point \ref{enu:g continuous in time}
and $\varPhi\in C^{\infty}\left(\mathbb{R}^{2}\right)$ (see Choice
\ref{choice: fu}), $a\left(t,x\right)\in C^{0}_{t}C^{0}_{x}\left(\left[0,1\right]\times\mathbb{R}^{2}\right)$
by hypothesis and $\left(\frac{\partial u_{B}}{\partial t}+u_{B}\cdot\nabla u_{B}\right)_{2}\in C^{0}_{t}C^{0}_{x}\left(\left[0,1\right]\times\mathbb{R}^{2}\right)$
by Proposition \ref{prop:bound material derivative u}.
\item We have seen in the last two points that the factors that multiply
$\frac{\partial\rho_{B}}{\partial x_{1}}$ and $\frac{\partial\rho_{B}}{\partial x_{2}}$
are $C^{0}_{t}C^{0}_{x}\left(\left[0,1\right]\times\mathbb{R}^{2}\right)$.
Now, by Proposition \ref{prop:bounded drhodx1}, $\frac{\partial\rho}{\partial x_{1}}\in C^{0}_{t}C^{0}_{x}\left(\left[0,1\right]\times\mathbb{R}^{2}\right)$;
and, by subsection \ref{subsec:summary Boussinesq} (or by Proposition
24 of \cite{Articulo Boussinesq}), it is also compactly supported
uniformly in time. Hence, $\frac{\partial\rho_{B}}{\partial x_{1}}\in C^{0}_{t}L^{8}_{x}\left(\left[0,1\right]\times\mathbb{R}^{2}\right)$.
Moreover, by Proposition \ref{prop:bad bound drhodx2}, $\frac{\partial\rho_{B}}{\partial x_{2}}\in C^{0}_{t}L^{8}_{x}\left(\left[0,1\right]\times\mathbb{R}^{2}\right)$.
Since the product of a function in $C^{0}_{t}L^{8}_{x}\left(\left[0,1\right]\times\mathbb{R}^{2}\right)$
times a function in $C^{0}_{t}L^{\infty}_{x}\left(\left[0,1\right]\times\mathbb{R}^{2}\right)$
belongs to $C^{0}_{t}L^{8}_{x}\left(\left[0,1\right]\times\mathbb{R}^{2}\right)$,
we conclude that the first two summands of \eqref{eq:varpsi} are
$C^{0}_{t}L^{8}_{x}\left(\left[0,1\right]\times\mathbb{R}^{2}\right)$.
\end{itemize}
\end{enumerate}
\end{proof}

\begin{prop}
\label{prop:MAP wins regularity 2}Let $\alpha\in\left(0,\alpha_{*}\right)$
be the regularity parameter associated to the Boussinesq solution
$\left(\rho_{B},\omega_{B},f_{\rho_{B}},f_{\omega_{B}}\right)$ given
in \cite{Articulo Boussinesq}. Provided that $\overline{\rho}>\left|\left|\rho_{B}\right|\right|_{C^{0}_{t}C^{0}_{x}\left(\left[0,1\right]\times\mathbb{R}^{2}\right)}$,
$\delta\leq\Upsilon_{\delta},$ $\mu\leq\Upsilon_{\mu}$ and $1-\gamma\leq\Upsilon_{\gamma}$,
where $\delta$, $\mu$ and $\gamma$ are the parameters of the solution
of Boussinesq defined in \cite{Articulo Boussinesq},
\[
\begin{matrix}T: &  & \left\{ a\in C^{0}_{t}C^{\frac{3}{4}}_{x}\left(\left[0,1\right]\times\mathbb{R}^{2};\mathbb{R}^{2}\right):a_{1}\text{ even in }x_{2},\;a_{2}\text{ odd in }x_{2}\right\}  & \longrightarrow\\
 & \longrightarrow & \left\{ a\in C^{0}_{t}C^{1,\alpha}_{x}\left(\left[0,1\right]\times\mathbb{R}^{2};\mathbb{R}^{2}\right):a_{1}\text{ even in }x_{2},\;a_{2}\text{ odd in }x_{2}\right\} \cap C^{0}_{t}L^{2}_{x}\left(\left[0,1\right]\times\mathbb{R}^{2};\mathbb{R}^{2}\right).
\end{matrix}
\]
\end{prop}
\begin{proof}
We will choose $\delta\leq\Upsilon_{\delta,1}$, $\mu\leq\Upsilon_{\mu,1}$
and $1-\gamma\leq\Upsilon_{\gamma,1}$ sufficiently small so that
$\rho\in C^{0}_{t}C^{\frac{3}{4}}_{x}\left(\left[0,1\right]\times\mathbb{R}^{2}\right)$
(see Proposition \ref{prop:bounded density}), $\frac{\partial u_{B}}{\partial t}+u_{B}\cdot\nabla u_{B}\in C^{0}_{t}C^{\frac{3}{4}}_{x}\left(\left[0,1\right]\times\mathbb{R}^{2}\right)$
(see Proposition \ref{prop:bound material derivative u}) and $\frac{\partial\rho_{B}}{\partial x_{1}}\in C^{0}_{t}C^{0}_{x}$
(see Proposition \ref{prop:bounded drhodx1}).
\begin{enumerate}
\item \label{enu:last Proposition}Notice that, $a\in C^{0}_{t}C^{\frac{3}{4}}_{x}\left(\left[0,1\right]\times\mathbb{R}^{2};\mathbb{R}^{2}\right)\subseteq C^{0}_{t}C^{0}_{x}\left(\left[0,1\right]\times\mathbb{R}^{2};\mathbb{R}^{2}\right)$,
so we may apply Proposition \ref{prop:MAP wins regularity 1} and
the arguments in its proof to ensure that:
\begin{enumerate}
\item \label{enu:last Proposition evenness and oddness}the required evenness
and oddness are satisfied (see Points \ref{enu:LHS odd} and \ref{enu:evenness and oddness T}),
\item \label{enu:last Proposition odd mean-free}the left-hand-side of \eqref{eq:varpsi}
is compactly supported uniformly in time and mean-free (see Points
\ref{enu:LHS compact support} and \ref{enu:LHS mean-free}),
\item \label{enu:last Proposition Lp}the left-hand-side of \eqref{eq:varpsi}
is $C^{0}_{t}L^{8}_{x}\left(\left[0,1\right]\times\mathbb{R}^{2}\right)$
(see Point \ref{enu:LHS L8}),
\item \label{enu:last Proposition T}$T\left(a\right)\in C^{0}_{t}C^{\frac{3}{4}}_{x}\left(\left[0,1\right]\times\mathbb{R}^{2};\mathbb{R}^{2}\right)$.
\item \label{enu:last Proposition g}$g\left(t\right)\in C\left(\left[0,1\right]\right).$
\end{enumerate}
Moreover, by inspection of the proof of Proposition \ref{prop:MAP wins regularity 1},
Points \ref{enu:last Proposition odd mean-free}, \ref{enu:last Proposition Lp}, and \ref{enu:last Proposition T}
are also true summand-wise in the left-hand-side of equation \eqref{eq:varpsi}.
\item $T\left(a\right)\in C^{0}_{t}L^{2}_{x}\left(\left[0,1\right]\times\mathbb{R}^{2};\mathbb{R}^{2}\right)$.
By Point 4 of Corollary \ref{cor:estimates Laplacian} (taking $q\leftarrow8$
and $r\leftarrow\frac{3}{2}$) and the linearity of the Laplace equation,
we need to prove that the left-hand-side of \eqref{eq:varpsi} is
$C^{0}_{t}L^{8}_{x}\left(\left[0,1\right]\times\mathbb{R}^{2}\right)$,
is compactly supported uniformly in time and has zero mean. We satisfy
all these conditions by Point \ref{enu:last Proposition}.
\item For the rest of the proof, making use of the linearity of the Laplacian,
we will split $T$ into two operators $T=T^{\left(1\right)}+T^{\left(2\right)}$,
where $T^{\left(1\right)}\left(a\right)=\nabla^{\perp}\varPsi^{\left(1\right)}$,
$T^{\left(2\right)}\left(a\right)=\nabla^{\perp}\varPsi^{\left(2\right)}$
and
\[
\begin{aligned}\Delta\varPsi^{\left(1\right)} & =f_{\omega_{B}}-g\left(t\right)\cdot\nabla^{\perp}\varPhi\left(x\right),\\
\Delta\varPsi^{\left(2\right)} & =+\frac{\partial\rho_{B}}{\partial x_{2}}\left(1+\frac{1}{\overline{\rho}+\rho_{B}}\left(g_{1}\left(t\right)\varPhi\left(x\right)+a_{1}\left(t,x\right)-\left(\frac{\partial u_{B}}{\partial t}+u_{B}\cdot\nabla u_{B}\right)_{1}\right)\right)+\\
 & \quad-\frac{1}{\overline{\rho}+\rho_{B}}\frac{\partial\rho_{B}}{\partial x_{1}}\left(g_{2}\left(t\right)\varPhi\left(x\right)+a_{2}\left(t,x\right)-\left(\frac{\partial u_{B}}{\partial t}+u_{B}\cdot\nabla u_{B}\right)_{2}\right).
\end{aligned}
\]
\item \label{enu:LHS1 Calpha}$\Delta\varPsi^{\left(1\right)}\in C^{0}_{t}C^{\alpha}_{x}\left(\left[0,1\right]\times\mathbb{R}^{2}\right)$.
\begin{enumerate}
\item By the Main Theorem of \cite{Articulo Boussinesq} (see Theorem \ref{thm:MAIN THM Boussinesq}),
$f_{\omega_{B}}\in C^{0}_{t}C^{\alpha}_{x}\left(\left[0,1\right]\times\mathbb{R}^{2}\right)$.
\item Moreover, since $\varPhi\in C^{\infty}\left(\mathbb{R}^{2}\right)$
(see Choice \ref{choice: fu}) and $g\in C\left(\left[0,1\right]\right)$
by point \ref{enu:last Proposition g}, $g\left(t\right)\cdot\nabla^{\perp}\varPhi\left(x\right)\in C^{0}_{t}C^{\alpha}_{x}\left(\left[0,1\right]\times\mathbb{R}^{2}\right)$.
\end{enumerate}
\item \label{enu:T1 C1alpha}$T^{\left(1\right)}\left(a\right)\in C^{0}_{t}C^{1,\alpha}_{x}\left(\left[0,1\right]\times\mathbb{R}^{2};\mathbb{R}^{2}\right)$.
By Corollary \ref{cor:estimates Laplacian} and the linearity of the
Laplacian, we need $\Delta\varPsi^{\left(1\right)}\in C^{0}_{t}L^{p}_{x}\left(\left[0,1\right]\times\mathbb{R}^{2}\right)$
for some $p>2$, compactly supported uniformly in time and mean-free
in order to have $T^{\left(1\right)}\left(a\right)\in C^{0}_{t}C^{0}_{x}\left(\left[0,1\right]\times\mathbb{R}^{2};\mathbb{R}^{2}\right)$.
These conditions are met because of Point \ref{enu:last Proposition}.
Now, by Proposition \ref{prop:Laplace D2u Holder} and the linearity
of the Laplacian, we require $\Delta\varPsi^{\left(1\right)}\in C^{0}_{t}\dot{C}^{\varepsilon}_{x}\left(\left[0,1\right]\times\mathbb{R}^{2}\right)\cap C^{0}_{t}L^{p}_{x}\left(\left[0,1\right]\times\mathbb{R}^{2}\right)$
for some $\varepsilon>0$ and $p>1$ so as to have $T^{\left(1\right)}\left(a\right)\in C^{0}_{t}\dot{C}^{1}_{x}\left(\left[0,1\right]\times\mathbb{R}^{2};\mathbb{R}^{2}\right)$.
The required conditions are satisfied because of Points \ref{enu:last Proposition} and \ref{enu:LHS1 Calpha}.
Lastly, also by Proposition \ref{prop:Laplace D2u Holder}, we need
$\Delta\varPsi^{\left(1\right)}\in C^{0}_{t}\dot{C}^{\alpha}_{x}\left(\left[0,1\right]\times\mathbb{R}^{2}\right)$
in order for $T^{\left(1\right)}\left(a\right)$ to be $C^{0}_{t}\dot{C}^{1,\alpha}_{x}\left(\left[0,1\right]\times\mathbb{R}^{2};\mathbb{R}^{2}\right)$.
As we have seen in Point \ref{enu:LHS1 Calpha}, this condition is
also satisfied. In conclusion, we may ensure that $T^{\left(1\right)}\left(a\right)\in C^{0}_{t}C^{1,\alpha}_{x}\left(\left[0,1\right]\times\mathbb{R}^{2};\mathbb{R}^{2}\right)$.
\item \label{enu:LHS2 Calpha*}$\Delta\varPsi^{\left(2\right)}\in L^{\infty}_{t}C^{\alpha_{*}}_{x}\left(\left[0,1\right]\times\mathbb{R}^{2}\right)$. 
\begin{enumerate}
\item \label{enu:LHS2 Calpha* drhodx2}On the one hand, by Proposition \ref{prop:bad bound drhodx2},
taking $\delta\leq\Upsilon_{\delta,2}$ and $\mu\leq\Upsilon_{\mu,2}$
sufficiently small, we can ensure that
\[
\begin{aligned}\left|\left|\frac{\partial\rho}{\partial x_{2}}\left(t,\cdot\right)\right|\right|_{L^{\infty}\left(\mathbb{R}^{2}\setminus B\left(0;r\right)\right)} & \lesssim_{\delta,\varphi,\mu,Y}\frac{1}{r^{\frac{1}{2}}},\\
\left|\left|\frac{\partial\rho}{\partial x_{2}}\left(t,\cdot\right)\right|\right|_{\dot{C}^{\frac{1}{4}}\left(\mathbb{R}^{2}\setminus B\left(0;r\right)\right)} & \lesssim_{\delta,\varphi,\mu,Y}\frac{1}{r^{2\left(\left(\frac{3}{4}-\frac{1}{2}\right)\left(1+k_{\max}\right)+\frac{1}{16}\right)}}
\end{aligned}
\]
$\forall r\in\left(0,1\right]$ and $\forall t\in\left[0,1\right]$.
On the other hand, the factor that multiplies $\frac{\partial\rho_{B}}{\partial x_{2}}$
is $C^{0}_{t}C^{\frac{3}{4}}_{x}\left(\left[0,1\right]\times\mathbb{R}^{2}\right)$
because $\frac{1}{\overline{\rho}+\rho_{B}}\in C^{0}_{t}C^{\frac{3}{4}}_{x}\left(\left[0,1\right]\times\mathbb{R}^{2}\right)$
by Lemma \ref{lem:inverse Holder} since $\rho_{B}\in C^{0}_{t}C^{\frac{3}{4}}_{x}\left(\left[0,1\right]\times\mathbb{R}^{2}\right)$
by Proposition \ref{prop:bounded density} and $\overline{\rho}>\left|\left|\rho_{B}\right|\right|_{C^{0}_{t}C^{0}_{x}\left(\left[0,1\right]\times\mathbb{R}^{2}\right)}$
by hypothesis, $g_{1}\left(t\right)\varPhi\left(x\right)\in C^{0}_{t}C^{\frac{3}{4}}_{x}\left(\left[0,1\right]\times\mathbb{R}^{2}\right)$
as $g\in C\left(\left[0,1\right]\right)$ by point \ref{enu:last Proposition g}
and $\varPhi\in C^{\infty}\left(\mathbb{R}^{2}\right)$ (see Choice
\ref{choice: fu}), $a\left(t,x\right)\in C^{0}_{t}C^{\frac{3}{4}}_{x}\left(\left[0,1\right]\times\mathbb{R}^{2}\right)$
by hypothesis and $\left(\frac{\partial u_{B}}{\partial t}+u_{B}\cdot\nabla u_{B}\right)_{1}\in C^{0}_{t}C^{\frac{3}{4}}_{x}\left(\left[0,1\right]\times\mathbb{R}^{2}\right)$
by Proposition \ref{prop:bound material derivative u}. Furthermore,
by the definition of $g_{1}\left(t\right)$ given in \eqref{eq:g},
it is evident that this factor vanishes at $x=0$ $\forall t\in\left[0,1\right]$
since $\varPhi\left(0\right)=1$ (see Choice \ref{choice: fu}). Then,
we are in a position to apply Proposition \ref{prop:dyadic decomposition H=0000F6lder}
for each $t\in\left[0,1\right]$ with
\[
\beta\leftarrow\frac{3}{4},\quad\sigma\leftarrow\frac{1}{2},\quad R\leftarrow1,\quad K_{0}\leftarrow K_{0}\left(\delta,\varphi,\mu,Y\right),\quad g\leftarrow\frac{\partial\rho_{B}}{\partial x_{2}}\left(t,x\right),
\]
\[
f\leftarrow1+\frac{1}{\overline{\rho}+\rho_{B}\left(t,x\right)}\left(g_{1}\left(t\right)\varPhi\left(x\right)+a_{1}\left(t,x\right)-\left(\frac{\partial u_{B}}{\partial t}+u_{B}\cdot\nabla u_{B}\right)_{1}\left(t,x\right)\right),
\]
\[
K_{\beta-\sigma}\leftarrow K_{\frac{1}{4}}\left(\delta,\varphi,\mu,Y\right),\quad\eta\leftarrow2\left(\left(\frac{3}{4}-\frac{1}{2}\right)\left(1+k_{\max}\right)+\frac{1}{16}\right)<\frac{3}{4},
\]
where the last inequality is due to the fact that 
\[
2\left(\left(\frac{3}{4}-\frac{1}{2}\right)\left(1+\underbrace{k_{\max}}_{=\alpha_{*}<\frac{1}{4}}\right)+\frac{1}{16}\right)<2\left(\frac{1}{4}\frac{5}{4}+\frac{1}{16}\right)=\frac{3}{4},
\]
$\alpha_{*}$ being the regularity parameter defined in \cite{Articulo Boussinesq}.
Thereby, Proposition \ref{prop:dyadic decomposition H=0000F6lder}
ensures that the first summand of the left-hand-side of equation \eqref{eq:varpsi}
is $L^{\infty}_{t}C^{\frac{1}{4}}_{x}\left(\left[0,1\right]\times\mathbb{R}^{2}\right)\subseteq L^{\infty}_{t}C^{\alpha_{*}}_{x}\left(\left[0,1\right]\times\mathbb{R}^{2}\right)$,
since $\alpha_{*}<\frac{1}{4}$.
\item On the one hand, Proposition \ref{prop:bounded drhodx1} ensures that
$\frac{\partial\rho_{B}}{\partial x_{1}}\in C^{0}_{t}C^{0}_{x}\left(\left[0,1\right]\times\mathbb{R}^{2}\right)$.
In particular, this means that there is $M\left(\varphi,\mu,Y\right)>0$
such that
\[
\left|\left|\frac{\partial\rho_{B}}{\partial x_{1}}\left(t,\cdot\right)\right|\right|_{L^{\infty}\left(\mathbb{R}^{2}\right)}\leq M\leq\frac{M}{r^{\frac{1}{2}}}
\]
$\forall r\in\left(0,1\right]$ and $\forall t\in\left[0,1\right]$.
Besides, taking $\delta\leq\Upsilon_{\delta,3}$ and $\mu\leq\Upsilon_{\mu,3}$,
Proposition \ref{prop:bad bound drhodx1} (which we can apply because
$\frac{1}{4}>\frac{\alpha_{*}}{1+\alpha_{*}}$) implies that
\[
\left|\left|\frac{\partial\rho}{\partial x_{1}}\left(t,\cdot\right)\right|\right|_{\dot{C}^{\frac{1}{4}}\left(\mathbb{R}^{2}\setminus B\left(0;r\right)\right)}\lesssim_{\varphi,\mu,Y,\delta}\frac{1}{r^{2\left(\left(\frac{3}{4}-\frac{1}{2}\right)\left(1+k_{\max}\right)+\frac{1}{16}\right)}}
\]
$\forall r\in\left(0,1\right]$ and $\forall t\in\left[0,1\right]$.
On the other hand, the factor that multiplies $\frac{\partial\rho_{B}}{\partial x_{1}}$
is $C^{0}_{t}C^{\frac{3}{4}}_{x}\left(\left[0,1\right]\times\mathbb{R}^{2}\right)$
because $\frac{1}{\overline{\rho}+\rho_{B}}\in C^{0}_{t}C^{\frac{3}{4}}_{x}\left(\left[0,1\right]\times\mathbb{R}^{2}\right)$
by Lemma \ref{lem:inverse Holder} since $\rho_{B}\in C^{0}_{t}C^{\frac{3}{4}}_{x}\left(\left[0,1\right]\times\mathbb{R}^{2}\right)$
by Proposition \ref{prop:bounded density} and $\overline{\rho}>\left|\left|\rho_{B}\right|\right|_{C^{0}_{t}C^{0}_{x}\left(\left[0,1\right]\times\mathbb{R}^{2}\right)}$
by hypothesis, $g_{2}\left(t\right)\varPhi\left(x\right)\in C^{0}_{t}C^{\frac{3}{4}}_{x}\left(\left[0,1\right]\times\mathbb{R}^{2}\right)$
as $g\in C\left(\left[0,1\right]\right)$ by point \ref{enu:last Proposition g}
and $\varPhi\in C^{\infty}\left(\mathbb{R}^{2}\right)$ (see Choice
\ref{choice: fu}), $a\left(t,x\right)\in C^{0}_{t}C^{\frac{3}{4}}_{x}\left(\left[0,1\right]\times\mathbb{R}^{2}\right)$
by hypothesis and $\left(\frac{\partial u_{B}}{\partial t}+u_{B}\cdot\nabla u_{B}\right)_{2}\in C^{0}_{t}C^{\frac{3}{4}}_{x}\left(\left[0,1\right]\times\mathbb{R}^{2}\right)$
by Proposition \ref{prop:bound material derivative u}. Furthermore,
by the definition of $g_{2}\left(t\right)$ given in \eqref{eq:g}
and the symmetries discussed in point \ref{enu:last Proposition evenness and oddness},
it is evident that this factor vanishes at $x=0$ $\forall t\in\left[0,1\right]$
since $\varPhi\left(0\right)=1$ (see Choice \ref{choice: fu}). Then,
we are in a position to apply Proposition \ref{prop:dyadic decomposition H=0000F6lder}
for each $t\in\left[0,1\right]$ with
\[
\beta\leftarrow\frac{3}{4},\quad\sigma\leftarrow\frac{1}{2},\quad R\leftarrow1,\quad K_{0}\leftarrow M,\quad g\leftarrow\frac{\partial\rho_{B}}{\partial x_{1}}\left(t,x\right),
\]
\[
f\leftarrow\frac{1}{\overline{\rho}+\rho_{B}\left(t,x\right)}\left(g_{2}\left(t\right)\varPhi\left(x\right)+a_{2}\left(t,x\right)-\left(\frac{\partial u_{B}}{\partial t}+u_{B}\cdot\nabla u_{B}\right)_{2}\left(t,x\right)\right),
\]
\[
K_{\beta-\sigma}\leftarrow K_{\frac{1}{4}}\left(\delta,\varphi,\mu,Y\right),\quad\eta\leftarrow2\left(\left(\frac{3}{4}-\frac{1}{2}\right)\left(1+k_{\max}\right)+\frac{1}{16}\right)<\frac{3}{4}.
\]
where the last inequality is due to the same fact as in Point \ref{enu:LHS2 Calpha* drhodx2}.
Thereby, Proposition \ref{prop:dyadic decomposition H=0000F6lder}
ensures that the second summand of the left-hand-side of equation
\eqref{eq:varpsi} is $L^{\infty}_{t}C^{\frac{1}{4}}_{x}\left(\left[0,1\right]\times\mathbb{R}^{2}\right)\subseteq L^{\infty}_{t}C^{\alpha_{*}}_{x}\left(\left[0,1\right]\times\mathbb{R}^{2}\right)$,
since $\alpha_{*}<\frac{1}{4}$.
\end{enumerate}
\item \label{enu:T2 LinftyC1alpha*}$T^{\left(2\right)}\left(a\right)\in L^{\infty}_{t}C^{1,\alpha_{*}}_{x}\left(\left[0,1\right]\times\mathbb{R}^{2};\mathbb{R}^{2}\right)$.
By Corollary \ref{cor:estimates Laplacian}, we need $\Delta\varPsi^{\left(2\right)}\in L^{\infty}_{t}L^{p}_{x}\left(\left[0,1\right]\times\mathbb{R}^{2}\right)$
for some $p>2$, compactly supported uniformly in time and mean-free
in order to have $T^{\left(2\right)}\left(a\right)\in L^{\infty}_{t}C^{0}_{x}\left(\left[0,1\right]\times\mathbb{R}^{2};\mathbb{R}^{2}\right)$.
These conditions are met because of Point \ref{enu:last Proposition}.
Now, by Proposition \ref{prop:Laplace D2u Holder}, we require $\Delta\varPsi^{\left(2\right)}\in L^{\infty}_{t}\dot{C}^{\varepsilon}_{x}\left(\left[0,1\right]\times\mathbb{R}^{2}\right)\cap L^{\infty}_{t}L^{p}_{x}\left(\left[0,1\right]\times\mathbb{R}^{2}\right)$
for some $\varepsilon>0$ and some $p>2$ so as to have $T^{\left(2\right)}\left(a\right)\in L^{\infty}_{t}\dot{C}^{1}_{x}\left(\left[0,1\right]\times\mathbb{R}^{2};\mathbb{R}^{2}\right)$.
Again, the required conditions are satisfied because of Points \ref{enu:last Proposition} and \ref{enu:LHS2 Calpha*}.
Lastly, also by Proposition \ref{prop:Laplace D2u Holder}, we need
$\Delta\varPsi^{\left(2\right)}\in L^{\infty}_{t}\dot{C}^{\alpha_{*}}_{x}\left(\left[0,1\right]\times\mathbb{R}^{2}\right)$
in order for $T^{\left(2\right)}\left(a\right)$ to be $L^{\infty}_{t}\dot{C}^{1,\alpha_{*}}_{x}\left(\left[0,1\right]\times\mathbb{R}^{2};\mathbb{R}^{2}\right)$.
As we have seen in Point \ref{enu:LHS2 Calpha*}, this condition is
also satisfied. In conclusion, we may ensure that $T^{\left(2\right)}\left(a\right)\in L^{\infty}_{t}C^{1,\alpha_{*}}_{x}\left(\left[0,1\right]\times\mathbb{R}^{2};\mathbb{R}^{2}\right)$.
\item \label{enu:interpolation T2}$T^{\left(2\right)}\left(a\right)\in C^{0}_{t}C^{1,\alpha}_{x}\left(\left[0,1\right]\times\mathbb{R}^{2};\mathbb{R}^{2}\right)$.
Since $\alpha<\alpha_{*}$ and $T^{\left(2\right)}\left(a\right)\in C^{0}_{t}C^{\frac{3}{4}}_{x}\left(\left[0,1\right]\times\mathbb{R}^{2};\mathbb{R}^{2}\right)\cap L^{\infty}_{t}C^{1,\alpha_{*}}_{x}\left(\left[0,1\right]\times\mathbb{R}^{2};\mathbb{R}^{2}\right)$
because of Points \ref{enu:last Proposition} and \ref{enu:T2 LinftyC1alpha*},
we may apply Point 3 of Proposition \ref{prop:hard interpolation H=0000F6lder}
(to differences $T^{\left(2\right)}\left(a\right)\left(t,\cdot\right)-T^{\left(2\right)}\left(a\right)\left(s,\cdot\right)$
with $t,s\in\left[0,1\right]$) to obtain that $T^{\left(2\right)}\left(a\right)\in C^{0}_{t}C^{1,\alpha}_{x}\left(\left[0,1\right]\times\mathbb{R}^{2};\mathbb{R}^{2}\right)$.
\item Combining Points \ref{enu:T1 C1alpha} and \ref{enu:interpolation T2},
we conclude the truth of the result.
\end{enumerate}
\end{proof}

\subsection{Contractivity}

It turns out that, in order for $T$ to be contractive, not every
choice of $\varPhi$ is valid. In the following Lemma, we hand-craft
a $\varPhi$ with the required properties.
\begin{lem}
\label{lem:construction varPhi}$\forall\varepsilon\in\left(0,1\right)$,
there exists $\varPhi\in C^{\infty}_{c}\left(\mathbb{R}^{2}\right)$
such that
\begin{enumerate}
\item $\varPhi$ is even in $x_{2}$ ,
\item $\varPhi\left(0\right)=1$,
\item ~
\[
\left|\left|\nabla^{\perp}\Delta^{-1}\frac{\partial}{\partial x_{2}}\varPhi\right|\right|_{L^{\infty}\left(\mathbb{R}^{2}\right)}\leq\frac{1}{2}\left(1+\varepsilon\right).
\]
\end{enumerate}
\end{lem}
\begin{proof}
Consider
\begin{equation}
\psi\left(x\right)=A\varphi\left(\mu x_{1}\right)\varphi\left(\mu x_{2}\right)\cos\left(x_{1}\right)\sin\left(x_{2}\right),\label{eq:const varphi def psi}
\end{equation}
where $A\in\mathbb{R}$ and $\mu>0$ are parameters whose values will
be fixed later, $\varphi\in C^{\infty}_{c}\left(\mathbb{R}\right)$,
$\left.\varphi\right|_{\left[-1,1\right]}=1$, $\varphi$ is even
and $\left|\left|\varphi\right|\right|_{L^{\infty}\left(\mathbb{R}\right)}\leq1$.
A direct computation shows that
\[
\begin{aligned}\Delta\psi & =-2A\varphi\left(\mu x_{1}\right)\varphi\left(\mu x_{2}\right)\cos\left(x_{1}\right)\sin\left(x_{2}\right)+\\
 & \quad-2A\mu\varphi'\left(\mu x_{1}\right)\varphi\left(\mu x_{2}\right)\sin\left(x_{1}\right)\sin\left(x_{2}\right)+\\
 & \quad+2A\mu\varphi\left(\mu x_{1}\right)\varphi'\left(\mu x_{2}\right)\cos\left(x_{1}\right)\cos\left(x_{2}\right)+\\
 & \quad+A\mu^{2}\varphi''\left(\mu x_{1}\right)\varphi\left(\mu x_{2}\right)\cos\left(x_{1}\right)\sin\left(x_{2}\right)+\\
 & \quad+A\mu^{2}\varphi\left(\mu x_{1}\right)\varphi''\left(\mu x_{2}\right)\cos\left(x_{1}\right)\sin\left(x_{2}\right)=\\
 & =+2A\varphi\left(\mu x_{1}\right)\cos\left(x_{1}\right)\frac{\partial}{\partial x_{2}}\left[\varphi\left(\mu x_{2}\right)\cos\left(x_{2}\right)\right]+\\
 & \quad-2A\mu\varphi'\left(\mu x_{1}\right)\varphi\left(\mu x_{2}\right)\sin\left(x_{1}\right)\sin\left(x_{2}\right)+\\
 & \quad+A\mu^{2}\varphi''\left(\mu x_{1}\right)\varphi\left(\mu x_{2}\right)\cos\left(x_{1}\right)\sin\left(x_{2}\right)+\\
 & \quad+A\mu^{2}\varphi\left(\mu x_{1}\right)\varphi''\left(\mu x_{2}\right)\cos\left(x_{1}\right)\sin\left(x_{2}\right)=\\
 & \eqqcolon+2A\varphi\left(\mu x_{1}\right)\cos\left(x_{1}\right)\frac{\partial}{\partial x_{2}}\left[\varphi\left(\mu x_{2}\right)\cos\left(x_{2}\right)\right]+E\left(x\right).
\end{aligned}
\]
Let us define $\varPhi$ by
\begin{equation}
\varPhi\left(x\right)\coloneqq\int^{x_{2}}_{-\infty}\Delta\psi\left(x_{1},s\right)\mathrm{d}s=2A\varphi\left(\mu x_{1}\right)\varphi\left(\mu x_{2}\right)\cos\left(x_{1}\right)\cos\left(x_{2}\right)+\int^{x_{2}}_{-\infty}E\left(x_{1},s\right)\mathrm{d}s.\label{eq:const varphi def varphi}
\end{equation}
By looking at the definition of $E$, we can see that it is odd in
$x_{2}$ and compactly supported. As the leading order term of $\varPhi$
is compactly supported as well, we deduce that $\varPhi$ is compactly
supported, too. Furthermore, it is obvious that $\varPhi\in C^{\infty}\left(\mathbb{R}^{2}\right)$.
Moreover, since $E$ is odd in $x_{2}$, $\int^{x_{2}}_{-\infty}E\left(x_{1},s\right)\mathrm{d}s$
must be even in $x_{2}$, as is the leading order term of $\varPhi$,
so $\varPhi$ is even in $x_{2}$. 

As for the value in the origin, we have
\begin{equation}
\varPhi\left(0\right)=2A+\int^{0}_{-\infty}E\left(0,s\right)\mathrm{d}s=2A+\int^{0}_{-\infty}A\mu^{2}\varphi''\left(\mu s\right)\sin\left(s\right)\mathrm{d}s.\label{eq:const varphi value at origin}
\end{equation}
The change of variables
\[
\mu s=v\iff s=\frac{v}{\mu}
\]
allows us to bound
\begin{equation}
\left|\int^{0}_{-\infty}A\mu^{2}\varphi''\left(\mu s\right)\sin\left(s\right)\mathrm{d}s\right|\leq\left|A\right|\left|\mu^{2}\int^{0}_{-\infty}\varphi''\left(v\right)\sin\left(\frac{v}{\mu}\right)\frac{1}{\mu}\mathrm{d}v\right|\leq\left|A\right|\mu\int^{0}_{-\infty}\left|\varphi''\left(v\right)\right|\mathrm{d}v.\label{eq:const varphi bound E}
\end{equation}

Moreover,
\[
\nabla^{\perp}\Delta^{-1}\frac{\partial}{\partial x_{2}}\varPhi=\nabla^{\perp}\psi=A\varphi\left(\mu x_{1}\right)\varphi\left(\mu x_{2}\right)\left(\begin{matrix}-\cos\left(x_{1}\right)\cos\left(x_{2}\right)\\
-\sin\left(x_{1}\right)\sin\left(x_{2}\right)
\end{matrix}\right)+A\mu\cos\left(x_{1}\right)\sin\left(x_{2}\right)\left(\begin{matrix}-\varphi\left(\mu x_{1}\right)\varphi'\left(\mu x_{2}\right)\\
\varphi'\left(\mu x_{1}\right)\varphi\left(\mu x_{2}\right)
\end{matrix}\right).
\]
Clearly, as $\left|\left|\varphi\right|\right|_{L^{\infty}\left(\mathbb{R}\right)}\leq1$,
\begin{equation}
\left|\left|\nabla^{\perp}\Delta^{-1}\frac{\partial}{\partial x_{2}}\varPhi\right|\right|_{L^{\infty}\left(\mathbb{R}^{2}\right)}\leq\left|A\right|\left(1+\mu\left|\left|\varphi\right|\right|_{\dot{C}^{1}\left(\mathbb{R}\right)}\right).\label{eq:const varphi bound Lip}
\end{equation}

First fixing $\mu$ small enough so that
\[
\mu\leq\frac{\varepsilon}{4}\frac{1}{\max\left\{ \int^{0}_{-\infty}\left|\varphi''\left(v\right)\right|\mathrm{d}v,\left|\left|\varphi\right|\right|_{\dot{C}^{1}\left(\mathbb{R}\right)}\right\} },
\]
and then taking
\[
1=2A+A\int^{0}_{-\infty}\mu^{2}\varphi''\left(\mu s\right)\sin\left(s\right)\mathrm{d}s\iff A=\frac{1}{2+\int^{0}_{-\infty}\mu^{2}\varphi''\left(\mu s\right)\sin\left(s\right)\mathrm{d}s},
\]
in view of equations \eqref{eq:const varphi value at origin}, \eqref{eq:const varphi bound E}, and \eqref{eq:const varphi bound Lip},
ensures that 
\[
\varPhi\left(0\right)=1,\quad A\leq\frac{1}{2-\frac{\varepsilon}{4}},\quad\left|\left|\nabla^{\perp}\Delta^{-1}\frac{\partial}{\partial x_{2}}\varPhi\right|\right|_{L^{\infty}\left(\mathbb{R}^{2}\right)}\leq\left|A\right|\left(1+\frac{\varepsilon}{4}\right).
\]
 Since $\varepsilon<1$ and $\frac{1}{2-z}\leq\frac{1}{2}\left(1+z\right)$
$\forall z\in\left[0,1\right]$, we deduce that
\[
A\leq\frac{1}{2}\left(1+\frac{\varepsilon}{4}\right),\quad\left|\left|\nabla^{\perp}\Delta^{-1}\frac{\partial}{\partial x_{2}}\varPhi\right|\right|_{L^{\infty}\left(\mathbb{R}^{2}\right)}\leq\frac{1}{2}\left(1+\frac{\varepsilon}{4}\right)^{2}\leq\frac{1}{2}\left(1+\frac{1}{16}\varepsilon^{2}+\frac{1}{2}\varepsilon\right)<\frac{1}{2}\left(1+\varepsilon\right)
\]
as $\varepsilon<1$.
\end{proof}

\begin{choice}
\label{choice:varPhi}We choose $\varPhi$ such that $\varPhi\in C^{\infty}_{c}\left(\mathbb{R}^{2}\right)$,
$\varPhi$ is even in $x_{2}$, $\varPhi\left(0\right)=1$ and $\left|\left|\nabla^{\perp}\Delta^{-1}\frac{\partial}{\partial x_{2}}\varPhi\right|\right|_{L^{\infty}\left(\mathbb{R}^{2}\right)}\leq\frac{3}{4}$.
(Notice that such a $\varPhi$ exists thanks to Lemma \ref{lem:construction varPhi}
with $\varepsilon=\frac{1}{2}$ and that it satisfies the conditions
of Choice \ref{choice: fu}).
\end{choice}
\begin{prop}
\label{prop:MAP contractive}Provided that $\overline{\rho}$ is large
enough, i.e., 
\[
\overline{\rho}\geq\max\left\{ 2\left|\left|\rho_{B}\right|\right|_{C^{0}_{t}C^{0}_{x}\left(\left[0,1\right]\times\mathbb{R}^{2}\right)},\Upsilon\left(\mathrm{\text{sol of Boussinesq}},\varPhi\right)\right\} ,
\]
the map
\[
\begin{matrix}T: &  & \left\{ b\in C^{0}_{t}C^{0}_{x}\left(\left[0,1\right]\times\mathbb{R}^{2};\mathbb{R}^{2}\right):b_{1}\text{ even in }x_{2},\;b_{2}\text{ odd in }x_{2}\right\}  & \longrightarrow\\
 & \longrightarrow & \left\{ b\in C^{0}_{t}C^{0}_{x}\left(\left[0,1\right]\times\mathbb{R}^{2};\mathbb{R}^{2}\right):b_{1}\text{ even in }x_{2},\;b_{2}\text{ odd in }x_{2}\right\} 
\end{matrix}
\]
is contractive. In fact,
\[
\left|\left|T\left(a\right)-T\left(b\right)\right|\right|_{C^{0}_{t}C^{0}_{x}\left(\left[0,1\right]\times\mathbb{R}^{2}\right)}\le\frac{9}{10}\left|\left|a-b\right|\right|_{C^{0}_{t}C^{0}_{x}\left(\left[0,1\right]\times\mathbb{R}^{2}\right)}.
\]
\end{prop}
\begin{proof}
We will denote $T\left(a\right)=\nabla^{\perp}\varPsi_{a}$ and $T\left(b\right)=\nabla^{\perp}\varPsi_{b}$.
Let us find the equation that $\varPsi_{a}-\varPsi_{b}$ satisfies.
On the one hand, looking at equation \eqref{eq:varpsi}, we deduce
that
\[
\begin{aligned}\Delta\left(\varPsi_{a}-\varPsi_{b}\right)\left(t,x\right) & =-\left(g_{a}\left(t\right)-g_{b}\left(t\right)\right)\cdot\nabla^{\perp}\varPhi\left(x\right)+\\
 & \quad+\frac{1}{\overline{\rho}+\rho_{B}}\frac{\partial\rho_{B}}{\partial x_{2}}\left[\left(g_{a,1}\left(t\right)-g_{b,1}\left(t\right)\right)\varPhi\left(x\right)+a_{1}\left(t,x\right)-b_{1}\left(t,x\right)\right]+\\
 & \quad-\frac{1}{\overline{\rho}+\rho_{B}}\frac{\partial\rho_{B}}{\partial x_{1}}\left[\left(g_{a,2}\left(t\right)-g_{b,2}\left(t\right)\right)\varPhi\left(x\right)+a_{2}\left(t,x\right)-b_{2}\left(t,x\right)\right].
\end{aligned}
\]
On the other hand, looking at equation \eqref{eq:g}, we see that
\[
g_{a}\left(t\right)-g_{b}\left(t\right)=\left[-a\left(t,0\right)+b\left(t,0\right)\right].
\]
Consequently,
\begin{equation}
\begin{aligned}\Delta\left(\varPsi_{a}-\varPsi_{b}\right)\left(t,x\right) & =+\left(a\left(t,0\right)-b\left(t,0\right)\right)\cdot\nabla^{\perp}\varPhi\left(x\right)+\\
 & \quad+\frac{1}{\overline{\rho}+\rho_{B}}\frac{\partial\rho_{B}}{\partial x_{2}}\left[a_{1}\left(t,x\right)-b_{1}\left(t,x\right)-\left(a_{1}\left(t,0\right)-b_{1}\left(t,0\right)\right)\varPhi\left(x\right)\right]+\\
 & \quad-\frac{1}{\overline{\rho}+\rho_{B}}\frac{\partial\rho_{B}}{\partial x_{1}}\left[a_{2}\left(t,x\right)-b_{2}\left(t,x\right)-\left(a_{2}\left(t,0\right)-b_{2}\left(t,0\right)\right)\varPhi\left(x\right)\right].
\end{aligned}
\label{eq:varpsi diference}
\end{equation}
To continue, we use the linearity of the equation to split it into
two. Let $\varPsi^{\left(1\right)}_{a}-\varPsi^{\left(1\right)}_{b}$
and $\varPsi^{\left(2\right)}_{a}-\varPsi^{\left(2\right)}_{b}$ solve
\[
\begin{aligned}\Delta\left(\varPsi^{\left(1\right)}_{a}-\varPsi^{\left(1\right)}_{b}\right)\left(t,x\right) & =\left(a\left(t,0\right)-b\left(t,0\right)\right)\cdot\nabla^{\perp}\varPhi\left(x\right),\\
\Delta\left(\varPsi^{\left(2\right)}_{a}-\varPsi^{\left(2\right)}_{b}\right)\left(t,x\right) & =+\frac{1}{\overline{\rho}+\rho_{B}}\frac{\partial\rho_{B}}{\partial x_{2}}\left[a_{1}\left(t,x\right)-b_{1}\left(t,x\right)-\left(a_{1}\left(t,0\right)-b_{1}\left(t,0\right)\right)\varPhi\left(x\right)\right]+\\
 & \quad-\frac{1}{\overline{\rho}+\rho_{B}}\frac{\partial\rho_{B}}{\partial x_{1}}\left[a_{2}\left(t,x\right)-b_{2}\left(t,x\right)-\left(a_{2}\left(t,0\right)-b_{2}\left(t,0\right)\right)\varPhi\left(x\right)\right].
\end{aligned}
\]
We will also denote $T^{\left(1\right)}\left(a\right)=\nabla^{\perp}\varPsi^{\left(1\right)}_{a}$,
$T^{\left(2\right)}\left(a\right)=\nabla^{\perp}\varPsi^{\left(2\right)}_{a}$,
$T^{\left(1\right)}\left(b\right)=\nabla^{\perp}\varPsi^{\left(1\right)}_{b}$
and $T^{\left(2\right)}\left(b\right)=\nabla^{\perp}\varPsi^{\left(2\right)}_{b}$.
Notice that, under these definitions,
\[
T\left(a\right)-T\left(b\right)=T^{\left(1\right)}\left(a\right)-T^{\left(1\right)}\left(b\right)+T^{\left(2\right)}\left(a\right)-T^{\left(2\right)}\left(b\right).
\]

On the one hand, since $a_{2}\left(t,x\right)$ and $b_{2}\left(t,x\right)$
are odd in $x_{2}$ by hypothesis, we deduce that $a_{2}\left(t,0\right)=0=b_{2}\left(t,0\right)$.
Thus,
\[
\Delta\left(\varPsi^{\left(1\right)}_{a}-\varPsi^{\left(1\right)}_{b}\right)\left(t,x\right)=+\left(a\left(t,0\right)-b\left(t,0\right)\right)\cdot\nabla^{\perp}\varPhi\left(x\right)=-\left(a_{1}\left(t,0\right)-b_{1}\left(t,0\right)\right)\frac{\partial\varPhi}{\partial x_{2}}\left(x\right).
\]
By Choice \ref{choice:varPhi}, we infer that
\[
\left|\left|\left(T^{\left(1\right)}\left(a\right)-T^{\left(1\right)}\left(b\right)\right)\left(t,\cdot\right)\right|\right|_{C^{0}\left(\mathbb{R}^{2}\right)}=\left|\left|\nabla^{\perp}\Delta^{-1}\frac{\partial}{\partial x_{2}}\left[-\left(a_{1}\left(t,0\right)-b_{1}\left(t,0\right)\right)\varPhi\right]\right|\right|_{C^{0}\left(\mathbb{R}^{2}\right)}\leq\frac{3}{4}\left|a_{1}\left(t,0\right)-b_{1}\left(t,0\right)\right|.
\]
Taking suprema in $t\in\left[0,1\right]$, we easily arrive at
\[
\left|\left|T^{\left(1\right)}\left(a\right)-T^{\left(1\right)}\left(b\right)\right|\right|_{C^{0}_{t}C^{0}_{x}\left(\left[0,1\right]\times\mathbb{R}^{2}\right)}\leq\frac{3}{4}\left|\left|a-b\right|\right|_{C^{0}_{t}C^{0}_{x}\left(\left[0,1\right]\times\mathbb{R}^{2}\right)}.
\]

On the other hand, in view of Corollary \ref{cor:estimates Laplacian},
if we wish to have $T^{\left(2\right)}\left(a\right)-T^{\left(2\right)}\left(b\right)\in C^{0}_{t}C^{0}_{x}\left(\left[0,1\right]\times\mathbb{R}^{2};\mathbb{R}^{2}\right)$,
we need $\Delta\left(\varPsi^{\left(2\right)}_{a}-\varPsi^{\left(2\right)}_{b}\right)\in C^{0}_{t}L^{2+\varepsilon}_{x}\left(\left[0,1\right]\times\mathbb{R}^{2}\right)$
for some $\varepsilon>0$, compactly supported uniformly in time and
mean-free. To see these last two properties, we may argue exactly
as we have done in the proof of Proposition \ref{prop:MAP wins regularity 1},
so we will skip the details. Now, concerning the bounds in $L^{8}_{x}$,
clearly,
\[
\begin{aligned}\left|\left|\Delta\left(\varPsi^{\left(2\right)}_{a}-\varPsi^{\left(2\right)}_{b}\right)\right|\right|_{C^{0}_{t}L^{8}_{x}\left(\left[0,1\right]\times\mathbb{R}^{2}\right)} & \leq+\frac{2}{\overline{\rho}}\left|\left|a-b\right|\right|_{C^{0}_{t}C^{0}_{x}\left(\left[0,1\right]\times\mathbb{R}^{2}\right)}\left(1+\left|\left|\varPhi\right|\right|_{L^{\infty}\left(\mathbb{R}^{2}\right)}\right)\left|\left|\frac{\partial\rho_{B}}{\partial x_{2}}\right|\right|_{C^{0}_{t}L^{8}_{x}\left(\left[0,1\right]\times\mathbb{R}^{2}\right)}+\\
 & \quad+\frac{2}{\overline{\rho}}\left|\left|a-b\right|\right|_{C^{0}_{t}C^{0}_{x}\left(\left[0,1\right]\times\mathbb{R}^{2}\right)}\left(1+\left|\left|\varPhi\right|\right|_{L^{\infty}\left(\mathbb{R}^{2}\right)}\right)\left|\left|\frac{\partial\rho_{B}}{\partial x_{1}}\right|\right|_{C^{0}_{t}L^{8}_{x}\left(\left[0,1\right]\times\mathbb{R}^{2}\right)},
\end{aligned}
\]
where we have used that $\overline{\rho}\geq2\left|\left|\rho_{B}\right|\right|_{C^{0}_{t}C^{0}_{x}\left(\left[0,1\right]\times\mathbb{R}^{2}\right)}$
by hypothesis. Now, point 3 of Corollary \ref{cor:estimates Laplacian}
and the linearity of the Laplacian ensure that
\[
\begin{aligned}\left|\left|T^{\left(2\right)}\left(a\right)-T^{\left(2\right)}\left(b\right)\right|\right|_{C^{0}_{t}C^{\frac{3}{4}}_{x}\left(\mathbb{R}^{2}\right)} & \lesssim\frac{1}{\overline{\rho}}\left(1+\mathrm{diam}\left(\bigcup_{t\in\left[0,1\right]}\mathrm{supp}\left(\rho_{B}\left(t,\cdot\right)\right)\right)\right)\left(1+\left|\left|\varPhi\right|\right|_{L^{\infty}\left(\mathbb{R}^{2}\right)}\right)\cdot\\
 & \quad\cdot\left(\left|\left|\frac{\partial\rho_{B}}{\partial x_{1}}\right|\right|_{C^{0}_{t}L^{8}_{x}\left(\left[0,1\right]\times\mathbb{R}^{2}\right)}+\left|\left|\frac{\partial\rho_{B}}{\partial x_{2}}\right|\right|_{C^{0}_{t}L^{8}_{x}\left(\left[0,1\right]\times\mathbb{R}^{2}\right)}\right)\left|\left|a-b\right|\right|_{C^{0}_{t}C^{0}_{x}\left(\left[0,1\right]\times\mathbb{R}^{2}\right)}.
\end{aligned}
\]
Thereby, taking $\overline{\rho}$ large enough (depending on $\left|\left|\varPhi\right|\right|_{L^{\infty}\left(\mathbb{R}^{2}\right)}$
and the support and $L^{8}$ norms of $\frac{\partial\rho_{B}}{\partial x_{1}}$
and $\frac{\partial\rho_{B}}{\partial x_{2}}$), we can make the factor
that precedes $\left|\left|a-b\right|\right|_{C^{0}_{t}C^{0}_{x}\left(\left[0,1\right]\times\mathbb{R}^{2}\right)}$
as small as we wish. As $T^{\left(1\right)}$ is contractive and we
can make $T^{\left(2\right)}$ as contractive as needed by fixing
the value of $\overline{\rho}$, we arrive at the result.
\end{proof}

\section{Proof of the main Theorem}

The proof of the Main Theorem is split in two parts: Theorem \ref{thm:HALF-MAIN}
and Corollary \ref{cor:SCREENED}.

\subsection{\label{subsec:Finite-time-singularity}Finite-time singularity}
\begin{thm}
\label{thm:HALF-MAIN}Let $\alpha\in\left(0,\alpha_{*}\right)$, where
$\alpha_{*}=\sqrt{\frac{4}{3}}-1$. Let $\overline{\rho}\in\mathbb{R}_{+}$
be sufficiently large. Then, there is a solution $\left(\rho,u,f_{\rho},f_{u}\right)$
in $\left[0,1\right)\times\mathbb{R}^{2}$ of the forced incompressible
non-homogeneous Euler equations \eqref{eq:inhomogeneous 2D Euler}
that satisfies:
\begin{enumerate}
\item There is a finite time singularity at time $t=1$, i.e.,
\[
\lim_{T\to1^{-}}\int^{T}_{0}\left|\left|\omega\left(s,\cdot\right)\right|\right|_{L^{\infty}\left(\mathbb{R}^{2}\right)}\mathrm{d}s=\infty.
\]
(See Theorem 2 of \cite{Danchin-Fanelli}).
\item ~
\[
f_{\rho}\in C^{0}_{t}C^{1,\alpha}_{x,c}\left(\left[0,1\right]\times\mathbb{R}^{2}\right),\quad f_{u}\in C^{0}_{t}C^{1,\alpha}_{x}\left(\left[0,1\right]\times\mathbb{R}^{2};\mathbb{R}^{2}\right)\cap C^{0}_{t}L^{2}_{x}\left(\left[0,1\right]\times\mathbb{R}^{2};\mathbb{R}^{2}\right).
\]
\item The solution has the following structure:
\[
\rho=\overline{\rho}+\rho_{B},\quad u=u_{B},
\]
where $\left(\rho_{B},u_{B}\right)$ is one of the solutions of the
forced Boussinesq system constructed in \cite{Articulo Boussinesq},
which also develops a singularity at time $t=1$\@.
\item $\forall\varepsilon>0$, $\rho,u\in C^{\infty}_{t}C^{\infty}_{x}\left(\left[0,1-\varepsilon\right]\times\mathbb{R}^{2}\right)$
and $\nabla\rho$ and $u$ are compactly supported uniformly in time.
\end{enumerate}
\end{thm}
\begin{proof}
~
\begin{enumerate}
\item Constructing a Boussinesq solution $\left(\rho_{B},u_{B},f_{\rho_{B}},f_{\omega_{B}}\right)$.
The Boussinesq solution must be built as explained in \cite{Articulo Boussinesq}
for the value of $\alpha$ of the statement with the additional requirements
on the parameters $\delta$, $\mu$ and $\gamma$ specified by Propositions
\ref{prop:MAP wins regularity 1}, \ref{prop:MAP wins regularity 2}, \ref{prop:MAP contractive}, and \ref{prop:blow-up criterion}.
These pose no issue because every condition only involves a single
parameter and the Boussinesq construction allows to choose $\delta$
and $\mu$ as small as desired and $\gamma$ as close to one as wished
for.
\item Take
\[
\begin{aligned}f_{u}\left(t,x\right) & =g\left(t\right)\varPhi\left(x\right)+\nabla^{\perp}\varPsi\left(t,x\right),\\
\rho\left(t,x\right) & =\overline{\rho}+\rho_{B}\left(t,x\right),\\
u\left(t,x\right) & =u_{B}\left(t,x\right)
\end{aligned}
\]
as dictated by Choices \ref{choice: fu} and \ref{choice:varPhi} in
equations \eqref{eq:inhomogeneous 2D Euler} and \eqref{eq:vorticity equation}.
As remarked in section \ref{sec:ideas of the proof}, this leads to
$f_{\rho}\left(t,x\right)=f_{\rho_{B}}\left(t,x\right)$ and equation
\eqref{eq:vorticity equation final first attempt}. Moreover, $\rho,u\in C^{\infty}_{t}C^{\infty}_{x}\left(\left[0,1-\varepsilon\right]\times\mathbb{R}^{2}\right)$
$\forall\varepsilon>0$ because $\rho_{B},u_{B}\in C^{\infty}_{t}C^{\infty}_{x}\left(\left[0,1-\varepsilon\right]\times\mathbb{R}^{2}\right)$
$\forall\varepsilon>0$ according to the Main Theorem of \cite{Articulo Boussinesq}
(see Theorem \ref{thm:MAIN THM Boussinesq}) and $\nabla\rho$ and
$u$ are compactly supported uniformly in time because $\rho_{B}$
and $u_{B}$ are (see subsection \ref{subsec:summary Boussinesq}
or Proposition 24 of \cite{Articulo Boussinesq}).
\item Consider the map $T$ defined in definition \ref{def:THE MAP}. Notice
that fixed-points of this map are solutions of equation \eqref{eq:vorticity equation final first attempt}.
We take $\overline{\rho}$ as large as Proposition \ref{prop:MAP contractive}
demands it. Then,
\[
\begin{matrix}T: &  & \left\{ b\in C^{0}_{t}C^{0}_{x}\left(\left[0,1\right]\times\mathbb{R}^{2};\mathbb{R}^{2}\right):b_{1}\text{ even in }x_{2},\;b_{2}\text{ odd in }x_{2}\right\}  & \longrightarrow\\
 & \longrightarrow & \left\{ b\in C^{0}_{t}C^{0}_{x}\left(\left[0,1\right]\times\mathbb{R}^{2};\mathbb{R}^{2}\right):b_{1}\text{ even in }x_{2},\;b_{2}\text{ odd in }x_{2}\right\} 
\end{matrix}
\]
is contractive. Since the domain (and image) of $T$ is a closed subspace
of a Banach space, it is a complete metric space. Thus, Banach's Fixed-Point
Theorem guarantees the existence of a (unique) fixed point $T\left(a\right)=a\in C^{0}_{t}C^{0}_{x}\left(\left[0,1\right]\times\mathbb{R}^{2};\mathbb{R}^{2}\right)$.
Then, Propositions \ref{prop:MAP wins regularity 1} and \ref{prop:MAP wins regularity 2}
allow us to win regularity. Indeed, Proposition \ref{prop:MAP wins regularity 1}
ensures that $a=T\left(a\right)\in C^{0}_{t}C^{\frac{3}{4}}_{x}\left(\left[0,1\right]\times\mathbb{R}^{2};\mathbb{R}^{2}\right)$
and $g\left(t\right)\in C\left(\left[0,1\right]\right)$ and then
Proposition \ref{prop:MAP wins regularity 2} gives us $a=T\left(a\right)\in C^{0}_{t}C^{1,\alpha}_{x}\left(\left[0,1\right]\times\mathbb{R}^{2};\mathbb{R}^{2}\right)\cap C^{0}_{t}L^{2}_{x}\left(\left[0,1\right]\times\mathbb{R}^{2};\mathbb{R}^{2}\right)$.
\item Thereby,
\[
\rho=\overline{\rho}+\rho_{B},\quad u=u_{B},\quad f_{\rho}=f_{\rho_{B}},\quad f_{u}=g_{a}\left(t\right)\varPhi\left(x\right)+a,
\]
where $g_{a}\left(t\right)\in C\left(\left[0,1\right]\right)$ can
be computed from $a$ by equation \eqref{eq:g}, is a solution of
system \eqref{eq:inhomogeneous 2D Euler}. Furthermore, by the Main
Theorem of \cite{Articulo Boussinesq} (see Theorem \ref{thm:MAIN THM Boussinesq})
and Choices \ref{choice: fu} and \ref{choice:varPhi}, we can ensure
that
\[
f_{\rho}\in C^{0}_{t}C^{1,\alpha}_{x,c}\left(\left[0,1\right]\times\mathbb{R}^{2}\right),\quad f_{u}\in C^{0}_{t}C^{1,\alpha}_{x}\left(\left[0,1\right]\times\mathbb{R}^{2};\mathbb{R}^{2}\right)\cap C^{0}_{t}L^{2}_{x}\left(\left[0,1\right]\times\mathbb{R}^{2};\mathbb{R}^{2}\right).
\]
\item There is a finite-time singularity at $t=1$. This is a direct consequence
of Proposition \ref{prop:blow-up criterion} since $u=u_{B}$.
\end{enumerate}
\end{proof}

\subsection{Screened singularity}
\begin{prop}
\label{prop:necessary condition unscreened}Assume $\left(\rho,u,f_{\rho},f_{u}\right)$,
the solution constructed in Theorem \ref{thm:HALF-MAIN}, is an unscreened
singularity. Then,
\[
\begin{aligned}-\frac{\partial\psi}{\partial x_{2}}\left(1,0\right) & =\left(\frac{\partial u_{B}}{\partial t}+u_{B}\cdot\nabla u_{B}\right)_{1}\left(1,0\right)-\left(\overline{\rho}+\rho_{B}\left(1,0\right)\right),\\
\Delta\psi\left(1,\cdot\right) & =f_{\omega_{B}}\left(1,\cdot\right),
\end{aligned}
\]
where $f_{u}=\frac{\nabla\phi}{\rho}+\nabla^{\perp}\psi$ is the canonical
decomposition given in equation \eqref{eq:inhomogeneous canonical decomposition fu}.
\end{prop}
\begin{proof}
We start from equation \eqref{eq:vorticity equation}, substituting
$f_{u}=\frac{\nabla\phi}{\rho}+\nabla^{\perp}\psi$, the canonical
decomposition given in equation \eqref{eq:inhomogeneous canonical decomposition fu}.
As noted in subsection \ref{subsec:Canonical-force-decomposition},
$\frac{\nabla\phi}{\rho}$ disappears from the equation, leading us
to
\[
\frac{\partial\omega}{\partial t}+u\cdot\nabla\omega=\left[\nabla^{\perp}\psi-\frac{\partial u}{\partial t}-u\cdot\nabla u\right]\cdot\frac{\nabla^{\perp}\rho}{\rho}+\Delta\psi.
\]
Substituting $u=u_{B}$ and $\rho=\overline{\rho}+\rho_{B}$ (see
Choice \ref{choice: fu}) provides
\[
\frac{\partial\omega_{B}}{\partial t}+u_{B}\cdot\nabla\omega_{B}=\left[\nabla^{\perp}\psi-\frac{\partial u_{B}}{\partial t}-u_{B}\cdot\nabla u_{B}\right]\cdot\frac{\nabla^{\perp}\rho_{B}}{\overline{\rho}+\rho_{B}}+\Delta\psi.
\]
Using the Boussinesq equation (see equation \eqref{eq:forced Boussinesq})
we arrive at
\begin{equation}
\begin{aligned}\Delta\psi= & +\frac{\partial\rho_{B}}{\partial x_{2}}\left(1+\frac{1}{\overline{\rho}+\rho_{B}}\left(-\frac{\partial\psi}{\partial x_{2}}-\left(\frac{\partial u_{B}}{\partial t}+u_{B}\cdot\nabla u_{B}\right)_{1}\right)\right)+\\
 & -\frac{1}{\overline{\rho}+\rho_{B}}\frac{\partial\rho_{B}}{\partial x_{1}}\left(\frac{\partial\psi}{\partial x_{1}}-\left(\frac{\partial u_{B}}{\partial t}+u_{B}\cdot\nabla u_{B}\right)_{2}\right)+f_{\omega_{B}}.
\end{aligned}
\label{eq:vorticity equation dynamic part of force}
\end{equation}
Since we are assuming that $\left(\rho,u,f_{\rho},f_{u}\right)$ is
an unscreened singularity, $\nabla^{\perp}\psi$ must have the same
regularity as $f_{u}$. By Theorem \ref{thm:HALF-MAIN}, we must have
$\nabla^{\perp}\psi\in C^{0}_{t}C^{1}_{x}$, i.e., $\Delta\psi\in C^{0}_{t}C^{0}_{x}$.
By the Main Theorem of \cite{Articulo Boussinesq} (see Theorem \ref{thm:MAIN THM Boussinesq}),
we know that $f_{\omega_{B}}\in C^{0}_{t}C^{0}_{x}$. The second summand
is also $C^{0}_{t}C^{0}_{x}$ because of Propositions \ref{prop:bounded density}, \ref{prop:bounded drhodx1}, and \ref{prop:bound material derivative u},
Lemma \ref{lem:inverse Holder} and because $\overline{\rho}$ has
been taken large enough so that $\overline{\rho}>\left|\left|\rho_{B}\right|\right|_{C^{0}_{t}C^{0}_{x}\left(\left[0,1\right]\times\mathbb{R}^{2}\right)}$
(otherwise, the requirements of Propositions \ref{prop:MAP wins regularity 1}, \ref{prop:MAP wins regularity 2}, and \ref{prop:MAP contractive}
could not be met). In conclusion, the first summand must also be $C^{0}_{t}C^{0}_{x}$.
Since $\int^{1}_{0}\left|\left|\frac{\partial\rho_{B}}{\partial x_{2}}\left(t,\cdot\right)\right|\right|_{L^{\infty}\left(\mathbb{R}^{2}\right)}\mathrm{d}t=\infty$,
$\frac{\partial\rho_{B}}{\partial x_{2}}\in C^{\infty}_{t}C^{\infty}_{x}\left(\left[0,1-r\right]\times\mathbb{R}^{2}\right)$
$\forall r>0$, $\frac{\partial\rho_{B}}{\partial x_{2}}\in L^{\infty}_{t}L^{\infty}_{x}\left(\left[0,1\right]\times\left(\mathbb{R}^{2}\setminus B\left(0;r\right)\right)\right)$
$\forall r>0$ (see Main Theorem of \cite{Articulo Boussinesq} {[}i.e.,
Theorem \ref{thm:MAIN THM Boussinesq}{]} and Propositions \ref{prop:bounded drhodx1} and \ref{prop:bad bound drhodx2}),
Proposition \ref{prop:dyadic decomposition f=00003D0 necessary} tells
us that, necessarily, 
\[
1+\frac{1}{\overline{\rho}+\rho_{B}\left(1,0\right)}\left(-\frac{\partial\psi}{\partial x_{2}}\left(1,0\right)-\left(\frac{\partial u_{B}}{\partial t}+u_{B}\cdot\nabla u_{B}\right)_{1}\left(1,0\right)\right)=0.
\]
This is equivalent to the first equality of the statement.

For the second claim, evaluate \eqref{eq:vorticity equation dynamic part of force}
at $t=1$, bearing in mind the cancellations that occur in the factors
that multiply $\frac{\partial\rho_{B}}{\partial x_{2}}$ (because
of the first equation of the statement) and $\frac{\partial\rho_{B}}{\partial x_{1}}$
(because of the oddness of $\left(\frac{\partial u_{B}}{\partial t}+u_{B}\cdot\nabla u_{B}\right)_{2}$
and of $\frac{\partial\psi}{\partial x_{1}}$ with respect to $x_{2}$).
\end{proof}

\begin{cor}
\label{cor:SCREENED}By taking a higher value of $\overline{\rho}$
if necessary in the proof of Theorem \ref{thm:HALF-MAIN}, we may
ensure that the solution constructed in Theorem \ref{thm:HALF-MAIN}
is a screened singularity.
\end{cor}
\begin{proof}
We will prove the result by showing that an unscreened singularity
is only possible for, at most, a bounded set of values of $\overline{\rho}$.
Let $\left(\rho_{;\overline{\rho}},u_{;\overline{\rho}},f_{\rho;\overline{\rho}},f_{u;\overline{\rho}}\right)$
denote the solution constructed with the procedure of Theorem \ref{thm:HALF-MAIN}
for the specified value of $\overline{\rho}$. Because of the nature
of the construction of Theorem \ref{thm:HALF-MAIN}, if it works for
some $\overline{\rho}=\overline{\rho}_{0}\in\left(0,\infty\right)$,
it also works $\forall\overline{\rho}\in\left[\overline{\rho}_{0},\infty\right)$,
i.e., the set of values of $\overline{\rho}$ for which the construction
of Theorem \ref{thm:HALF-MAIN} works has no upper bound. Now, let
$U\subseteq\left[0,\infty\right)$ be the set of $\overline{\rho}$
for which the construction of Theorem \ref{thm:HALF-MAIN} works and
provides an unscreened singularity. We will show that $U$ has an
upper bound. Notice that this will imply the result.

We will prove the claim by contradiction. In this way, assume the
contrary of the claim, i.e., that there exists an increasing sequence
$\left(\overline{\rho}_{n}\right)_{n\in\mathbb{N}}\subseteq U$ such
that $\lim_{n\to\infty}\overline{\rho}_{n}=\infty$. On the one hand,
by Proposition \ref{prop:necessary condition unscreened}, since $\frac{\partial u_{B}}{\partial t}+u_{B}\cdot\nabla u_{B}$
and $\rho_{B}$ do not change with $\overline{\rho}$, we must have
\[
\left|\frac{\partial\psi_{;\overline{\rho}_{n}}}{\partial x_{2}}\left(1,0\right)\right|\approx\overline{\rho}_{n}
\]
for $n$ sufficiently large (which corresponds with $\overline{\rho}_{n}$
sufficiently large). On the other hand, also because of Proposition
\ref{prop:necessary condition unscreened}, we have
\[
f_{\omega_{B}}\left(1,\cdot\right)=\Delta\psi_{;\overline{\rho}_{n}}\left(1,\cdot\right),
\]
which does not scale with $\overline{\rho}_{n}$, so neither can $\frac{\partial\psi_{;\overline{\rho}_{n}}}{\partial x_{2}}\left(1,0\right)$,
a contradiction.
\end{proof}

\begin{rem}
The real problem behind constructing an unscreened singularity with
our method is not the obstacle shown in Proposition \ref{prop:necessary condition unscreened}.
This Proposition is merely an easy necessary condition that we use
to prove that the solution built is a screened singularity, but we
strongly believe that, even if the conclusions of Proposition \ref{prop:necessary condition unscreened}
were simultaneously satisfied, our solution would still be a screened
singularity. Actually, by playing with the scaling of the solutions
of the Boussinesq system, it is possible to set up a similar construction
to the one presented in this article giving a singularity with $g\left(1\right)=0$.
In this case, $\psi\left(1,\cdot\right)=\varPsi\left(1,\cdot\right)$
and both conclusions of Proposition \ref{prop:necessary condition unscreened}
are satisfied. This, indeed, would provide a force $f_{u}$ such that
$\nabla^{\perp}\psi\left(1,\cdot\right)$ is as regular as $f_{u}\left(1,\cdot\right)$
(since they would be equal). Even so, we are strongly convinced that
$\left|\left|\nabla^{\perp}\psi\left(t,\cdot\right)\right|\right|_{\dot{C}^{1}\left(\mathbb{R}^{2}\right)}$
would blow up as $t\to1^{-}$. Indeed, by Remark \ref{rem:first-class singularity},
nothing can be won in the search for an unscreened singularity by
assuming $\phi\neq0$. In other words, we should not only think about
imposing $g\left(1\right)=0$, but about achieving $g\left(t\right)\equiv0$.
Then, the problem reduces to ensuring that, at least,
\[
\frac{\partial\rho_{B}}{\partial x_{2}}\left(1+\frac{1}{\overline{\rho}+\rho_{B}}\left(-\frac{\partial\psi}{\partial x_{2}}-\left(\frac{\partial u_{B}}{\partial t}+u_{B}\cdot\nabla u_{B}\right)_{1}\right)\right)\in L^{\infty}_{t}L^{\infty}_{x}.
\]
Since we only have $\overline{\rho}$ and the scaling parameter of
the Boussinesq solution at our disposal, the best we can do is to
make the factor that multiplies $\frac{\partial\rho_{B}}{\partial x_{2}}$
vanish at the blow-up time and at the blow-up point. As one might
think, using Proposition \ref{prop:dyadic decomposition H=0000F6lder},
this gives good properties for the force at the blow-up time. However,
in order to extend these good properties to times strictly less than
the blow-up time, we would need $\left|\left|\frac{\partial\rho_{B}}{\partial x_{2}}\left(t,\cdot\right)\right|\right|_{L^{\infty}\left(\mathbb{R}^{2}\right)}$
to blow up sufficiently slowly in time. By the blow-up criterion of
the Boussinesq equation (see Main Theorem of \cite{Articulo Boussinesq},
i.e., Theorem \ref{thm:MAIN THM Boussinesq}), $\left|\left|\frac{\partial\rho_{B}}{\partial x_{2}}\left(t,\cdot\right)\right|\right|_{L^{\infty}\left(\mathbb{R}^{2}\right)}$
cannot be integrable in time, i.e., $\left|\left|\frac{\partial\rho_{B}}{\partial x_{2}}\left(t,\cdot\right)\right|\right|_{L^{\infty}\left(\mathbb{R}^{2}\right)}\gtrsim\frac{1}{1-t}$.
Under this condition, we would need $\left|1+\frac{1}{\overline{\rho}+\rho_{B}}\left(-\frac{\partial\psi}{\partial x_{2}}-\left(\frac{\partial u_{B}}{\partial t}+u_{B}\cdot\nabla u_{B}\right)_{1}\right)\right|\lesssim\left(1-t\right)^{1+\varepsilon}$
for some $\varepsilon>0$, or, at least, $\frac{\partial}{\partial t}\left(1+\frac{1}{\overline{\rho}+\rho_{B}}\left(-\frac{\partial\psi}{\partial x_{2}}-\left(\frac{\partial u_{B}}{\partial t}+u_{B}\cdot\nabla u_{B}\right)_{1}\right)\right)\left(1,0\right)=0$,
to have any chance of achieving the wanted regularity and, in our
case, this time derivative does not even exist, so no progress can
be made following this direction. Thus, we believe a different approach
is needed to construct an unscreened singularity.
\end{rem}
\newpage{}

\section*{Acknowledgements}

This work is supported in part by the Spanish Ministry of Science
and Innovation, through the “Severo Ochoa Programme for Centres of
Excellence in R\&D (CEX2023-001347-S)” and PID2023-152878NB-I00. We
were also partially supported by the ERC Advanced Grant 788250, and
by the SNF grant FLUTURA: Fluids, Turbulence, Advection No. 212573.

Furthermore, the authors express their gratitude to J.-Y. Chemin,
R. Danchin and F. Fanelli for useful and insightful remarks during
preliminary versions of this work.

\bibliographystyle{alpha}

\end{document}